\documentclass[reqno,11pt]{amsart}
\usepackage[utf8]{inputenc}
\usepackage{amsmath,amsfonts,amssymb,amsthm,dsfont}
\usepackage{tikz}
\usetikzlibrary{automata,arrows,positioning,calc,shapes.misc,arrows.meta,patterns}
\usepackage{mathtools}
\usepackage{float}
\usepackage{dsfont}
\usepackage{mathrsfs}
\RequirePackage[colorlinks=true, pdfstartview=FitV, linkcolor=blue,
  citecolor=blue, urlcolor=blue]{hyperref}
\usepackage[
top    = 3.5cm,
bottom = 3.25cm,
left   = 2.9cm,
right  = 2.9cm]{geometry}

\newcommand{\1}{\mathds{1}}

\DeclarePairedDelimiterX{\norm}[1]{\lVert}{\rVert}{#1}
\DeclarePairedDelimiterX{\bignorm}[1]{\big\lVert}{\big\rVert}{#1}

\newcommand{\RN}[1]{%
  \textup{\uppercase\expandafter{\romannumeral#1}}%
}

\newcommand{\cG}{\mathcal{G}}

\numberwithin{equation}{section}

\theoremstyle{plain}
\newtheorem{definition}{Definition}[section]
\newtheorem{remark}[definition]{Remark}

\theoremstyle{plain}
\newtheorem{theorem}{Theorem}[section]
\newtheorem{lemma}[theorem]{Lemma}
\newtheorem{prop}[theorem]{Proposition}

\newtheorem{claim}[theorem]{Claim}

\title{Mixing Times for the Facilitated Exclusion Process}
\author[J. Ayre]{James Ayre}
\address{J. Ayre, Nuclear Department, HMS Sultan, Gosport, PO12 3BY, United Kingdom}
\email{\href{mailto:james.ayre104@mod.gov.uk}{james.ayre104@mod.gov.uk}}
\thanks{James Ayre was supported by EPSRC grant EP/S515541/1.}

\author[P. Chleboun]{Paul Chleboun}
\address{P. Chleboun, Department of Statistics, University of Warwick,  Coventry,  CV4 7AL,  United Kingdom }
\email{\href{mailto:paul.i.chleboun@warwick.ac.uk}{paul.i.chleboun@warwick.ac.uk}}
\date{\today}


\begin{document}

\begin{abstract}
The facilitated simple exclusion process (FEP) is a one--dimensional exclusion process with a dynamical constraint. 
We establish bounds on the mixing time of the FEP on the segment, with closed boundaries, and the circle. 
The FEP on these spaces exhibits transient states that, if the macroscopic density of particles is at least $1/2$, the process will eventually
exit to reach an ergodic component. 
If the
macroscopic density is less than $1/2$ the process will hit an absorbing state.
We show that the symmetric FEP (SFEP) on the segment $\{1,\ldots,N\}$, with $k>N/2$ particles, has mixing time of order $N^{2}\log(N-k)$ and exhibits the pre--cutoff phenomenon. 
For the asymmetric FEP (AFEP) on the segment, we show that there
exists initial conditions for which the hitting time of the ergodic component is exponentially slow
in the number of holes $N-k$. 
In particular, when $N-k$ is large enough, the hitting time of the
ergodic component determines the mixing time.
For the SFEP on the circle of size $N$, and macroscopic particle density $\rho \in(1/2,1)$,
we establish bounds on the mixing time of order $N^{2}\log N$ for the process restricted to its ergodic component.
We also give an upper bound on the hitting time
of the ergodic component of order $N^{2}\log N$ for a large class of initial conditions.
The proofs rely on couplings with exclusion processes (both open and closed boundaries) via a novel lattice path (height function) construction of the FEP.
\end{abstract}

\maketitle

\newcommand*{\xMin}{0}%
\newcommand*{\xMax}{15}%
\newcommand*{\yMin}{0}%
\newcommand*{\yMax}{15}%
\newcommand*{\upmax}{2}%
\newcommand*{\yMaxx}{8}%
\newcommand*{\xMaxx}{8}%

\newcommand{\ergodic}{\begin{tikzpicture}
\tikzset{
	b1/.style = {circle, draw=black,fill=black,minimum size = 3mm, inner sep = 1pt},
	b2/.style = {circle, draw=black,fill=white,minimum size = 3mm, inner sep = 1pt}}
	\node[b1] at (1,0) {};
	\node[b2] at (2,0) {};
	\node[b1] at (3,0) {};
	\node[b1] at (4,0) {};
	\node (A) [b1] at (5,0) {};
	\node (B) [b2] at (6,0) {};
	\node[b1] at (7,0) {};
	\node[b2] at (8,0) {};
	\node[b1] at (9,0) {};
	\node[b2] at (10,0) {};
	\node[b1] at (11,0) {};
	\node (D) [b2] at (12,0) {};
	\node (C) [b1] at (13,0) {};
	\node[b1] at (14,0) {};
	\node[b1] at (15,0) {};
	\node[b2] at (16,0) {};
	\node[b1] at (17,0) {};
	\node[b1] at (18,0) {};
	\node[b2] at (19,0) {};
	\node[b1] at (20,0) {};
	\node[b1] at (21,0) {};
	\node[b1] at (22,0) {};
	\draw[->]
	(A) edge[bend left=40] (B)
	(C) edge[bend right=40] (D);
	\node at (5.5,0.5) {$p$};
	\node at (12.5,0.5) {$q$};
	
    \foreach \i in {1,...,22}
        {\draw (\i-0.25,-0.25) -- (\i+0.25,-0.25);
        }
	\end{tikzpicture}}

\newcommand{\latticepath}{\begin{tikzpicture}
\tikzset{
	b1/.style = {circle, draw=black,fill=black,minimum size = 1mm, inner sep = 1pt}}
	\node[b1] at (1,0) {};
	\node[b1] at (2,1) {};
	\node[b1] at (3,0) {};
	\node[b1] at (4,-1) {};
	\node[b1] at (5,0) {};
	\node[b1] at (6,1) {};
	\node[b1] at (7,2) {};
	\node[b1] at (8,3) {};
	\node[b1] at (9,2) {};
	\node[b1] at (10,1) {};
	\node[b1] at (11,2) {};
	\node[b1] at (12,1) {};
	\node[b1] at (13,2) {};
	\node[b1] at (14,1) {};
	\node[b1] at (15,0) {};
    \draw[thick] (1,0) -- (2,1);
    \draw[thick] (2,1) -- (3,0);
    \draw[thick] (3,0) -- (4,-1);
    \draw[thick] (4,-1) -- (5,0);
    \draw[thick] (5,0) -- (6,1);
    \draw[thick] (6,1) -- (7,2);
    \draw[thick] (7,2) -- (8,3);
    \draw[thick] (8,3) -- (9,2);
    \draw[thick] (9,2) -- (10,1);
    \draw[thick] (10,1) -- (11,2);
    \draw[thick] (11,2) -- (12,1);
    \draw[thick] (12,1) -- (13,2);
    \draw[thick] (13,2) -- (14,1);
    \draw[thick] (14,1) -- (15,0);
    \draw[->] (0.5,-1.5) -- (15.5,-1.5);
    \draw[->] (0.5,-1.5) -- (0.5,3.5);
    \draw[->] (8,3) -- (8,2);
    \node at (8,1.75) {$q$};
    \node[b1] at (8,1) {};
    \draw[dashed] (7,2) -- (8,1);
    \draw[dashed] (8,1) -- (9,2);
    \draw[->] (4,-1) -- (4,0);
    \node at (4,0.25) {$p$};
    \node[b1] at (4,1) {};
    \draw[dashed] (3,0) -- (4,1);
    \draw[dashed] (4,1) -- (5,0);
    \node[left] at (0.25,-1) {$-0.5$};
    \node[left] at (0.25,0) {$0$};
    \node[left] at (0.25,1) {$0.5$};
    \node[left] at (0.25,2) {$1$};
    \node[left] at (0.25,3) {$1.5$};
    \node[below] at (1,-1.75) {$1$};
    \node[below] at (2,-1.75) {$2$};
    \node[below] at (3,-1.75) {$3$};
    \node[below] at (4,-1.75) {$4$};
    \node[below] at (5,-1.75) {$5$};
    \node[below] at (6,-1.75) {$6$};
    \node[below] at (7,-1.75) {$7$};
    \node[below] at (8,-1.75) {$8$};
    \node[below] at (9,-1.75) {$9$};
    \node[below] at (10,-1.75) {$10$};
    \node[below] at (11,-1.75) {$11$};
    \node[below] at (12,-1.75) {$12$};
    \node[below] at (13,-1.75) {$13$};
    \node[below] at (14,-1.75) {$14$};
    \node[below] at (15,-1.75) {$15$};
    \node at (15.75,-1.5) {$i$};
    \node at (0.5,3.75) {$\eta^{\xi}(i)$};
	\end{tikzpicture}}

\newcommand{\gconfig}{\begin{tikzpicture}
\tikzset{
	b1/.style = {circle, draw=black,fill=white,minimum size = 0.5cm, inner sep = 1pt},
	b2/.style = {circle, draw=black,fill=white,minimum size = 3mm, inner sep = 1pt}}
    \node at (1,0) {};
    \node at (2,0) {};
    \node (Y) at (3,0) {};
    \node (Z) [b1] at (4,0) {$1$};
    \node[b1] at (5,0) {$2$};
    \node[b1] at (6,0) {$3$};
    \node[b1] at (7,0) {$4$};
    \node[b1] at (8,0) {$5$};
    \node[b1] at (9,0) {$6$};
    \node (A) [b1] at (10,0) {$7$};
    \node (B) at (11,0) {};
    \node (C) at (12,0) {};
    \node (D) [b1] at (13,0) {$8$};
    \node [b1] at (14,0) {$9$};
    \node (E) [b1] at (15,0) {$10$};
    \node (F) at (16,0) {};
    \node (G) [b1] at (17,0) {$11$};
    \node[b1] at (18,0) {$12$};
    \node[b1] at (19,0) {$13$};
    \node[b1] at (20,0) {$14$};
    \node (X) [b1] at (21,0) {$15$};
    \node (W) at (22,0) {};
    \draw[red,->]
	(A) edge[bend left=40] (B)
	(E) edge[bend left=40] (F)
	(X) edge[bend left=40] (W);
	\draw[blue,->]
	(D) edge[bend right=40] (C)
	(G) edge[bend right=40] (F)
	(Z) edge[bend right=40] (Y);
	\node at (10.5,0.5) {$p$};
	\node at (12.5,0.5) {$q$};
	\node at (15.5,0.5) {$p$};
	\node at (16.5,0.5) {$q$};
	\node at (21.5,0.5) {$p$};
	\node at (3.5,0.5) {$q$};
    \foreach \i in {1,...,22}
        {\draw (\i-0.25,-0.3) -- (\i+0.25,-0.3);
        \node[below] at (\i,-0.4) {$\i$};
        }
	\end{tikzpicture}}

\newcommand{\glatticepath}{\begin{tikzpicture}
\tikzset{
	b1/.style = {circle, draw=black,fill=black,minimum size = 1mm, inner sep = 1pt},
	b2/.style = {circle, draw=red,fill=red,minimum size = 1mm, inner sep = 1pt},
	b3/.style = {circle, draw=blue,fill=blue,minimum size = 1mm, inner sep = 1pt}}
    \node[b1] at (1,6) {};
    \node[b1] at (2,5) {};
    \node[b1] at (3,4) {};
    \node[b1] at (4,3) {};
    \node[b1] at (5,2) {};
    \node[b1] at (6,1) {};
    \node[b1] at (7,0) {};
    \node[b1] at (8,3) {};
    \node[b1] at (9,2) {};
    \node[b1] at (10,1) {};
    \node[b1] at (11,2) {};
    \node[b1] at (12,1) {};
    \node[b1] at (13,0) {};
    \node[b1] at (14,-1) {};
    \node[b1] at (15,-2) {};
    \draw[->] (0.5,-2.5) -- (15.5,-2.5);
    \draw[->] (0.5,-2.5) -- (0.5,6.5);
    \node at (15.75,-2.5) {$i$};
    \node at (0.5,6.75) {$\eta^{\xi}(i)$};
    \draw[thick] (1,6) -- (2,5);
    \draw[thick] (2,5) -- (3,4);
    \draw[thick] (3,4) -- (4,3);
    \draw[thick] (4,3) -- (5,2);
    \draw[thick] (5,2) -- (6,1);
    \draw[thick] (6,1) -- (7,0);
    \draw[thick] (7,0) -- (8,3);
    \draw[thick] (8,3) -- (9,2);
    \draw[thick] (9,2) -- (10,1);
    \draw[thick] (10,1) -- (11,2);
    \draw[thick] (11,2) -- (12,1);
    \draw[thick] (12,1) -- (13,0);
    \draw[thick] (13,0) -- (14,-1);
    \draw[thick] (14,-1) -- (15,-2);
    \node[below] at (1,-2.75) {$1$};
    \node[below] at (2,-2.75) {$2$};
    \node[below] at (3,-2.75) {$3$};
    \node[below] at (4,-2.75) {$4$};
    \node[below] at (5,-2.75) {$5$};
    \node[below] at (6,-2.75) {$6$};
    \node[below] at (7,-2.75) {$7$};
    \node[below] at (8,-2.75) {$8$};
    \node[below] at (9,-2.75) {$9$};
    \node[below] at (10,-2.75) {$10$};
    \node[below] at (11,-2.75) {$11$};
    \node[below] at (12,-2.75) {$12$};
    \node[below] at (13,-2.75) {$13$};
    \node[below] at (14,-2.75) {$14$};
    \node[below] at (15,-2.75) {$15$};
    \node[left] at (0.25,-2) {$-1$};
    \node[left] at (0.25,-1) {$-0.5$};
    \node[left] at (0.25,0) {$0$};
    \node[left] at (0.25,1) {$0.5$};
    \node[left] at (0.25,2) {$1$};
    \node[left] at (0.25,3) {$1.5$};
    \node[left] at (0.25,4) {$2$};
    \node[left] at (0.25,5) {$2.5$};
    \node[left] at (0.25,6) {$3$};
    \draw[red,->] (7,0) -- (7,1);
    \draw[blue,->] (8,3) -- (8,2);
    \node[b2] at (7,2) {};
    \draw[red,dashed] (6,1) -- (7,2);
    \draw[red,dashed] (7,2) -- (8,3);
    \node[b3] at (8,1) {};
    \draw[blue,dashed] (9,2) -- (8,1);
    \draw[blue,dashed] (8,1) -- (7,0);
    \draw[red,->] (10,1) -- (10,2);
    \node[b2] at (10,3) {};
    \draw[red,dashed] (9,2) -- (10,3);
    \draw[red,dashed] (10,3) -- (11,2);
    \draw[blue,->] (11,2) -- (11,1);
    \node[b3] at (11,0) {};
    \draw[blue,dashed] (10,1) -- (11,0);
    \draw[blue,dashed] (11,0) -- (12,1);
    \draw[blue,->] (1,6) -- (1,5);
    \node[b3] at (1,4) {};
    \draw[blue,dashed] (2,5) -- (1,4);
    \draw[red,->] (15,-2) -- (15,-1);
    \node[b2] at (15,0) {};
    \draw[red,dashed] (14,-1) -- (15,0);
    \node at (7,1.25) {$p$};
    \node at (8,1.75) {$q$};
    \node at (10,2.25) {$p$};
    \node at (11,0.75) {$q$};
    \node at (1,4.75) {$q$};
    \node at (15,-0.75) {$p$};
	\end{tikzpicture}}

\newcommand{\regions}{
\begin{tikzpicture}[a/.style={circle, draw=black, fill=white, minimum size=3mm, inner sep=1pt},
b/.style={circle, draw=black, fill=black, minimum size=3mm, inner sep=1pt}]
  \draw[very thick] (0,0) circle(2cm);
  \node[a] at (0*18:2){};
  \node[b] at (1*18:2){};
  \node at (1*18:2.5){$I_{2}$};
  \node[b] at (2*18:2){};
  \node[b] at (3*18:2){};
  \node[a] at (4*18:2){};
  \node[a] at (5*18:2){};
  \node[below] at (5*18:1.9) {$0$};
  \node[b] at (6*18:2){};
  \node[a] at (7*18:2){};
  \node[b] at (8*18:2){};
  \node at (8.5*18:2.5) {$I_{1}$};
  \node[a] at (9*18:2){};
  \node[b] at (10*18:2){};
  \node[b] at (11*18:2){};
  \node[a] at (12*18:2){};
  \node[a] at (13*18:2){};
  \node[b] at (14*18:2){};
  \node[b] at (15*18:2){};
  \node at (15*18:2.5){$I_{3}$};
  \node[b] at (16*18:2){};
  \node[a] at (17*18:2){};
  \node[a] at (18*18:2){};
  \node[b] at (19*18:2){};
  \draw (-1.5*18:2.25) arc (-1.5*18:3.5*18:2.25);
  \draw (-1.5*18:1.75) arc (-1.5*18:3.5*18:1.75);
  \draw (-1.5*18:2.25) -- (-1.5*18:1.75);
  \draw (3.5*18:2.25) -- (3.5*18:1.75);
  \draw (5.5*18:2.25) arc (5.5*18:11.5*18:2.25);
  \draw (5.5*18:1.75) arc (5.5*18:11.5*18:1.75);
  \draw (5.5*18:2.25) -- (5.5*18:1.75);
  \draw (11.5*18:2.25) -- (11.5*18:1.75);
  \draw (13.5*18:2.25) arc (13.5*18:16.5*18:2.25);
  \draw (13.5*18:1.75) arc (13.5*18:16.5*18:1.75);
  \draw (13.5*18:2.25) -- (13.5*18:1.75);
  \draw (16.5*18:2.25) -- (16.5*18:1.75);
\end{tikzpicture}}

\newcommand{\zeromap}{
\begin{tikzpicture}[a/.style={circle, draw=black, fill=white, minimum size=3mm, inner sep=1pt},
b/.style={circle, draw=black, fill=black, minimum size=3mm, inner sep=1pt}]
  \draw[very thick] (0,0) circle(2cm);
  \node[b] at (90+0*24:2){};
  \node[below] at (90+0*24:1.8) {$0$};
  \node[b] at (90-1*24:2){};
  \node[a] at (90-2*24:2){};
  \node[b] at (90-3*24:2){};
  \node[b] at (90-4*24:2){};
  \node[a] at (90-5*24:2){};
  \node[b] at (90-6*24:2){};
  \node[b] at (90-7*24:2){};
  \node[b] at (90-8*24:2){};
  \node[b] at (90-9*24:2){};
  \node[a] at (90-10*24:2){};
  \node[b] at (90-11*24:2){};
  \node[a] at (90-12*24:2){};
  \node[a] at (90-13*24:2){};
  \node[a] at (90-14*24:2){};
  \draw[very thick] (7,0) circle(1.5cm);
  \node[below] at ($(7,0)+(90+0*60:1.3)$) {$e_{0}$};
  \node[b] at ($(7,0)+(90+0*60:1.5)$) {};
  \node[b] at ($(7,0)+(90+0*60:2)$) {};
  \node[b] at ($(7,0)+(90-1*60:1.5)$) {};
  \node[b] at ($(7,0)+(90-1*60:2)$) {};
  \node[b] at ($(7,0)+(90-2*60:1.5)$) {};
  \node[b] at ($(7,0)+(90-2*60:2)$) {};
  \node[b] at ($(7,0)+(90-2*60:2.5)$) {};
  \node[b] at ($(7,0)+(90-2*60:3)$) {};
  \node[b] at ($(7,0)+(90-3*60:1.5)$) {};
  \node[a] at ($(7,0)+(90-4*60:1.5)$) {};
  \node[a] at ($(7,0)+(90-5*60:1.5)$) {};
  \draw[->] (2.5,0)--(5,0);
  \node[above] at (3.75,0.5) {$\Pi_{N,k}$};
\end{tikzpicture}}

\newcommand{\zeropen}{\begin{tikzpicture}
\tikzset{
	b1/.style = {circle, draw=black,fill=black,minimum size = 3mm, inner sep = 1pt},
	b2/.style = {circle, draw=black,fill=white,minimum size = 3mm, inner sep = 1pt},
	b3/.style = {rectangle,draw=black,fill=white,minimum size = 3mm, inner sep = 4pt}}
	\node[b3] at (1,0){Reservoir};
	\node[b2] (E1) at (3,0){$1$};
    \node (F1) at (4,0) {};
	\node[b2] (A1) at (6,0){$2$};
	\node (B1) at (5,0){};
	\node[b2] (G1) at (7,0){$3$};
	\node[b3] (C1) at (10,0){Reservoir};
	\node (D1) at (8,0){};
    \node[b1] at (4,2.5){};
    \node (E) at (4,3.5){};
    \node[b1] (A) at (5,3.5){};
    \node[b1] at (5,3.0){};
    \node[b1] at (5,2.5){};
    \node (B) at (6,3.5){};
    \node[b1] at (6,2.5){};
    \node (F) at (6,3){};
    \node[b1] (C) at (7,3.0){};
    \node[b1] at (7,2.5){};
    \node (D) at (8,3){};
    \node at (2,2.5){$\cdots$};
    \node at (9,2.5){$\cdots$};
    \node at (3.5,1.7){$\omega$};
    \node at (3.5,-0.8){$\Phi(\omega)$};
    \draw[->,thick] (5.5,2)--(5.5,0.5);
    \draw[->,red] 
    (A) edge[bend left=40] (B)
    (A1) edge[bend right=40] (B1);
    \draw[->,violet]
    (C) edge[bend left=40] (D)
    (C1) edge[bend right=40] (D1);
    \draw[->] 
    (E1) edge[bend left=40] (F1)
    (A) edge[bend right=40] (E);
    \draw[->,blue]
    (C) edge[bend right=40] (F)
    (G1) edge[bend left=40] (D1);
    \node at (5.8,1.25) {$\Phi$};
    \foreach \i in {1,...,6}
        {\draw (\i+1.75,-0.2) -- (\i+2.25,-0.2);
        }
    \foreach \i in {1,...,6}
        {\draw (\i+1.75,2.2) -- (\i+2.25,2.2);
        }
\end{tikzpicture}
}

\newcommand{\aseplattice}{\begin{tikzpicture}
\tikzset{
	b1/.style = {circle, draw=black,fill=black,minimum size = 3mm, inner sep = 1pt},
	b2/.style = {circle, draw=black,fill=white,minimum size = 3mm, inner sep = 1pt},
	b3/.style = {circle, draw=black,fill=black,minimum size = 1mm, inner sep = 1pt}}
    \node[b1] at (1,-3) {};
    \node[b1] at (2,-3) {};
    \node[b1] at (3,-3) {};
    \node[b2] at (4,-3) {};
    \node[b2] at (5,-3) {};
    \node[b1] at (6,-3) {};
    \node (P1) [b1] at (7,-3) {};
    \node (H1) [b2] at (8,-3) {};
    \node (H2) [b2] at (9,-3) {};
    \node (P2) [b1] at (10,-3) {};
    \node[b2] at (11,-3) {};
    \node[b1] at (12,-3) {};
    \node[b2] at (13,-3) {};
    \node[b2] at (14,-3) {};
    
    \node (A) [b3] at (0,0){};
    \node (B) [b3] at (1,1){};
    \node (C) [b3] at (2,2){};
    \node (D) [b3] at (3,3){};
    \node (E) [b3] at (4,2){};
    \node (F) [b3] at (5,1){};
    \node (G) [b3] at (6,2){};
    \node (H) [b3] at (7,3){};
    \node (I) [b3] at (8,2){};
    \node (J) [b3] at (9,1){};
    \node (K) [b3] at (10,2){};
    \node (L) [b3] at (11,1){};
    \node (M) [b3] at (12,2){};
    \node (N) [b3] at (13,1){};
    \node (O) [b3] at (14,0){};
    \draw
    (A) edge (B)
    (B) edge (C)
    (C) edge (D)
    (D) edge (E)
    (E) edge (F)
    (F) edge (G)
    (G) edge (H)
    (H) edge (I)
    (I) edge (J)
    (J) edge (K)
    (K) edge (L)
    (L) edge (M)
    (M) edge (N)
    (N) edge (O);
    \draw[->] (7,3) -- (7,2);
    \node at (7,1.75) {$p$};
    \draw[dotted] (6,2) -- (7,1);
    \draw[dotted] (7,1) -- (8,2);
    \draw[->] (9,1) -- (9,2);
    \node at (9,2.25) {$q$};
    \draw[dotted] (8,2) -- (9,3);
    \draw[dotted] (9,3) -- (10,2);
    \draw[->] (-0.5,-0.5) -- (-0.5,3.5);
    \node[left] at (-0.75,0) {$0$};
    \node[left] at (-0.75,1) {$0.5$};
    \node[left] at (-0.75,2) {$1$};
    \node[left] at (-0.75,3) {$1.5$};
    \draw[->] (-0.5,-0.5) -- (14.5,-0.5);
    \node[below] at (0,-0.75) {$0$};
    \node[below] at (1,-0.75) {$1$};
    \node[below] at (2,-0.75) {$2$};
    \node[below] at (3,-0.75) {$3$};
    \node[below] at (4,-0.75) {$4$};
    \node[below] at (5,-0.75) {$5$};
    \node[below] at (6,-0.75) {$6$};
    \node[below] at (7,-0.75) {$7$};
    \node[below] at (8,-0.75) {$8$};
    \node[below] at (9,-0.75) {$9$};
    \node[below] at (10,-0.75) {$10$};
    \node[below] at (11,-0.75) {$11$};
    \node[below] at (12,-0.75) {$12$};
    \node[below] at (13,-0.75) {$13$};
    \node[below] at (14,-0.75) {$14$};
    \draw[->]
    (P1) edge[bend left=40] (H1)
    (P2) edge[bend right=40] (H2);
    \node at (7.5,-2.5) {$p$};
    \node at (9.5,-2.5) {$q$};
    \foreach \i in {1,...,14}
        {\draw (\i-0.25,-3.25) -- (\i+0.25,-3.25);
        }
	\end{tikzpicture}}

\newcommand{\quasi}{
\begin{tikzpicture}[a/.style={circle, draw=black, fill=white, minimum size=3mm, inner sep=1pt},
b/.style={circle, draw=black, fill=black, minimum size=3mm, inner sep=1pt}]
  \draw[very thick] (0,0) circle(2cm);
  \node[a] at (0*18:2){};
  \node[b] at (1*18:2){};
  \node[b] at (2*18:2){};
  \node[b] at (3*18:2){};
  \node[b] at (4*18:2){};
  \node[a] at (5*18:2){};
  \node[below] at (5*18:1.6) {$0$};
  \node[above] at (-5*18:1.6) {$10$};
  \node[above] at (4.5*18:2.3) {$\mathcal{G}_{0}^{(3)}$};
  \node[below] at (-5.5*18:2.3) {$\mathcal{G}_{10}^{(1)}$};
  \node[b] at (6*18:2){};
  \node[a] at (7*18:2){};
  \node[b] at (8*18:2){};
  \node[a] at (9*18:2){};
  \node[b] at (10*18:2){};
  \node[b] at (11*18:2){};
  \node[a] at (12*18:2){};
  \node[b] at (13*18:2){};
  \node[b] at (14*18:2){};
  \node[b] at (15*18:2){};
  \node[b] at (16*18:2){};
  \node[a] at (17*18:2){};
  \node[b] at (18*18:2){};
  \node[b] at (19*18:2){};
  \draw (3.5*18:2.25) arc (3.5*18:5.5*18:2.25);
  \draw (3.5*18:1.75) arc (3.5*18:5.5*18:1.75);
  \draw (3.5*18:2.25) -- (3.5*18:1.75);
  \draw (5.5*18:2.25) -- (5.5*18:1.75);
  \draw (13.5*18:2.25) arc (13.5*18:15.5*18:2.25);
  \draw (13.5*18:1.75) arc (13.5*18:15.5*18:1.75);
  \draw (13.5*18:2.25) -- (13.5*18:1.75);
  \draw (15.5*18:2.25) -- (15.5*18:1.75);
\end{tikzpicture}}

\newcommand{\omap}{\begin{tikzpicture}
\tikzset{
	b1/.style = {circle, draw=black,fill=black,minimum size = 3mm, inner sep = 1pt},
	b2/.style = {circle, draw=black,fill=white,minimum size = 3mm, inner sep = 1pt},
	b3/.style = {rectangle,draw=black,fill=white,minimum size = 3mm, inner sep = 4pt}}
	\node[b3] at (1,0){Reservoir};
	\node[b1] at (3,0){};
	\node[b1] (A1) at (6,0){};
	\node (B1) at (5,0){};
	\node[b1] at (7,0){};
	\node[b3] (C1) at (10,0){Reservoir};
	\node (D1) at (8,0){};
    \node[b2] at (1,2.5){$1$};
    \node[b2] at (3,2.5){$2$};
    \node[b2] at (4,2.5){$3$};
    \node[b2] (A) at (5,2.5){$4$};
    \node (B) at (6,2.5){};
    \node[b2] at (7,2.5){$5$};
    \node[b2] at (9,2.5){$6$};
    \node[b2] (C) at (10,2.5){$7$};
    \node (D) at (11,2.5){};
    \node at (-1,2.5){$\cdots$};
    \node at (12,2.5){$\cdots$};
    \node at (3.5,1.7){$\xi^{I}$};
    \node at (3.5,0.5){$\phi\bigl(\xi^{I}\bigr)$};
    \draw[->,thick] (5.5,2)--(5.5,0.5);
    \draw[->,red]
    (A) edge[bend left=40] (B)
    (A1) edge[bend right=40] (B1);
    \draw[->,blue]
    (C) edge[bend left=40] (D)
    (C1) edge[bend right=40] (D1);
    \node at (5.8,1.25) {$\phi$};
    \foreach \i in {0,...,11}
        {\draw (\i-0.25,2.2) -- (\i+0.25,2.2);
        }
    \foreach \i in {1,...,6}
        {\draw (\i+1.75,-0.2) -- (\i+2.25,-0.2);
        }
\end{tikzpicture}
}

\newcommand{\zeroseg}{\begin{tikzpicture}
\tikzset{
	b1/.style = {circle, draw=black,fill=black,minimum size = 3mm, inner sep = 1pt},
	b2/.style = {circle, draw=black,fill=white,minimum size = 3mm, inner sep = 1pt}}
	\node[b1] at (0,0){};
	\node[b2] at (1,0){};
	\node[b1] at (2,0){};
	\node[b1] at (3,0){};
	\node[b1] at (4,0){};
	\node[b2] at (5,0){};
	\node[b1] at (6,0){};
	\node[b1] at (7,0){};
	\node[b1] at (10,0){};
	\node[b1] at (11,0){};
	\node[b1] at (11,0.5){};
	\node[b1] at (11,1){};
	\node[b1] at (12,0){};
	\node[b1] at (12,0.5){};
	\node at (8.5,0.7){$\Pi$};
	\draw[->,thick] (7.5,0.5) -- (9.5,0.5);
	\foreach \i in {0,...,7}
	    {\draw (\i-0.25,-0.3) -- (\i+0.25,-0.3);}
	\foreach \i in {10,...,12}
	    {\draw (\i-0.25,-0.3) -- (\i+0.25,-0.3);}
	 \draw[red] (9.75,-0.2) rectangle ++ (2.5,0.4);
\end{tikzpicture}
}

\newcommand{\zeroerg}{\begin{tikzpicture}
\tikzset{
	b1/.style = {circle, draw=black,fill=black,minimum size = 3mm, inner sep = 1pt},
	b2/.style = {circle, draw=black,fill=white,minimum size = 3mm, inner sep = 1pt}}
	\node[b2] at (-1,0){};
	\node[b1] at (0,0){};
	\node[b2] at (1,0){};
	\node[b1] at (2,0){};
	\node[b1] at (3,0){};
	\node[b1] at (4,0){};
	\node[b2] at (5,0){};
	\node[b1] at (6,0){};
	\node[b1] at (7,0){};
	\node[b1] at (12,0){};
	\node[b1] at (12,0.5){};
	\node[b1] at (13,0){};
	\node at (8.5,0.5){$\widetilde{\Pi}$};
	\draw[->,thick] (7.5,0.25) -- (9.5,0.25);
	\foreach \i in {-1,0,...,7}
	    {\draw (\i-0.25,-0.3) -- (\i+0.25,-0.3);}
	\foreach \i in {10,...,13}
	    {\draw (\i-0.25,-0.3) -- (\i+0.25,-0.3);}
\end{tikzpicture}
}

\newcommand{\keyidea}{\begin{tikzpicture}
\tikzset{
	b1/.style = {circle, draw=black,fill=black,minimum size = 3mm, inner sep = 1pt},
	b2/.style = {circle, draw=black,fill=white,minimum size = 3mm, inner sep = 1pt}}
	\node at (0,0) {$\cdots$};
	\node at (11,0) {$\cdots$};
	\node at (0,2) {$\cdots$};
	\node at (11,2) {$\cdots$};
	\node[b1] at (1,2){};
	\node[b1] at (2,2){};
	\node[b2] (B) at (3,2){};
	\node[b1] (A) at (4,2){};
	\node[b1] at (5,2){};
	\node[b2] at (6,2){};
	\node[b1] at (7,2){};
	\node[b2] at (8,2){};
	\node[b1] at (9,2){};
	\node[b1] at (10,2){};
	\node[b1] at (1,0){};
	\node[b1] at (2,0){};
	\node[b1] at (3,0){};
	\node[b2] at (4,0){};
	\node[b1] at (5,0){};
	\node[b2] at (6,0){};
	\node[b1] at (7,0){};
	\node[b2] at (8,0){};
	\node[b1] at (9,0){};
	\node[b1] at (10,0){};
	\node (C) at (2.5,2.2) {};
	\node (D) at (4,2.2) {};
	\draw[->,thick] (5.5,1.5) -- (5.5,0.7);
	\draw[->] (A) edge[bend right=40] (B);
	\draw[->,red] (C) edge[bend left=40] (D);
	\draw[red] (1.75,1.8) rectangle ++ (1.5,0.4)
	(4.75,1.8) rectangle ++ (1.5,0.4)
	(6.75,1.8) rectangle ++ (1.5,0.4)
	(2.75,-0.2) rectangle ++ (1.5,0.4)
	(4.75,-0.2) rectangle ++ (1.5,0.4)
	(6.75,-0.2) rectangle ++ (1.5,0.4);
	\foreach \i in {1,...,10}
	    {\draw (\i-0.25,-0.3) -- (\i+0.25,-0.3);
	    \draw (\i-0.25,1.7) -- (\i+0.25,1.7);}
\end{tikzpicture}
}

\newcommand{\polylattice}{\begin{tikzpicture}
\tikzset{
	b1/.style = {circle, draw=black,fill=black,minimum size = 1mm, inner sep = 1pt}}
	\draw[->] (-0.5,-0.5) -- (8.5,-0.5);
	\draw[->] (-0.5,-0.5) -- (-0.5,2.5);
	\draw[->] (2,2) -- (2,1); 
	\draw[->] (6,0) -- (6,1);
	\node at (2,0.75) {$p$};
	\node at (6,1.25) {$q$};
	\node at (8.75,-0.5) {$z$};
	\node at (-0.5,2.75) {$\eta^{\boldsymbol{x}}(i)$};
	\node[left] at (-0.75,0) {$0$};
	\node[left] at (-0.75,1) {$0.5$};
	\node[left] at (-0.75,2) {$1$};
	\node[below] at (0,-0.75) {$0$};
	\node[below] at (1,-0.75) {$1$};
	\node[below] at (2,-0.75) {$2$};
	\node[below] at (3,-0.75) {$3$};
	\node[below] at (4,-0.75) {$4$};
	\node[below] at (5,-0.75) {$5$};
	\node[below] at (6,-0.75) {$6$};
	\node[below] at (7,-0.75) {$7$};
	\node[below] at (8,-0.75) {$8$};
	\node[b1] at (0,0) {};
	\node[b1] at (1,1) {};
	\node[b1] at (2,2) {};
	\node[b1] at (3,1) {};
	\node[b1] at (4,2) {};
	\node[b1] at (5,1) {};
	\node[b1] at (6,0) {};
	\node[b1] at (7,1) {};
	\node[b1] at (8,0) {};
	\node[b1] at (2,0) {};
	\node [b1] at (6,2) {};
	\draw(0,0) -- (1,1)
	(1,1) -- (2,2)
	(2,2) -- (3,1)
	(3,1) -- (4,2)
	(4,2) -- (5,1)
	(5,1) -- (6,0)
	(6,0) -- (7,1)
	(7,1) -- (8,0);
	\draw[dashed] (1,1) -- (2,0)
	(2,0) -- (3,1)
	(5,1) -- (6,2) 
	(6,2) -- (7,1);
\end{tikzpicture}}

\newcommand{\polyconfig}{\begin{tikzpicture}
    \foreach \i in {1,...,14}
        {\draw (\i-0.25,-0.25) -- (\i+0.25,-0.25);
         \node[below] at (\i,-0.3) {$\i$};
        }
    \draw (0.75,0) rectangle ++ (0.5,0.5);
    \draw (1.75,0) rectangle ++ (3.5,0.5);
    \draw (6.75,0) rectangle ++ (1.5,0.5);
    \draw (10.75,0) rectangle ++ (2.5,0.5);
    \node (A) at (5.25,0.25) {};
    \node (B) at (6,0.25) {}; 
    \node (C) at (10.75,0.25){};
    \node (D) at (10,0.25){};
    \node at (5.65,0.4) {$p$};
    \node at (10.35,0.4) {$q$};
    \draw[->]
    (A) edge (B)
    (C) edge (D);
	\end{tikzpicture}
}

\newcommand{\ssephf}{\begin{tikzpicture}
\tikzset{
	b1/.style = {circle, draw=black,fill=black,minimum size = 1mm, inner sep = 1pt},
    b2/.style = {circle, draw=red,fill=red,minimum size = 1mm, inner sep = 1pt},
    b3/.style = {circle, draw=blue,fill=blue,minimum size = 1mm, inner sep = 1pt}}
\draw[->] (-0.5,-5.5) -- (-0.5,5.5);
\draw[->] (-0.5,0) -- (10.5,0);
\draw[dashed] (0,0) -- (5,-5);
\draw[dashed] (5,-5) -- (10,0);
\draw[dashed] (0,0) -- (5,5);
\draw[dashed] (5,5) -- (10,0);
\node[b1] at (0,0) {};
\node[b1] at (1,-1) {};
\node[b1] at (2,0) {};
\node[b1] at (3,-1) {};
\node[b1] at (4,-2) {};
\node[b1] at (5,-1) {};
\node[b1] at (6,0) {};
\node[b1] at (7,1) {};
\node[b1] at (8,2) {};
\node[b1] at (9,1) {};
\node[b1] at (10,0) {};
\draw[thick] (0,0) -- (1,-1);
\draw[thick] (1,-1) -- (2,0);
\draw[thick] (2,0) -- (3,-1);
\draw[thick] (3,-1) -- (4,-2);
\draw[thick] (4,-2) -- (5,-1);
\draw[thick] (5,-1) -- (6,0);
\draw[thick] (6,0) -- (7,1);
\draw[thick] (7,1) -- (8,2);
\draw[thick] (8,2) -- (9,1);
\draw[thick] (9,1) -- (10,0);
\foreach \i in {0,...,10}
    {\node at (\i,-0.3) {$\i$};
    }
\node[left] at (-0.5,0) {$0.0$};
\node[left] at (-0.5,1) {$0.5$};
\node[left] at (-0.5,2) {$1.0$};
\node[left] at (-0.5,3) {$1.5$};
\node[left] at (-0.5,4) {$2.0$};
\node[left] at (-0.5,5) {$2.5$};
\node[left] at (-0.5,-1) {$-0.5$};
\node[left] at (-0.5,-2) {$-1.0$};
\node[left] at (-0.5,-3) {$-1.5$};
\node[left] at (-0.5,-4) {$-2.0$};
\node[left] at (-0.5,-5) {$-2.5$};
\node[b3] at (4,0) {};
\draw[dashed,blue] (3,-1) -- (4,0);
\draw[dashed,blue] (4,0) -- (5,-1);
\draw[->,blue] (4,-2) -- (4,-1);
\node[above] at (4,-1) {$q$};
\node[b2] at (8,0) {};
\draw[dashed,red] (7,1) -- (8,0);
\draw[dashed,red] (8,0) -- (9,1);
\draw[->,red] (8,2) -- (8,1);
\node[below] at (8,1) {$p$};
\node at (6.5,1) {$\eta^{\zeta}$};
\end{tikzpicture}
}

\newcommand{\ssepcon}{\begin{tikzpicture}
\tikzset{
	b1/.style = {circle, draw=black,fill=black,minimum size = 3mm, inner sep = 1pt},
	b2/.style = {circle, draw=black,fill=white,minimum size = 3mm, inner sep = 1pt}}
	\node[b2] at (1,0) {};
	\node[b1] at (2,0) {};
	\node[b2] at (3,0) {};
	\node[b2] (B) at (4,0) {};
    \node[b1] (A) at (5,0) {};
    \node[b1] at (6,0) {};
	\node[b1] at (7,0) {};
	\node[b1] (C) at (8,0) {};
	\node[b2] (D) at (9,0) {};
	\node[b2] at (10,0) {};
    \draw[->,blue] (A) edge[bend right=40] (B);
    \draw[->,red] (C) edge[bend left=40] (D);
    \foreach \i in {1,...,10}
        {\draw (\i-0.2,-0.2) -- (\i+0.2,-0.2);
        }
    \node at (4.5,0.5) {$q$};
    \node at (8.5,0.5) {$p$};
    \node at (0.5,0.5) {$\zeta$};
\end{tikzpicture}
}

\newcommand{\specialconfigs}{\begin{tikzpicture}
\tikzset{
	b1/.style = {circle, draw=black,fill=black,minimum size = 1mm, inner sep = 1pt}}
    \draw[<->] (0.5,-7.5) -- (0.5,7.5);
    \draw[->] (0.5,0.0) -- (15.5,0.0);
    \node at (0.5,8) {$\eta(i)$};
    \node at (16.0,0.0) {$i$};
    \foreach \i in {1,...,14}
        {\node[b1] at (\i,7.5-0.5*\i) {};
         \node[b1] at (\i+1,-0.5*\i) {};
        }
    \foreach \i in {1,...,7}
        {\node[b1] at (\i,-0.5+0.5*\i) {};
         \node[b1] at (8+\i,-3.5+0.5*\i) {};
        }
    \draw[thick] (1,7) -- (15,0);
    \draw[thick] (1,0) -- (15,-7);
    \draw[dashed] (1,0) -- (8,3.5);
    \draw[thick] (8,3.5) -- (15,0);
    \draw[thick] (1,0) -- (8,-3.5);
    \draw[dashed] (8,-3.5) -- (15,0);
    \foreach \i in {1,...,15}
        {\node[below] at (\i,0) {$\i$};
        }
    \foreach \i in {-7.0,-6.5,-6.0,-5.5,-5.0,-4.5,-4.0,-3.5,-3.0,-2.5,-2.0,-1.5,-1.0,-0.5,0.0,0.5,1.0,1.5,2.0,2.5,3.0,3.5,4.0,4.5,5.0,5.5,6.0,6.5,7.0}
        {\node[left] at (0.3,\i) {$\i$};
        }     
    \node[above] at (3,6.2) {$\eta^{+}$};
    \node[above] at (6,2.7) {$\eta^{\wedge}$};
    \node[above] at (10,-2.3) {$\eta^{\vee}$};
    \node[above] at (13,-5.8) {$\eta^{-}$};
\end{tikzpicture}
}

\newcommand{\specconfigs}{\begin{tikzpicture}
\tikzset{
	b1/.style = {circle, draw=black,fill=white,minimum size = 0.5cm, inner sep = 1pt}}
    %\eta^{+} config
    \node at (0,4.5) {$\xi^{+}$};
    \node[b1] at (8,4.5) {$1$};
    \node[b1] at (9,4.5) {$2$};
    \node[b1] at (10,4.5) {$3$};
    \node[b1] at (11,4.5) {$4$};
    \node[b1] at (12,4.5) {$5$};
    \node[b1] at (13,4.5) {$6$};
    \node[b1] at (14,4.5) {$7$};
    \node[b1] at (15,4.5) {$8$};
    \node[b1] at (16,4.5) {$9$};
    \node[b1] at (17,4.5) {$10$};
    \node[b1] at (18,4.5) {$11$};
    \node[b1] at (19,4.5) {$12$};
    \node[b1] at (20,4.5) {$13$};
    \node[b1] at (21,4.5) {$14$};
    \node[b1] at (22,4.5) {$15$};
    %\eta^{\wedge} config
    \node at (0,3) {$\xi^{\wedge}$};
    \node[b1] at (1,3) {$1$};
    \node[b1] at (3,3) {$2$};
    \node[b1] at (5,3) {$3$};
    \node[b1] at (7,3) {$4$};
    \node[b1] at (9,3) {$5$};
    \node[b1] at (11,3) {$6$};
    \node[b1] at (13,3) {$7$};
    \node[b1] at (15,3) {$8$};
    \node[b1] at (16,3) {$9$};
    \node[b1] at (17,3) {$10$};
    \node[b1] at (18,3) {$11$};
    \node[b1] at (19,3) {$12$};
    \node[b1] at (20,3) {$13$};
    \node[b1] at (21,3) {$14$};
    \node[b1] at (22,3) {$15$};
    %\eta^{\vee} config
    \node at (0,1.5) {$\xi^{\vee}$};
    \node[b1] at (1,1.5) {$1$};
    \node[b1] at (2,1.5) {$2$};
    \node[b1] at (3,1.5) {$3$};
    \node[b1] at (4,1.5) {$4$};
    \node[b1] at (5,1.5) {$5$};
    \node[b1] at (6,1.5) {$6$};
    \node[b1] at (7,1.5) {$7$};
    \node[b1] at (8,1.5) {$8$};
    \node[b1] at (10,1.5) {$9$};
    \node[b1] at (12,1.5) {$10$};
    \node[b1] at (14,1.5) {$11$};
    \node[b1] at (16,1.5) {$12$};
    \node[b1] at (18,1.5) {$13$};
    \node[b1] at (20,1.5) {$14$};
    \node[b1] at (22,1.5) {$15$};
    %\eta^{-} config
    \node at (0,0) {$\xi^{-}$};
    \node[b1] at (1,0) {$1$};
    \node[b1] at (2,0) {$2$};
    \node[b1] at (3,0) {$3$};
    \node[b1] at (4,0) {$4$};
    \node[b1] at (5,0) {$5$};
    \node[b1] at (6,0) {$6$};
    \node[b1] at (7,0) {$7$};
    \node[b1] at (8,0) {$8$};
    \node[b1] at (9,0) {$9$};
    \node[b1] at (10,0) {$10$};
    \node[b1] at (11,0) {$11$};
    \node[b1] at (12,0) {$12$};
    \node[b1] at (13,0) {$13$};
    \node[b1] at (14,0) {$14$};
    \node[b1] at (15,0) {$15$};
    \foreach \i in {1,...,22}
        {\node[below] at (\i,-0.4) {$\i$};
         \foreach \j in {0,1.5,3,4.5}
            {\draw (\i-0.25,\j-0.3) -- (\i+0.25,\j-0.3);
            }
         
        }
	\end{tikzpicture}}

\newcommand{\obep}{\begin{tikzpicture}
\tikzset{
	b1/.style = {circle, draw=black,fill=black,minimum size = 3mm, inner sep = 1pt},
    b2/.style = {rectangle,draw=black,fill=white,minimum size = 3mm, inner sep = 4pt}}
    \foreach \i in {1,...,5}
        {\draw (\i+1.8,-0.2) -- (\i+2.2,-0.2);
        }
    \node[b1] (P) at (5,0) {};
    \node[b2] (R1) at (1,0) {Reservoir};
    \node[b2] (R2) at (9,0) {Reservoir};
    \node (L) at (3,0) {};
    \node (R) at (7,0) {};
    \node (A) at (4,0) {};
    \node (B) at (6,0) {};
    \draw[->] (P) edge[bend right=40] (A);
    \draw[->] (P) edge[bend left=40] (B);
    \draw[->] (R1) edge[bend left=40] (L);
    \draw[->] (R) edge[bend left=40] (8.75,0.35);
    \node (RR1) at (1,-0.3) {};
    \node (LL) at (3,-0.2) {};
    \node (LL1) at (9,-0.3) {};
    \node (RR) at (7,-0.2) {};
    \draw[->] (LL) edge[bend left=40] (RR1);
    \draw[->] (LL1) edge[bend left=40] (RR);
    \node at (2,0.7) {$\alpha$};
    \node at (2,-1) {$\gamma$};
    \node at (8,0.7) {$\beta$};
    \node at (8,-1) {$\delta$};
    \node at (4.5,0.5) {$p$};
    \node at (5.5,0.5) {$q$};
\end{tikzpicture}
}

\newcommand{\obepheight}{\begin{tikzpicture}
\tikzset{
         b1/.style = {circle, draw=black,fill=black,minimum size = 1mm, inner sep = 1pt}}
    \draw[->] (-0.5,0) -- (14.5,0);
    \draw[<->] (-0.5,-7.5) -- (-0.5,7.5);
    \foreach \i in {0,...,14}
        {\node[b1] at (\i,-0.5*\i) {};
         \node[b1] at (\i,0.5*\i) {};
        }
    \foreach \i in {1,...,15}
        {\node[below] at (\i-1,0) {$\i$};
        }
    \draw[dashed] (3,1.5) -- (14,7);
    \draw[dashed] (1,-0.5) -- (14,-7);
    \foreach \i in {-7.0,-6.5,-6.0,-5.5,-5.0,-4.5,-4.0,-3.5,-3.0,-2.5,-2.0,-1.5,-1.0,-0.5,0.0,0.5,1.0,1.5,2.0,2.5,3.0,3.5,4.0,4.5,5.0,5.5,6.0,6.5,7.0}
        {\node[left] at (-0.7,\i) {$\i$};
        }
    \node[b1] at (2,0) {};
    \node[b1] at (3,-0.5) {};
    \node[b1] at (4,-1) {};
    \node[b1] at (5,-0.5) {};
    \node[b1] at (6,-1) {};
    \node[b1] at (7,-1.5) {};
    \node[b1] at (8,-1) {};
    \node[b1] at (9,-0.5) {};
    \node[b1] at (10,-1) {};
    \node[b1] at (11,-1.5) {};
    \node[b1] at (12,-2) {};
    \node[b1] at (13,-2.5) {};
    \node[b1] at (14,-3) {};
    \node[b1] at (14,-2) {};
    \draw[thick] (0,0) -- (1,-0.5);
    \draw[thick] (1,-0.5) -- (2,0);
    \draw[thick] (2,0) -- (3,-0.5);
    \draw[thick] (3,-0.5) -- (4,-1);
    \draw[thick] (4,-1) -- (5,-0.5);
    \draw[thick] (5,-0.5) -- (6,-1);
    \draw[thick] (6,-1) -- (7,-1.5);
    \draw[thick] (7,-1.5) -- (8,-1);
    \draw[thick] (8,-1) -- (9,-0.5);
    \draw[thick] (9,-0.5) -- (14,-3);
    \draw[dashed] (13,-2.5) -- (14,-2);
    \draw[->] (14,-3) -- (14,-2.5);
    \node at (14,-2.3) {$\delta$};
    \node[b1] at (4,1) {};
    \node[b1] at (5,1.5) {};
    \node[b1] at (6,2) {};
    \node[b1] at (7,1.5) {};
    \node[b1] at (8,2) {};
    \node[b1] at (9,2.5) {};
    \node[b1] at (10,3) {};
    \node[b1] at (11,2.5) {};
    \node[b1] at (12,3) {};
    \node[b1] at (13,3.5) {};
    \node[b1] at (14,4) {};
    \node[b1] at (14,3) {};
    \draw[thick] (0,0) -- (3,1.5);
    \draw[thick] (3,1.5) -- (4,1);
    \draw[thick] (4,1) -- (5,1.5);
    \draw[thick] (5,1.5) -- (6,2);
    \draw[thick] (6,2) -- (7,1.5);
    \draw[thick] (7,1.5) -- (8,2);
    \draw[thick] (8,2) -- (9,2.5);
    \draw[thick] (9,2.5) -- (10,3);
    \draw[thick] (10,3) -- (11,2.5);
    \draw[thick] (11,2.5) -- (12,3);
    \draw[thick] (12,3) -- (13,3.5);
    \draw[thick] (13,3.5) -- (14,4);
    \draw[dashed] (13,3.5) -- (14,3);
    \draw[->] (14,4) -- (14,3.5);
    \node at (14,3.3) {$\beta$};
    \node[b1] at (7,2.5) {};
    \draw[dashed] (6,2) -- (7,2.5);
    \draw[dashed] (7,2.5) -- (8,2);
    \draw[->] (7,1.5) -- (7,2);
    \node at (7,2.2) {$p$};
    \node[b1] at (9,-1.5) {};
    \draw[dashed] (8,-1) -- (9,-1.5);
    \draw[dashed] (9,-1.5) -- (10,-1);
    \draw[->] (9,-0.5) -- (9,-1);
    \node at (9,-1.2) {$q$};
    \node at (14.7,0) {$x$};
    \node at (0,7.7) {$h^{\zeta}(x)$};
    \node at (13,4) {$h^{\zeta_{1}}$};
    \node at (13,-3) {$h^{\zeta_{2}}$};
    \coordinate (A) at (0.8,-0.6);
    \coordinate (B) at (14.2,-7.3);
    \coordinate (C) at (14.2,-0.4);
    \coordinate (D) at (13.4,-0.4);
    \coordinate (E) at (13.4,0.6);
    \coordinate (F) at (7.1,3.75);
    \coordinate (G) at (0.8,0.6);
    \draw [black!30,pattern=dots,pattern color=black!30] (A) -- (B) -- (C) -- (D) -- (E) -- (F) -- (G) -- cycle;
    \node at (10,-3) {\Huge{$\RN{1}$}};
\end{tikzpicture}
}

\newcommand{\obepconfigs}{\begin{tikzpicture}
\tikzset{
	b1/.style = {circle, draw=black,fill=black,minimum size = 0.5cm, inner sep = 1pt},
    b2/.style = {rectangle, draw=black,fill=white,minimum size = 0.5cm, inner sep = 1pt}}
    \node at (0,1.5) {$\zeta_{1}$};
    \node[b1] at (1,1.5) {};
    \node[b1] at (2,1.5) {};
    \node[b1] at (3,1.5) {};
    \node[b1] at (5,1.5) {};
    \node[b1] at (6,1.5) {};
    \node[b1] (P1) at (8,1.5) {};
    \node (P2) at (7,1.5) {};
    \node[b1] at (9,1.5) {};
    \node[b1] at (10,1.5) {};
    \node[b1] at (12,1.5) {};
    \node[b1] at (13,1.5) {};
    \node[b1] (B1) at (14,1.5) {};
    \node[b2] at (16,1.5) {Reservoir};
    \node (B2) at (16,1.75){};
    \draw[->] (P1) edge[bend right=40] (P2);
    \draw[->] (B1) edge[bend left=40] (B2);
    \node at (15,2.3) {$\beta$};
    \node at (7.5,2) {$p$};
    \node at (0,0) {$\zeta_{2}$};
    \node[b1] at (2,0) {};
    \node[b1] at (5,0) {};
    \node[b1] at (8,0) {};
    \node[b1] (Q1) at (9,0) {};
    \node (Q2) at (10,0) {};
    \node[b2] (D1) at (16,0) {Reservoir};
    \node (D2) at (14,0) {};
    \draw[->] (Q1) edge[bend left=40] (Q2);
    \draw[->] (D1) edge[bend right=40] (D2);
    \node at (15,0.7) {$\delta$};
    \node at (9.5,0.5) {$q$};
    \foreach \i in {1,...,14}
        {\node[below] at (\i,-0.4) {$\i$};
        \foreach \j in {0,1.5}
            {\draw (\i-0.25,\j-0.3) -- (\i+0.25,\j-0.3);
            }
        }
	\end{tikzpicture}}

\newcommand{\keyopen}{\begin{tikzpicture}
\tikzset{
	b1/.style = {circle, draw=black,fill=black,minimum size = 0.5cm, inner sep = 1pt},
    b2/.style = {circle, draw=black,fill=white,minimum size = 0.5cm, inner sep = 1pt},
    b3/.style = {rectangle, draw=black,fill=white,minimum size = 0.5cm, inner sep = 1pt}}
    \node[b1] at (1,3) {};
    \node (F1) at (1.5,2.7) {};
    \node[b2] at (2,3) {};
    \node[b1] at (3,3) {};
    \node (F2) at (3,2.6) {};
    \node[b1]  at (4,3) {};
    \node (F3) at (4,2.6) {};
    \node[b1] at (5,3) {};
    \node (F4) at (5.5,2.7) {};
    \node[b2] at (6,3) {};
    \node[b1] at (7,3) {};
    \node (F5) at (7,2.6) {};
    \node[b1] at (8,3) {};
    \node (F6) at (8,2.6) {};
    \node[b1] at (9,3) {};
    \node (F7) at (9.5,2.7) {};
    \node[b2] at (10,3) {};
    \node[b1] at (11,3) {};
    \node (F8) at (11,2.6) {};
    \node[b1] at (12,3) {};
    \node (F9) at (12,2.6) {};
    \node[b1] (B1) at (13,3) {};
    \node (F10) at (13,2.6) {};
    \node[b2] (B2) at (14,3) {};
    \node[b2] at (15,3) {};
    \node[b2] at (16,3) {};
    \node[b2] at (17,3) {};
    \node[b1] (A1) at (5,0) {};
    \node[b2] (A2) at (6,0) {};
    \node[b2] (A3) at (7,0) {};
    \node[b1] (A4) at (8,0) {};
    \node[b2] (A5) at (9,0) {};
    \node[b2] (A6) at (10,0) {};
    \node[b1] (A7) at (11,0) {};
    \node[b2] (A8) at (12,0) {};
    \node[b2] (A9) at (13,0) {};
    \node[b3] (R) at (15,0) {Reservoir};
    \node (R1) at (15.5,0.2) {};
    \draw[red] (0.7,2.7) rectangle ++ (1.6,0.6);
    \draw[red] (4.7,2.7) rectangle ++ (1.6,0.6);
    \draw[red] (8.7,2.7) rectangle ++ (1.6,0.6);
    \foreach \i in {1,...,17}
        {\draw (\i-0.2,2.6) -- (\i+0.2,2.6);
        }
    \foreach \i in {5,...,13}
        {\draw (\i-0.2,-0.4) -- (\i+0.2,-0.4);
        }
    \draw[dashed] (F1) -- (A1);
    \draw[dashed] (F2) -- (A2);
    \draw[dashed] (F3) -- (A3);
    \draw[dashed] (F4) -- (A4);
    \draw[dashed] (F5) -- (A5);
    \draw[dashed] (F6) -- (A6);
    \draw[dashed] (F7) -- (A7);
    \draw[dashed] (F8) -- (A8);
    \draw[dashed] (F9) -- (A9);
    \draw[dashed] (F10) -- (R1);
    \draw[->,red] (A7) edge[bend left=40] (A8);
    \draw[->] (11,3.3) edge[bend right=40] (10,3.4);
    \node (B) at (9.5,3.2) {};
    \draw[->,red] (B) edge[bend left=50] (11,3.6);
    \node at (10.5,4) {$q$};
    \node at (11.5,0.5) {$q$};
    \node at (13.5,3.6) {$p$};
    \node at (14,0.7) {$p$};
    \draw[->] (B1) edge[bend left=40] (B2);
    \draw[->] (R) edge[bend right=40] (A9);
    \end{tikzpicture}}

\newcommand{\keysep}{\begin{tikzpicture}
\tikzset{
	b1/.style = {circle, draw=black,fill=black,minimum size = 0.5cm, inner sep = 1pt},
    b2/.style = {circle, draw=black,fill=white,minimum size = 0.5cm, inner sep = 1pt}}
    \node[b1] at (1,3){};
    \node (F1) at (1,2.7) {};
    \node[b1] at (2,3){};
    \node (F2) at (2.5,2.6) {};
    \node[b2] at (3,3){};
    \node[b1] at (4,3){};
    \node (F3) at (4,2.6) {};
    \node[b1] at (5,3){};
    \node (F4) at (5,2.6) {};
    \node[b1] at (6,3){};
    \node (F5) at (6.5,2.6) {};
    \node[b2] at (7,3){};
    \node[b1] at (8,3){};
    \node (F6) at (8,2.6) {};
    \node[b1] at (9,3){};
    \node (F7) at (9,2.6) {};
    \node[b1] at (10,3){};
    \node (F8) at (10.5,2.6) {};
    \node[b2] at (11,3){};
    \node[b1] at (12,3){};
    \node (F9) at (12.5,2.6) {};
    \node[b2] at (13,3){};
    \node[b1] at (14,3){};
    \node (F10) at (14,2.6) {};
    \node[b1] at (15,3){};
    \node (F11) at (15.5,2.6) {};
    \node[b2] at (16,3){};
    \node[b1] at (17,3){};
    \node[b2] (A1) at (4,0) {};
    \node[b1] (A2) at (5,0) {};
    \node[b2] (A3) at (6,0) {};
    \node[b2] (A4) at (7,0) {};
    \node[b1] (A5) at (8,0) {};
    \node[b2] (A6) at (9,0) {};
    \node[b2] (A7) at (10,0) {};
    \node[b1] (A8) at (11,0) {};
    \node[b1] (A9) at (12,0) {};
    \node[b2] (A10) at (13,0) {};
    \node[b1] (A11) at (14,0) {};
    \foreach \i in {1,...,17}
        {\draw (\i-0.2,2.6) -- (\i+0.2,2.6);
        }
    \foreach \i in {4,...,14}
        {\draw (\i-0.2,-0.4) -- (\i+0.2,-0.4);
        }  
    \draw[dashed] (F1) -- (A1);
    \draw[dashed] (F2) -- (A2);
    \draw[dashed] (F3) -- (A3);
    \draw[dashed] (F4) -- (A4);
    \draw[dashed] (F5) -- (A5);
    \draw[dashed] (F6) -- (A6);
    \draw[dashed] (F7) -- (A7);
    \draw[dashed] (F8) -- (A8);
    \draw[dashed] (F9) -- (A9);
    \draw[dashed] (F10) -- (A10);
    \draw[dashed] (F11) -- (A11);
    \draw[red] (1.7,2.7) rectangle ++ (1.6,0.6);
    \draw[red] (5.7,2.7) rectangle ++ (1.6,0.6);
    \draw[red] (9.7,2.7) rectangle ++ (1.6,0.6);
    \draw[red] (11.7,2.7) rectangle ++ (1.6,0.6);
    \draw[red] (14.7,2.7) rectangle ++ (1.6,0.6);
\end{tikzpicture}}

\newcommand{\sizecoupling}{\begin{tikzpicture}
\tikzset{
    b1/.style = {circle, draw=black,fill=black,minimum size = 1mm, inner sep = 1pt},
    b2/.style = {circle, draw=red,fill=red,minimum size = 1mm, inner sep = 1pt},
    b3/.style = {circle, draw=blue,fill=blue,minimum size = 1mm, inner sep = 1pt}}
    \draw[->] (0,0) -- (0,15);
    \draw[->] (0,0) -- (16,0);
    \foreach \i in {1,...,15}
        {\node[b1] at (\i,15-\i){};
         \node[below] at (\i,-0.1){$\i$};
        }
    \foreach \i in {0.0,0.5,1.0,1.5,2.0,2.5,3.0,3.5,4.0,4.5,5.0,5.5,6.0,6.5,7.0}
        {\node[left] at (-0.1,2*\i){$\i$};
        }
    \node at (16.2,0) {$x$};
    \node at (0,15.5) {$h^{\zeta}(x)$};
    \draw[dashed] (15,0) -- (8,7);
    \draw[thick] (8,7) -- (1,14);
    \coordinate (A) at (0.8,14.4);
    \coordinate (B) at (7.2,8);
    \coordinate (C) at (0.8,1.6);
    \draw [black!30,pattern=dots,pattern color=black!30] (A) -- (B) -- (C) -- cycle;
    \foreach \i in {9,...,15}
        {\node[b2] at (16-\i,15-\i){};
        }
    \draw[red] (8,7) -- (1,0);
    \node[b3] at (6,7) {};
    \node[b3] at (5,8) {};
    \node[b3] at (4,7) {};
    \node[b3] at (5,6) {};
    \node[b3] at (4,9) {};
    \node[b3] at (3,8) {};
    \node[b3] at (2,9) {};
    \node[b3] at (1,8) {};
    \node[b3] at (1,10) {};
    \draw[blue,dashed] (4,7) -- (5,6);
    \draw[blue,dashed] (5,6) -- (6,7);
    \draw[blue,->] (5,8) -- (5,7);
    \node[below] at (5,6.8) {$q$};
    \draw[blue, dashed] (4,9) -- (3,8);
    \draw[blue, dashed] (4,9) -- (5,8);
    \draw[blue,->] (4,7) -- (4,8);
    \node[above] at (4,8.2) {$p$};
    \draw[blue] (7,8) -- (6,7);
    \draw[blue] (6,7) -- (5,8);
    \draw[blue] (5,8) -- (4,7);
    \draw[blue] (4,7) -- (3,8);
    \draw[blue] (3,8) -- (2,9);
    \draw[blue] (2,9) -- (1,10);
    \draw[blue,dashed] (1,8) -- (2,9);
    \draw[blue,->] (1,10) -- (1,9);
    \node[below] at (1,8.8) {$q$};
    \node at (3,7) {\Huge{$\RN{2}$}};
    \end{tikzpicture}}

\newcommand{\zerosegmap}{\begin{tikzpicture}[a/.style={circle, draw=black, fill=white, minimum size=3mm, inner sep=1pt},
b/.style={circle, draw=black, fill=black, minimum size=3mm, inner sep=1pt}]
\node[b] at (1,4) {};
\node[b] at (2,4) {};
\node[a] at (3,4) {};
\node[b] at (4,4) {};
\node[b] at (5,4) {};
\node[a] at (6,4) {};
\node[b] at (7,4) {};
\node[b] at (8,4) {};
\node[b] at (9,4) {};
\node[b] at (10,4) {};
\node[a] at (11,4) {};
\node[b] at (12,4) {};
\node[a] at (13,4) {};
\node[a] at (14,4) {};
\node[a] at (15,4) {};
\node[b] at (5,0) {};
\node[b] at (5,0.5) {};
\node[b] at (6,0) {};
\node[b] at (6,0.5) {};
\node[b] at (7,0) {};
\node[b] at (7,0.5) {};
\node[b] at (7,1) {};
\node[b] at (7,1.5) {};
\node[b] at (8,0) {};
\node[a] at (9,0) {};
\node[a] at (10,0) {};
\node[a] at (11,0) {};
\foreach \i in {1,...,15}
    {\draw (\i-0.2,3.8) -- (\i+0.2,3.8);
    }
\foreach \i in {1,...,7}
    {\draw (3.8+\i,-0.2) -- (4.2+\i,-0.2);
    }
\draw[->] (9,3) -- (9,1);
\node[right] at (9.2,2) {$\widetilde{\Pi}_{N,k}$};
\end{tikzpicture}}

\newcommand{\freezero}{\begin{tikzpicture}
[a/.style={circle, draw=black, fill=white, minimum size=3mm, inner sep=1pt},
b/.style={circle, draw=black, fill=black, minimum size=3mm, inner sep=1pt}]
\node[b] at (7,0){};
\node[b] at (8,0){};
\node[b] at (8,0.5){};
\node[b] (A1) at (8,1){};
\node (B1) at (7,1){};
\node[b]  at (9,0){};
\node[b] at (6,5){};
\node[b] at (7,5){};
\node[b] at (7,5.5){};
\node[b] at (8,5){};
\node[b] at (8,5.5){};
\node[b] at (8,6){};
\node[b] (A2) at (8,6.5){};
\node (B2) at (7,6.5){};
\node[b] at (9,5){};
\node[b] at (9,5.5){};
\node[a] at (1,10){};
\node[b] at (2,10){};
\node[a] at (3,10){};
\node[b] at (4,10){};
\node[b] at (5,10){};
\node[a] (B3) at (6,10){};
\node[b] (A3) at (7,10){};
\node[b] at (8,10){};
\node[b] at (9,10){};
\node[b] at (10,10){};
\node[a] at (11,10){};
\node[b] at (12,10){};
\node[b] at (13,10){};
\foreach \i in {1,...,13}
    {\draw (\i-0.2,9.7) -- (\i+0.2,9.7);
     \node[below] at (\i,9.6) {$\i$};
    }
\foreach \i in {1,...,5}
    {\draw (\i+3.8,4.7) -- (\i+4.2,4.7);
     \draw (\i+3.8,-0.3) -- (\i+4.2,-0.3);
     \node[below] at (\i+4,4.6) {$\i$};
     \node[below] at (\i+4,-0.4) {$\i$};
    }
\draw[->] (6,8.5)--(6,6.5);
\draw[->] (6,3.5)--(6,1.5);
\node[left] at (5.8,7.5) {$\widetilde\Pi_{13,9}$};
\node[left] at (5.8,2.5) {$\Lambda$};
\draw[->] 
(A1) edge[bend right=40] (B1)
(A2) edge[bend right=40] (B2)
(A3) edge[bend right=40] (B3);
\node[above] at (6.5,10.3){$q$};
\node[above] at (7.5,6.8){$q$};
\node[above] at (7.5,1.3){$q$};
\draw[red] (5.7,4.78) rectangle ++ (3.6,0.46);
\end{tikzpicture}}

\newcommand{\leftrightjump}{\begin{tikzpicture}[b/.style={circle, draw=black, fill=black, minimum size=3mm, inner sep=1pt}]
\node[b] at (1,2) {};
\node[b] (R1) at (2,2) {};
\node (R2) at (3,2) {};
\node (L2) at (1,0) {};
\node[b] (L1) at (2,0) {};
\node[b] at (3,0) {};
\node at (0,0) {$\cdots$};
\node at (0,2) {$\cdots$};
\node at (4,0) {$\cdots$};
\node at (4,2) {$\cdots$};
\foreach \i in {1,2,3}
    {\draw (\i-0.2,-0.2) -- (\i+0.2,-0.2);
     \draw (\i-0.2,1.8) -- (\i+0.2,1.8);
    }
\draw[->] (R1) edge[bend left=40] (R2);
\draw[->] (L1) edge[bend right=40] (L2);
\node at (2.5,2.5) {$p$};
\node at (1.5,0.5) {$q$};
\end{tikzpicture}
}

\newcommand{\dependencies}{\begin{tikzpicture}
\tikzset{a/.style = {rectangle,draw=black,fill=white,minimum size = 0.5cm, inner sep = 4pt}}
\node[a] (A) at (0,0) {Section 4.5};
\node[a] (B) at (3,0) {Section 4.4};
\node[a] (C) at (3,2) {Section 4.3};
\node[a] (D) at (6,1) {Section 4.2};
\node (D1) at (5,1) {};
\node[a] at (9,2) {Theorem 2.3};
\node (E) at (8,2){};
\node[a] (F) at (6,3) {Section 4.1};
\draw[-{latex}] 
(A) edge (B)
(B) edge (D1)
(C) edge (D1)
(D) edge (E)
(F) edge (E);
\end{tikzpicture}}

\newcommand{\dependintro}{\begin{tikzpicture}
\tikzset{a/.style = {rectangle,draw=black,fill=white,minimum size = 0.5cm, inner sep = 4pt}}
\node[a] (A) at (0,-1) {Section 1.2};
\node[a] (B) at (2,0) {Section 1.3};
\node (B1) at (1.1,0){};
\node[a] (C) at (4,1) {Section 1.4};
\node (C1) at (3.1,1){};
\node[a] (D) at (4,-1) {Section 1.5};
\node (D1) at (3.1,-1){};
\node[a] (E) at (2,-2) {Section 1.7};
\node (E1) at (1.1,-2) {};
\node[a] (F) at (0,-3) {Section 1.8};
\node[a] at (6,0) {Chapter 2};
\node (G1) at (5.2,0){};
\node[a] (H) at (6,-2) {Chapter 3};
\node[a] (I) at (6,-3) {Chapter 4};
\draw[-{latex}] 
(A) edge (B1)
(A) edge (E1)
(B) edge (C1)
(B) edge (D1)
(B) edge (G1)
(C) edge (G1)
(D) edge (G1)
(E) edge (H)
(F) edge (I);
\end{tikzpicture}}

\newcommand{\segtransrates}{\begin{tikzpicture}[b/.style={circle, draw=black, fill=black, minimum size=3mm, inner sep=1pt},
w/.style={circle, draw=black, fill=white, minimum size=3mm, inner sep=1pt}]
\node[w] at (1,0) {};
\node[b] at (2,0) {};
\node[w] (R2) at (3,0) {};
\node[b] (R1) at (4,0) {};
\node[b] at (5,0) {};
\node[b] (L1) at (6,0) {};
\node[w] (L2) at (7,0) {};
\node[b] at (8,0) {};
\foreach \i in {1,...,8}
    {\draw (\i-0.25,-0.25) -- (\i+0.25,-0.25);
    }
\draw[->] (R1) edge[bend right=40] (R2);
\draw[->] (L1) edge[bend left=40] (L2);
\node at (3.5,0.6) {$q$};
\node at (6.5,0.6) {$p$};
\end{tikzpicture}
}

\newcommand{\ergodicfigex}{\begin{tikzpicture}[b/.style={circle, draw=black, fill=black, minimum size=3mm, inner sep=1pt},
w/.style={circle, draw=black, fill=white, minimum size=3mm, inner sep=1pt}]
\node[b] at (1,0) {};
\node[b] at (2,0) {};
\node[w] (R2) at (3,0) {};
\node[b] (R1) at (4,0) {};
\node[b] at (5,0) {};
\node[b] (L1) at (6,0) {};
\node[w] (L2) at (7,0) {};
\node[b] at (8,0) {};
\foreach \i in {1,...,8}
    {\draw (\i-0.25,-0.25) -- (\i+0.25,-0.25);
    }
\end{tikzpicture}
}
\newcommand{\cirtransrates}{
\begin{tikzpicture}[a/.style={circle, draw=black, fill=white, minimum size=3mm, inner sep=1pt},
b/.style={circle, draw=black, fill=black, minimum size=3mm, inner sep=1pt}]
  \draw[very thick] (0,0) circle(2cm);
  \node[b] (J) at (90+0*24:2){};
  \node[below] at (90+0*24:1.8) {$0$};
  \node[b] (A) at (90-1*24:2){};
  \node at (90-1.5*24:2.7){$1/2$};
  \node[a] (B) at (90-2*24:2){};
  \node at (90-2.5*24:2.7){$1/2$};
  \node[b] (C) at (90-3*24:2){};
  \node at (90-4.5*24:2.7){$1/2$};
  \node[b] (D) at (90-4*24:2){};
  \node[a] (E) at (90-5*24:2){};
  \node at (90-5.5*24:2.7){$1/2$};
  \node[b] (F) at (90-6*24:2){};
  \node[b] at (90-7*24:2){};
  \node[b] at (90-8*24:2){};
  \node[b] (G) at (90-9*24:2){};
  \node at (90-9.5*24:2.7){$1/2$};
  \node[a] (H) at (90-10*24:2){};
  \node[b] at (90-11*24:2){};
  \node[a] at (90-12*24:2){};
  \node[a] at (90-13*24:2){};
  \node[a] (I) at (90-14*24:2){};
  \node at (90-14.5*24:2.7){$1/2$};
  \draw[->] (A) edge[bend left=60] (B);
  \draw[->] (C) edge[bend right=60] (B);
  \draw[->] (D) edge[bend left=60] (E);
  \draw[->] (F) edge[bend right=60] (E);
  \draw[->] (G) edge[bend left=60] (H);
  \draw[->] (J) edge[bend right=60] (I);
\end{tikzpicture}}
\newcommand{\ergcirfigex}{
\begin{tikzpicture}[a/.style={circle, draw=black, fill=white, minimum size=3mm, inner sep=1pt},
b/.style={circle, draw=black, fill=black, minimum size=3mm, inner sep=1pt}]
  \draw[very thick] (0,0) circle(2cm);
  \node[b] at (90+0*24:2){};
  \node[below] at (90+0*24:1.8) {$0$};
  \node[b] at (90-1*24:2){};
  \node[a] at (90-2*24:2){};
  \node[b] at (90-3*24:2){};
  \node[b] at (90-4*24:2){};
  \node[a] at (90-5*24:2){};
  \node[b] at (90-6*24:2){};
  \node[b] at (90-7*24:2){};
  \node[b] at (90-8*24:2){};
  \node[b] at (90-9*24:2){};
  \node[a] at (90-10*24:2){};
  \node[b] at (90-11*24:2){};
  \node[a] at (90-12*24:2){};
  \node[b] at (90-13*24:2){};
  \node[a] at (90-14*24:2){};
\end{tikzpicture}}

\newcommand{\zeropens}{\begin{tikzpicture}
\tikzset{
b1/.style = {circle, draw=black,fill=black,minimum size = 3mm, inner sep = 1pt},
b2/.style = {circle, draw=black,fill=white,minimum size = 3mm, inner sep = 1pt},
b3/.style = {rectangle,draw=black,fill=white,minimum size = 3mm, inner sep = 4pt}}
\draw[very thick] (0,0) circle(2cm);
\node[b1] at (90:2){};
\node[b1] at (90:2.5){};
\node[b1] at (90:3){};
\node[b1] (A) at (90:3.5){};
\node[b1] at (90+45:2){};
\node (F) at (90+45:2.5){};
\node[b1] at (90-45:2){};
\node (B) at (90-45:2.5){};
\node[b1] at (90-2*45:2){};
\node[b1] (G) at (90-2*45:2.5){};
\node[b2] at (90-3*45:2){};
\node (D) at (90-3*45:2.5){};
\node[b2] at (90-4*45:2){};
\node[b2] at (90-5*45:2){};
\node[b2] at (90-6*45:2){};
\node at (0,0) {$\omega$};
\draw[->,thick] (4,0) -- (6,0);
\node[left] at (5.25,0.25) {$\Phi$};
\node[b3] at (7,-1){Reservoir};
\node[b2] (E1) at (9,-1){$1$};
\node (F1) at (10,-1) {};
\node[b2] (A1) at (13,-1){$2$};
\node (B1) at (12,-1){};
\node[b2] (G1) at (14,-1){$3$};
\node[b3] (C1) at (17,-1){Reservoir};
\node (D1) at (15,-1){};
\draw[->,red] (A1) edge[bend right=40] (B1)
(A) edge[bend left=40] (B);
\draw[->,violet] (C1) edge[bend right=40] (D1)
(G) edge[bend left=40] (D);     
\draw[->] (E1) edge[bend left=40] (F1)
(A) edge[bend right=40] (F);
\draw[->,blue](G1) edge[bend left=40] (D1)
(G) edge[bend right=40] (B);
\node at (12,0) {$\Phi(\omega)$};
\foreach \i in {9,...,15}
        {\draw (\i-0.25,-1.25) -- (\i+0.25,-1.25);
        }
\node at (9.5,-0.5) {$1/2$};
\node at (12.5,-0.5) {$1/2$};
\node at (14.5,-0.5) {$1/2$};
\node at (16,-0.25) {$1/2$};
\node at (67.5:3.75) {$1/2$};
\node at (22.5:3.25) {$1/2$};
\node at (-22.5:3.25) {$1/2$};
\node at (112.5:3.75) {$1/2$};
\node at (135:1.5) {$1$};
\node at (90:1.5) {$2$};
\node at (45:1.5) {$3$};
\node at (0:1.5) {$4$};
 \end{tikzpicture}}

\newcommand{\fullcouple}{\begin{tikzpicture}
\tikzset{
	b1/.style = {circle, draw=black,fill=black,minimum size = 1mm, inner sep = 1pt},
	b2/.style = {circle, draw=red,fill=red,minimum size = 1mm, inner sep = 1pt},
	b3/.style = {circle, draw=blue,fill=blue,minimum size = 1mm, inner sep = 1pt}}
    \node[b1] at (1,6) {};
    \node[b1] at (2,5) {};
    \node[b1] at (3,4) {};
    \node[b1] at (4,3) {};
    \node[b1] at (5,2) {};
    \node[b1] at (6,1) {};
    \node[b1] at (7,0) {};
    \node[b1] at (8,3) {};
    \node[b1] at (9,2) {};
    \node[b1] at (10,1) {};
    \node[b1] at (11,2) {};
    \node[b1] at (12,1) {};
    \node[b1] at (13,0) {};
    \node[b1] at (14,-1) {};
    \node[b1] at (15,-2) {};
    \draw[->] (0.5,-2.5) -- (15.5,-2.5);
    \draw[->] (0.5,-2.5) -- (0.5,6.5);
    \node at (15.75,-2.5) {$i$};
    \node at (0.5,6.75) {$\eta^{\xi}(i)$};
    \draw[thick] (1,6) -- (2,5);
    \draw[thick] (2,5) -- (3,4);
    \draw[thick] (3,4) -- (4,3);
    \draw[thick] (4,3) -- (5,2);
    \draw[thick] (5,2) -- (6,1);
    \draw[thick] (6,1) -- (7,0);
    \draw[thick] (7,0) -- (8,3);
    \draw[thick] (8,3) -- (9,2);
    \draw[thick] (9,2) -- (10,1);
    \draw[thick] (10,1) -- (11,2);
    \draw[thick] (11,2) -- (12,1);
    \draw[thick] (12,1) -- (13,0);
    \draw[thick] (13,0) -- (14,-1);
    \draw[thick] (14,-1) -- (15,-2);
    \node[below] at (1,-2.75) {$1$};
    \node[below] at (2,-2.75) {$2$};
    \node[below] at (3,-2.75) {$3$};
    \node[below] at (4,-2.75) {$4$};
    \node[below] at (5,-2.75) {$5$};
    \node[below] at (6,-2.75) {$6$};
    \node[below] at (7,-2.75) {$7$};
    \node[below] at (8,-2.75) {$8$};
    \node[below] at (9,-2.75) {$9$};
    \node[below] at (10,-2.75) {$10$};
    \node[below] at (11,-2.75) {$11$};
    \node[below] at (12,-2.75) {$12$};
    \node[below] at (13,-2.75) {$13$};
    \node[below] at (14,-2.75) {$14$};
    \node[below] at (15,-2.75) {$15$};
    \node[left] at (0.25,-2) {$-2$};
    \node[left] at (0.25,-1) {$-1$};
    \node[left] at (0.25,0) {$0$};
    \node[left] at (0.25,1) {$1$};
    \node[left] at (0.25,2) {$2$};
    \node[left] at (0.25,3) {$3$};
    \node[left] at (0.25,4) {$4$};
    \node[left] at (0.25,5) {$5$};
    \node[left] at (0.25,6) {$6$};
    \draw[red,->] (7,0) -- (7,1);
    \draw[blue,->] (8,3) -- (8,2);
    \node[b2] at (7,2) {};
    \draw[red,dashed] (6,1) -- (7,2);
    \draw[red,dashed] (7,2) -- (8,3);
    \node[b3] at (8,1) {};
    \draw[blue,dashed] (9,2) -- (8,1);
    \draw[blue,dashed] (8,1) -- (7,0);
    \draw[red,->] (10,1) -- (10,2);
    \node[b2] at (10,3) {};
    \draw[red,dashed] (9,2) -- (10,3);
    \draw[red,dashed] (10,3) -- (11,2);
    \draw[blue,->] (11,2) -- (11,1);
    \node[b3] at (11,0) {};
    \draw[blue,dashed] (10,1) -- (11,0);
    \draw[blue,dashed] (11,0) -- (12,1);
    \draw[blue,->] (1,6) -- (1,5);
    \node[b3] at (1,4) {};
    \draw[blue,dashed] (2,5) -- (1,4);
    \draw[red,->] (15,-2) -- (15,-1);
    \node[b2] at (15,0) {};
    \draw[red,dashed] (14,-1) -- (15,0);
    \node at (7,1.25) {$p$};
    \node at (8,1.75) {$q$};
    \node at (10,2.25) {$p$};
    \node at (11,0.75) {$q$};
    \node at (1,4.75) {$q$};
    \node at (15,-0.75) {$p$};
\end{tikzpicture}}

\newcommand{\spconfigs}{\begin{tikzpicture}
\tikzset{
	b1/.style = {circle, draw=black,fill=black,minimum size = 1mm, inner sep = 1pt}}
    \draw[<->] (0.5,-7.5) -- (0.5,7.5);
    \draw[->] (0.5,0.0) -- (15.5,0.0);
    \node at (0.5,8) {$\eta(i)$};
    \node at (16.0,0.0) {$i$};
    \foreach \i in {1,...,14}
        {\node[b1] at (\i,7.5-0.5*\i) {};
         \node[b1] at (\i+1,-0.5*\i) {};
        }
    \foreach \i in {1,...,7}
        {\node[b1] at (\i,-0.5+0.5*\i) {};
         \node[b1] at (8+\i,-3.5+0.5*\i) {};
        }
    \draw[thick] (1,7) -- (15,0);
    \draw[thick] (1,0) -- (15,-7);
    \draw[dashed] (1,0) -- (8,3.5);
    \draw[thick] (8,3.5) -- (15,0);
    \draw[thick] (1,0) -- (8,-3.5);
    \draw[dashed] (8,-3.5) -- (15,0);
    \foreach \i in {1,...,15}
        {\node[below] at (\i,0) {$\i$};
        }
    \foreach \i in {-14,-13,-12,-11,-10,-9,-8,-7,-6,-5,-4,-3,-2,-1,0,1,2,3,4,5,6,7,8,9,10,11,12,13,14}
        {\node[left] at (0.3,0.5*\i) {$\i$};
        }     
    \node[above] at (3,6.2) {$\eta^{+}$};
    \node[above] at (6,2.7) {$\eta^{\wedge}$};
    \node[above] at (10,-2.3) {$\eta^{\vee}$};
    \node[above] at (13,-5.8) {$\eta^{-}$};
    \coordinate (A) at (0.6,0.0);
    \coordinate (B) at (15.2,-7.3);
    \coordinate (C) at (15.2,-0.4);
    \coordinate (D) at (14.4,-0.4);
    \coordinate (E) at (14.4,0.6);
    \coordinate (F) at (8.1,3.75);
    \draw [black!30,pattern=dots,pattern color=black!30] (A) -- (B) -- (C) -- (D) -- (E) -- (F) -- cycle;
    \node at (11,-3) {\Huge{$\RN{1}$}};
\end{tikzpicture}
}

\newcommand{\spconfigsnew}{\begin{tikzpicture}
\tikzset{
	b1/.style = {circle, draw=black,fill=black,minimum size = 1mm, inner sep = 1pt}}
    \draw[<->] (0.0,-3.25) -- (0.0,3.25);
    \draw[->] (0.0,0.0) -- (7.8,0.0);
    \node at (0,3.6) {$\eta(i)$};
    \node at (8.2,0.0) {$i$};
    \foreach \i in {1,...,7}
        {\node[b1] at (\i,3.5 -0.5*\i) {};
         \node[b1] at (\i,0.5 -0.5*\i) {};
        }
    \foreach \i in {1,2,3}
        {\node[b1] at (\i,0.5*\i-0.5) {};
         \node[b1] at (4+\i,-1.5+0.5*\i) {};
        }
    \draw[thick] (1,3) -- (7,0);
    \draw[thick] (1,0) -- (7,-3);
    \draw[dashed] (1,0) -- (4,1.5);
    \draw[dashed] (4,-1.5) -- (7,0);
    \foreach \i in {1,...,7}
        {\node[below] at (\i,0) {\tiny{$\i$}};
        }
    \node[left] at (0,0) {\tiny{$0$}};
    \foreach \i in {1,...,6}
        {\node[left] at (0,0.5*\i) {\tiny{$\i$}};
         \node[left] at (0,-0.5*\i) {\tiny{$-\i$}};
        }     
    \node[above] at (2.5,2.5) {\textbf{$\eta^{+}$}};
    \node[below] at (4,1.4) {$\mathbf{\eta^{\wedge}}$};
    \node[above] at (4,-1.4) {$\mathbf{\eta^{\vee}}$};
    \node[below] at (2.5,-1) {$\mathbf{\eta^{-}}$};
    \coordinate (A1) at (0.6,0.0);
    \coordinate (B1) at (7.2,-3.3);
    \coordinate (C1) at (7.2,-0.9);
    \coordinate (D1) at (3,1.2);
    \coordinate (A2) at (7.4,0.0);
    \coordinate (B2) at (0.8,3.3);
    \coordinate (C2) at (0.8,0.9);
    \coordinate (D2) at (5,-1.2);
    \draw [red!40] (A1) -- (B1) -- (C1) -- (D1) -- cycle;
    \draw [blue!40] (A2) -- (B2) -- (C2) -- (D2) -- cycle;
    \node at (6.5,-1.75) {\textcolor{red}{\huge{$\RN{2}$}}};
    \node at (1.5,1.75) {\textcolor{blue}{\huge{$\RN{1}$}}};
\end{tikzpicture}
}
%pattern=vertical lines,pattern color=red!10

\newcommand{\specconfigstwo}{\begin{tikzpicture}
\tikzset{
	b1/.style = {circle, draw=black,fill=white,minimum size = 0.5cm, inner sep = 1pt}}
    %\eta^{+} config
    \node at (0,4.5) {$\xi^{+}$};
    \node[b1] at (4,4.5) {$1$};
    \node[b1] at (5,4.5) {$2$};
    \node[b1] at (6,4.5) {$3$};
    \node[b1] at (7,4.5) {$4$};
    \node[b1] at (8,4.5) {$5$};
    \node[b1] at (9,4.5) {$6$};
    \node[b1] at (10,4.5) {$7$};
    %\eta^{\wedge} config
    \node at (0,3) {$\xi^{\wedge}$};
    \node[b1] at (1,3) {$1$};
    \node[b1] at (3,3) {$2$};
    \node[b1] at (5,3) {$3$};
    \node[b1] at (7,3) {$4$};
    \node[b1] at (8,3) {$5$};
    \node[b1] at (9,3) {$6$};
    \node[b1] at (10,3) {$7$};
    %\eta^{\vee} config
    \node at (0,1.5) {$\xi^{\vee}$};
    \node[b1] at (1,1.5) {$1$};
    \node[b1] at (2,1.5) {$2$};
    \node[b1] at (3,1.5) {$3$};
    \node[b1] at (4,1.5) {$4$};
    \node[b1] at (6,1.5) {$5$};
    \node[b1] at (8,1.5) {$6$};
    \node[b1] at (10,1.5) {$7$};
    %\eta^{-} config
    \node at (0,0) {$\xi^{-}$};
    \node[b1] at (1,0) {$1$};
    \node[b1] at (2,0) {$2$};
    \node[b1] at (3,0) {$3$};
    \node[b1] at (4,0) {$4$};
    \node[b1] at (5,0) {$5$};
    \node[b1] at (6,0) {$6$};
    \node[b1] at (7,0) {$7$};
    \foreach \i in {1,...,10}
        {\node[below] at (\i,-0.4) {$\i$};
         \foreach \j in {0,1.5,3,4.5}
            {\draw (\i-0.25,\j-0.3) -- (\i+0.25,\j-0.3);
            }
         
        }
	\end{tikzpicture}}

\newcommand{\ideakey}{\begin{tikzpicture}
    \tikzset{
        b1/.style = {circle, draw=black,fill=black,minimum size = 3mm, inner sep = 1pt},
        b2/.style = {circle, draw=black,fill=white,minimum size = 3mm, inner sep = 1pt}
    }
    
    \node at (0,0) {$\cdots$};
    \node at (11,0) {$\cdots$};
    \node at (0,2) {$\cdots$};
    \node at (11,2) {$\cdots$};

    \foreach \i/\style in {1/b1,2/b1,3/b2,4/b1,5/b1,6/b2,7/b1,8/b2,9/b1,10/b1} {
        \node[\style] at (\i,2) {};
    }
    
    \foreach \i/\style in {1/b1,2/b1,3/b1,4/b2,5/b1,6/b2,7/b1,8/b2,9/b1,10/b1} {
        \node[\style] at (\i,0) {};
    }

    \foreach \x/\y in {1.6/1.6, 4.6/1.6, 6.6/1.6, 2.6/-0.4, 4.6/-0.4, 6.6/-0.4} {
        \draw[red, dashed, thick] (\x,\y) rectangle ++(1.8,0.8);
    }

    \foreach \i in {1,...,10} {
        \draw (\i-0.25,-0.3) -- (\i+0.25,-0.3);
        \draw (\i-0.25,1.7) -- (\i+0.25,1.7);
    }
 
    \draw[-stealth,thick] (5.5,1.4) -- (5.5,0.6);

    \draw[-stealth,thick] (3.9,2.12) to[bend right=45] (3.1,2.12);

    \draw[-stealth,red,thick] (3,2.5) to[bend left=45] (4,2.5);
\end{tikzpicture}}

%%%%%%%%%%%%%%%%%%%%%%%%%%%%%%%%%%%%%%%%%%
%%%%%%%%%%%%%%%%%%%%%%%%%%%%%%%%%%%%%%%%%%

\section{Introduction}

%%%%%%%%%%%%%%%%%%%%%%%%%%%%%%%%%%%%%%%%%%
%%%%%%%%%%%%%%%%%%%%%%%%%%%%%%%%%%%%%%%%%%

The facilitated simple exclusion process (FEP) was originally introduced by Rossi, Pastor--Satorras and Vespignani \cite{uclassfield} as a one dimensional system with active--absorbing phase transition in the presence of a conserved field. 
It is a one--dimensional exclusion process (at most one particle per site) with a dynamical constraint that prevents a particle at site $x$ from jumping to site $x\pm1$ unless there is a particle at site $x\mp1$.
That is, a particle can only move if it has one empty neighbour and one occupied neighbour. 
The processes can exhibit frozen regions in which all the particles are separated by empty sites and, therefore, stuck until the first time their neighbour is occupied. 
If, on a finite system, the density of particles is smaller than $1/2$ then the process will eventually be absorbed in one of these frozen configurations.
On the other hand, at densities greater than $1/2$, the system will exhibit transient states but eventually evolve to reach an ergodic component of the state space.  
In this article, we focus on the case of densities greater than $1/2$ and investigate the time to reach equilibrium, including the time to escape from the transient component.

Since its introduction, the FEP has attracted significant interest in physics and mathematics literature. 
Gabel, Krapivsky and Redner \cite{PhysRevLett.105.210603} investigated the steady-state current and distribution of cluster sizes of the totally asymmetric FEP on the circle. 
Moreover, on the infinite line, suitable initial conditions give rise to a shock wave or a rarefaction wave with a jump discontinuity at the front.
Baik et al.\ \cite{Jinho2018} investigated statistics of particle positions in the totally asymmetric FEP (TAFEP) on the integers with step initial condition. 
The translation invariant stationary states of the TAFEP on the integers have been investigated by Goldstein, Lebowitz and Speer \cite{GLSexact,GLSdiscrete}, as well as the translation invariant states for the symmetric FEP (SFEP) on the integers with synchronous discrete time dynamics \cite{glsstat}.
The translation invariant stationary states of the asymmetric FEP (AFEP) have been examined by Ayyer et al.\ \cite{ayyer2020stationary}.
Invariant measures, in the case when the density is less than or equal to $1/2$, have also been examined in \cite{zhaochen}.
In the physics literature, the critical behaviour such as the critical exponents have been studied in \cite{PhysRevE.79.041143,PhysRevE.71.016112,PhysRevE.64.016123}. 

The macroscopic behaviour of the FEP, at large time and space scales, is  rich, since it can exhibit both active and inactive phases.
Recently, significant progress has been made on the hydrodynamics limit behaviour of the FEP in one dimension.
In particular, SFEP on the circle (with periodic boundary conditions) has been shown to satisfy a Stefan problem in the hydrodynamic limit \cite{blondel2020,Blondel_2021}.
It is shown that in super--critical regions the macroscopic density profile satisfies a diffusion equation with a dynamic free boundary, which invades the sub--critical regions, until one of the phases disappears.
The hydrodynamic limit for the process in contact with reservoirs has recently been considered in \cite{dacunha2024hydrodynamic}.
For the AFEP, a hyperbolic Stefan problem for the hydrodynamic limit has recently been derived in \cite{ESZ}.
Fluctuations have also been considered very recently.
Barraquand et al. \cite{barraquand2023weakly} investigate the fluctuations in the weakly asymmetric setting with step--like initial distributions.
The stationary macroscopic equilibrium fluctuations, in the symmetric and weakly asymmetric situations, have also been considered in \cite{EZ23}.

This article concerns the mixing time of the FEP, i.e. the speed of convergence to equilibrium started from `bad' initial configurations.
In particular, this includes the time to escape from transient states and, subsequently, to mix on an ergodic component.
We treat the one--dimensional system on a segment (with closed boundaries) and on the circle (with periodic boundaries).
On the segment we consider both the symmetric and asymmetric models, whereas on the circle we only consider the symmetric system.
In particular, in this work, we consider the case in which the dynamics restricted to the ergodic component are reversible with respect to the equilibrium distribution.
On the segment of $N$ sites, $[N]=\{1,\ldots,N\}$, with $k>N/2$ particles, we show that the mixing time of the SFEP is of order $N^{2}\log(N-k)$ and exhibits the pre--cutoff phenomenon (see Theorem \ref{th:fssepsegment}). 
In contrast, for the asymmetric FEP on the segment, we show that there
exists initial conditions for which the hitting time of the ergodic component is exponentially large
in the number of holes $N-k$. 
This phenomena turns out to be related to the reverse bias phase of the open boundary exclusion process for which the mixing time has been considered in \cite{gantert2020mixing}. 
In particular, when $N-k$ is large enough, the hitting time of the
ergodic component determines the mixing time, which is exponentially long in the system size (see Theorem \ref{th:fasepsegment}).
For the SFEP on the circle of size $N$ and macroscopic particle density $\rho \in(1/2,1)$,
we establish bounds on the mixing time of order $N^{2}\log N$ for the process restricted to its ergodic component.
We also give an upper bound on the hitting time
of the ergodic component of order $N^{2}\log N$ for a large class of initial conditions (see Proposition \ref{prop:hiterg}).

The proofs of our main results rely heavily on couplings with exclusion processes (both open and closed boundaries). 
One of the main novelties here is a new graphical construction that allows us to couple the dynamics of the FEP on the interval with both closed and open boundary exclusion models in terms of lattice paths (height functions) that are stochastically monotone (see Section \ref{sec:lattice}--\ref{sec:coupleSEP}).
%Similar lattice path dynamics have seen significant independent investigation and have many applications (see, for example, \cite{LacoinIsing}).
We believe that the construction here may have applications in the future for studying FEP models, for example, in establishing cutoff results. 
Once we have established this coupling, we appeal to recent results on the mixing time of exclusion processes (see, for example, \cite{gantert2020mixing,tran2023,salez2023}). 
For the time to escape the transient states, we use recent results of Gantert, Nestoridi and Schmid \cite{gantert2020mixing} on the mixing time for the open boundary exclusion processes (both symmetric and asymmetric). 
We also appeal to results for sharp mixing of the exclusion process, such as \cite{lacoinsegment} in the symmetric setting and \cite{cutoffasep} in the asymmetric case. 
%For the lower bound we use the Aldous and Brown \cite{aldousbrown} hitting time bound. 
For the lower bound on the circle we restrict to the ergodic component and map to an exclusion process on a smaller circle of size $k$. 
The lower bound for the FEP then follows by adapting the lower bound for the symmetric simple exclusion process (SSEP) from Morris \cite{Morris_2006} (see Section \ref{sec:cirlower}).  
In Section \ref{sec:lowersob}, we bound the log-Sobolev constant using a quasi-factorization result due to Cesi \cite{cesientropy}, and known bounds on the log-Sobolev constant for the SSEP \cite{yaulogsobolev}. 
This yields an upper bound of $N^2\log N$ restricted to the ergodic component.
Using an approach based on Yau's relative entropy method, bounds on the log--Sobolev constant have recently been used to get sharp results on convergence to equilibrium for the symmetric exclusion process in contact with reservoirs \cite{GoncalvesReservoir}. 
There is hope that similar methods could also apply for the FEP.
The bound for the hitting time from a class of initial is established using a mapping to a zero--range process (see, for example, \cite{gantert2020mixing}) and again a comparison with open boundary exclusion processes. 

%%%%%%%%%%%%%%%%%%%%%%%%%%%%%%%%%%%%%%%%%%
%%%%%%%%%%%%%%%%%%%%%%%%%%%%%%%%%%%%%%%%%%

\section{Notation and results}
\label{sec:notation}
%%%%%%%%%%%%%%%%%%%%%%%%%%%%%%%%%%%%%%%%%%
%%%%%%%%%%%%%%%%%%%%%%%%%%%%%%%%%%%%%%%%%%

In the dynamics of the FEP a particle at site $x$ attempts to jump right to site $x+1$ at rate $p$, only doing so if there is a hole at site $x+1$ and a particle at site $x-1$. 
Furthermore, a particle at site $x$ attempts to jump left to the site $x-1$ at rate $q$, only doing so if there is a hole at site $x-1$ and a particle at site $x+1$. 
See, for example, Figure \ref{fig:leftrightjump}.
\begin{figure}[tb]
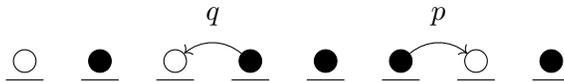

\centering
\segtransrates
\caption{Transition rates for the FEP on the segment. The indicated jumps are the only ones that may occur in the dynamics.}
\label{fig:leftrightjump}
\end{figure}
% This configuration is not ergodic as the leftmost endpoint is not occupied

For the FEP on the segment, we consider the state space
\begin{equation} \label{eq:segstate}
    \Omega_{N,k}=\Biggl\{\xi\in\{0,1\}^{N}\Bigm|\sum_{x=1}^{N}\xi(x)=k\Biggr\},
\end{equation}
i.e.\ the set of exclusion configurations on the segment of size $N$ with exactly $k$ particles. 
The generator $\mathcal{L}$ of FEP on the segment acts on functions $f:\{0,1\}^{N}\rightarrow\mathbb{R}$ via 
\begin{equation}
\begin{aligned} \label{eq:FEPgenseg}
    \mathcal{L}f(\xi)=&\sum_{x=2}^{N-1}p\xi(x)\xi(x-1)\bigl(1-\xi(x+1)\bigr)\bigl[f(\xi^{x,x+1})-f(\xi)\bigr]
    \\[1ex] &+\sum_{x=2}^{N-1}q\xi(x)\xi(x+1)\bigl(1-\xi(x-1)\bigr)\bigl[f(\xi^{x,x-1})-f(\xi)\bigr]\,,
\end{aligned}
\end{equation}
where $\xi^{x,y}$ denotes the configuration obtained from $\xi$ after interchanging the local configuration at $x$ and $y$, i.e. $\xi^{x,y}(z)=\xi(z)\1_{z\notin \{x,y\}}+\xi(y)\1_{z=x}+\xi(x)\1_{z=y}$.

As the number of particles is conserved in the FEP we may restrict the dynamics to the state space $\Omega_{N,k}$ in \eqref{eq:segstate}. 
Henceforth, we assume that $k>N/2$ so that the FEP on the segment (and circle) will always reach an ergodic component.

\begin{definition} \label{def:ergseg}
We define the \textit{ergodic component} $\mathcal{E}_{N,k}\subset\Omega_{N,k}$ for the FEP on the segment by
$$
\mathcal{E}_{N,k}=\bigl\{\xi\in\Omega_{N,k}: \forall\,x\in [N-1],\;\xi(x)+\xi(x+1)\geq 1 \textrm{ and } \xi(1)\xi(N)=1\bigr\}. 
$$
\end{definition}
In particular, on $\mathcal{E}_{N,k}$ the endpoints of the segment are occupied and no two holes are adjacent. 
\begin{figure}[tb]
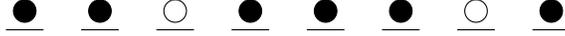

\centering
\ergodicfigex
\caption{An ergodic configuration in $\mathcal{E}_{8,6}$.}
\end{figure}

The FEP on the segment with parameter $p\in[1/2,1)$ is reversible and its equilibrium distribution $\mu_{N,k}$ is supported on the ergodic component. 
For $\xi\in\mathcal{E}_{N,k}$ and $i\in[N-k]$ we let $x_{i}$ denote the position of the $i$--th leftmost hole on the segment. 
If $p>1/2$, letting $\lambda=q/p$,  the detailed balance equations yield 
\begin{equation}
    \label{eq:int-stat}
    \mu_{N,k}(\xi)=\frac{\lambda^{-A(\xi)}}{\sum_{\xi'\in\mathcal{E}_{N,k}}\lambda^{-A(\xi')}}\,, \quad \textrm{where} \quad A(\xi)=\sum_{i=1}^{N-k}x_{i}\,.
\end{equation}
When $p=1/2$ then $\mu_{N,k}$ is the uniform distribution on $\mathcal{E}_{N,k}$.

Let $P_{t}^{\xi}$ denote the law of the FEP with initial condition $\xi\in\Omega_{N,k}$ at time $t$. 
Let
\begin{equation} \label{eq:mixdist}
d_{N,k}(t)=\max_{\xi\in\Omega_{N,k}}\bigl\lVert P_{t}^{\xi}(\cdot)-\mu_{N,k}(\cdot)\bigr\rVert_{\mathrm{TV}}, 
\end{equation}
denote the total variation distance between the law of the FEP at time $t$, started from the worst--case initial condition, and the stationary measure $\mu_{N,k}$. 
We denote the $\epsilon$--mixing time by
$$
T_{\mathrm{seg}}^{N,k}(\epsilon):=\inf\bigl\{t\geq0:d_{N,k}(t)\leq\epsilon\bigr\}.
$$ 

Let $\mathbb{T}_{N}:=\mathbb{Z}/N\mathbb{Z}$ denote the discrete circle of $N$ sites.
For the FEP on the circle, we consider the state space 
\begin{equation} \label{eq:cirstate}
\Omega_{N,k}^\circ:=\Biggl\{\xi\in\{0,1\}^{\mathbb{T}_{N}}\Bigm| \sum_{x\in\mathbb{T}_{N}}\xi(x)=k\Biggr\},
\end{equation}
i.e.\ the set of exclusion configurations on the circle of size $N$ with exactly $k$ particles. 
The generator $\mathcal{L}^{\circ}$ of the FEP on the circle acts on functions $f:\{0,1\}^{\mathbb{T}_{N}}\rightarrow\mathbb{R}$ via
\begin{equation}
\begin{aligned} \label{eq:genFEPcirc}
\mathcal{L}^{\circ}f(\xi)=&\sum_{x\in\mathbb{T}_{N}}\frac{1}{2}\xi(x)\xi(x-1)\bigl(1-\xi(x+1)\bigr)\bigl[f(\xi^{x,x+1})-f(\xi)\bigr] \\[1ex]
&+\sum_{x\in\mathbb{T}_{N}}\frac{1}{2}\xi(x)\xi(x+1)\bigl(1-\xi(x-1)\bigr)\bigl[f(\xi^{x,x-1})-f(\xi)\bigr]\,.
\end{aligned}
\end{equation}

\begin{definition} \label{def:ergcirc}
We define the \textit{ergodic component} $\mathcal{G}_{N,k}$ for the FEP on the circle by 
$$
\mathcal{G}_{N,k}=\bigl\{\xi\in\Omega_{N,k}^\circ:\forall \,x\in\mathbb{T}_{N}\,,\; \xi(x)+\xi(x+1)\geq1 \bigr\}\,.
$$
\end{definition}
In particular, on $\mathcal{G}_{N,k}$, no two holes are adjacent. 

\begin{figure}[tb]
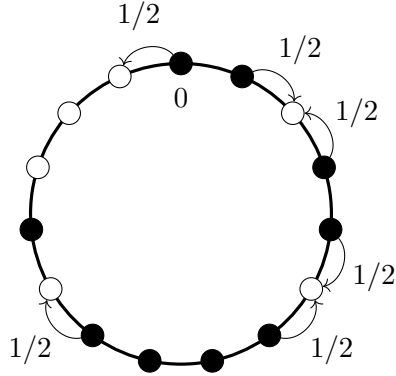

\centering
\cirtransrates
\caption{Transition rates for the SFEP on the circle. The indicated jumps are the only ones which may occur in the dynamics. The configuration in $\Omega_{15,9}^{\circ}$ does not belong to the ergodic component $\mathcal{G}_{N,k}$ as there are two adjacent holes.}
\label{fig:cirtransrates}
\end{figure}

We only consider the SFEP on the circle; the SFEP is reversible on the ergodic component and its equilibrium distribution $\nu_{N,k}$ is uniform on $\mathcal{G}_{N,k}$. 
The asymmetric FEP on the circle is not reversible, and the methods we use here do not generalise.

With an abuse on notation we let $P_{t}^{\xi}$ denote the law of the FEP on the circle with initial condition $\xi\in\Omega_{N,k}^\circ$, and let 
$$
d_{N,k}^{\circ}(t):=\max_{\xi\in\Omega_{N,k}^\circ}\bigl\lVert P_{t}^{\xi}-\nu_{N,k}\bigr\rVert_{\mathrm{TV}}.
$$
We denote the $\epsilon$--mixing time by 
$$
T_{\mathrm{cir}}^{N,k}(\epsilon):=\inf\bigl\{t\geq0:d_{N,k}^{\circ}(t)\leq\epsilon\bigr\}.
$$ 

%%%%%%%%%%%%%%%%%%%%%%%%%%%%%%%%%%%%%%%%%%
%%%%%%%%%%%%%%%%%%%%%%%%%%%%%%%%%%%%%%%%%%

\subsection{Results for the segment}

%%%%%%%%%%%%%%%%%%%%%%%%%%%%%%%%%%%%%%%%%%
%%%%%%%%%%%%%%%%%%%%%%%%%%%%%%%%%%%%%%%%%%

\begin{theorem}[SFEP] \label{th:fssepsegment} 
For the FEP on the segment with parameter $p=1/2$, let $k(N)$ be a sequence satisfying $N/2<k(N)<N$ such that $N-k$ and $k-N/2$ go to infinity. 
For all $\epsilon\in(0,1)$ there exists positive constants $0<C_{1}<C_{2}$ which do not depend on $\epsilon$ such that 
\begin{equation} \label{eq:fssep1}
C_{1}\leq\liminf_{N\rightarrow\infty}\frac{T^{N,k}_{\mathrm{seg}}(\epsilon)}{N^{2}\log (N-k)}\leq\limsup_{N\rightarrow\infty}\frac{T^{N,k}_{\mathrm{seg}}(\epsilon)}{N^{2}\log (N-k)}\leq C_{2}. 
\end{equation}
\end{theorem}
In particular the SFEP mixes on the order of $N^{2}\log(N-k)$. 
Moreover, pre--cutoff holds for the SFEP under the conditions of Theorem \ref{th:fssepsegment}. 
We note that Theorem \ref{th:fssepsegment} covers the regime where the density of particles $k/N\rightarrow\rho\in(1/2,1]$, and $k/N \searrow 1/2$, as $N\rightarrow\infty$. 
We conjecture that cutoff holds for the SFEP under the conditions of Theorem \ref{th:fssepsegment}, i.e.\ we may take $C_{1}=C_{2}$. 

As a consequence of our proof, we observe that cutoff of order $k^{2}\log\min(N-k,2k-N)$ holds for the SFEP restricted to the ergodic component due to cutoff for the SSEP on the segment \cite[Theorem 2.4]{lacoinsegment}. 
In particular, when $\log(2k-N)\ll\log(N-k)$, i.e.\ $\log(2k-N)/\log(N-k)\to 0$ as $N \to \infty$, the full mixing time is determined by the hitting time of the ergodic component. 

For the AFEP on the segment, we show that the mixing time is exponentially slow in the number of holes $N-k$, provided that the number of holes is large enough.  

\begin{theorem}[AFEP] \label{th:fasepsegment}
For the FEP on the segment with parameter $p\in(1/2,1)$, let $\epsilon\in(0,1)$ and $k(N)$ be a sequence such that $k>N/2$ and $N-k \gg \log N$, then the $\epsilon$--mixing time satisfies 
\begin{equation} \label{eq:fasep1}
\lim_{N\rightarrow\infty}\frac{\log T_{\mathrm{seg}}^{N,k}(\epsilon)}{N-k}=\log\Bigl(\frac{p}{q}\Bigr).
\end{equation}
\end{theorem}

Again, as a consequence of the proof, we observe that cutoff of order $N$ holds for the AFEP restricted to the ergodic component due to cutoff for the ASEP on the segment \cite[Theorem 2]{cutoffasep}.
Therefore, the AFEP on the segment mixes very rapidly, with order at most $N$, once it reaches the ergodic component, and the mixing time is dominated by the hitting time of the ergodic component. 

%%%%%%%%%%%%%%%%%%%%%%%%%%%%%%%%%%%%%%%%%%
%%%%%%%%%%%%%%%%%%%%%%%%%%%%%%%%%%%%%%%%%%

\subsection{Results for the circle}

%%%%%%%%%%%%%%%%%%%%%%%%%%%%%%%%%%%%%%%%%%
%%%%%%%%%%%%%%%%%%%%%%%%%%%%%%%%%%%%%%%%%%

Let $\alpha_{N,k}$ denote the log--Sobolev constant for the generator $\mathcal{L}^{\circ}$ restricted to the ergodic component $\mathcal{G}_{N,k}$.  
We briefly recall the definition of the log--Sobolev constant, see for example \cite{DSClogsob} for details. 
The entropy of a non--negative function $f\colon \Omega_{N,k}^\circ \to (0,\infty)$, with respect to the measure $\nu_{N,k}$ on $\mathcal{G}_{N,k}$, is defined by $\mathrm{Ent}_\nu(f) = \nu_{N,k}\left(f \log f\right) - \nu_{N,k}(f)\log\nu_{N,k}(f)$ and the Dirichlet form associated with $\mathcal{L}^\circ$ is given by
\begin{align}
\label{eq:dirichletform}
\mathcal{D}(f) = -\nu_{N,k}(f\mathcal{L}^\circ f) = \frac{1}{2}\sum_{\xi,\xi'\in\mathcal{G}}\nu_{N,k}(\xi)c(\xi,\xi')\bigl(f(\xi)-f(\xi')\bigr)^{2}\,.
\end{align}

The log--Sobolev constant, $\alpha$, is the largest constant  such that 
$$
\alpha\,\mathrm{Ent}_{\nu_{N,k}}(f^{2})\leq\mathcal{D}(f), \quad \textrm{for all}\quad f\colon \Omega_{N,k}^\circ \to \mathbb{R}\,.
$$

The following theorem gives a lower bound on the mixing time and the log--Sobolev constant for the FEP restricted to the ergodic component. 

\begin{theorem} \label{th:ergodicmixing}
Let $\epsilon\in(0,1)$ be given, and let $k(N)>N/2$ be a sequence such that $k(N)/N\rightarrow\rho\in(1/2,1)$ as $N\rightarrow\infty$. 
\begin{itemize}
\item[(a)]
There exists a constant, $C_\rho$, depending only on $\rho$ such that 
\begin{equation} \label{eq:lowercircle1}
\liminf_{N\to \infty} \frac{T^{N,k}_{\mathrm{cir}}(\epsilon)}{N^2\log N} \geq C_\rho\,.
\end{equation}
\item[(b)] 
For $N$ sufficiently large the log--Sobolev constant $\alpha_{N,k}$ for the FEP on $\mathcal{G}_{N,k}$ satisfies
\begin{equation} \label{eq:logsob1}
    \alpha_{N,k}\geq C'N^{-2},
\end{equation}
for some constant $C'>0$ which does not depend on $N$. 
\end{itemize}
\end{theorem}
The conditions on the sequence $k(N)$ are required in Theorem \ref{th:ergodicmixing}(b), but \eqref{eq:lowercircle1} may be shown under more general conditions. 

We note that Theorem \ref{th:ergodicmixing}(b) may be used to upper bound the mixing time restricted to the ergodic component by a standard bound of the mixing time in terms of the log--Sobolev constant. 
In particular, let 
\begin{equation} \label{eq:ergodicmix}
    T_{\mathcal{G}_{N,k}}(\epsilon)=\inf\Bigl\{t\geq0:\max_{\xi\in\mathcal{G}_{N,k}}\bigl\lVert P_{t}^{\xi}-\nu_{N,k}\bigr\rVert_{\mathrm{TV}}\leq\epsilon\Bigl\},
\end{equation}
for $\epsilon\in(0,1)$, then \cite[Corollary 2.2.7]{Saloff-Coste1997} (for example) yields
\begin{equation} \label{eq:uppercircle2}
T_{\mathcal{G}_{N,k}}(\epsilon)\leq\frac{\lceil\log_{e/2}(\epsilon^{-1})\rceil}{4\alpha_{N,k}}\Bigl(4+\log^{+}\log\bigl\lvert\mathcal{G}_{N,k}\bigr\rvert\Bigr),
\end{equation}
where $\log^{+}t=\max(0,\log t)$.

Using a straightforward counting argument, see for example \cite[Lemma 6.1]{blondel2020}, the size of the state space is given by
\begin{equation}\label{eq:circergodiccount}
\bigl\lvert\mathcal{G}_{N,k}\bigr\rvert=\binom{k}{N-k}+\binom{k-1}{N-k-1}=\frac{N}{k}\binom{k}{N-k}.
\end{equation}
Using the standard bounds on the binomial coefficient and  
the fact that $N/k\leq2$, it follows that 
\begin{equation} \label{eq:logsize}
\log^{+}\log\bigl\lvert\mathcal{G}_{N,k}\bigr\rvert \leq \log(N-k)+\log^{+}\log\Biggl(2^{1/(N-k)}\frac{e\cdot k}{N-k}\Biggr)=O(\log N),
\end{equation}
as $N\rightarrow\infty$ and $k/N \to \rho \in (1/2,1)$. 
It follows from Theorem \ref{th:ergodicmixing}(b), \eqref{eq:uppercircle2} and \eqref{eq:logsize} that 
\begin{equation} \label{eq:uppercircle3}
    T_{\mathcal{G}_{N,k}}(\epsilon)\leq C(\epsilon) \, N^{2}\log N,
\end{equation}
as $N\rightarrow\infty$ for some constant $C(\epsilon)>0$. 

Before stating the result for the hitting time of the ergodic component we introduce some notation. 
\begin{definition} \label{def:regions}
Let $\xi\in\mathcal{G}_{N,k}^{c}$ and define $[x,y]=\{x,x+1,\ldots,y\}\subset\mathbb{T}_{N}$ to be a (clockwise) interval on the circle. 
We say that $[x,y]$ is an \textit{ergodic interval} if $\xi(x)=1$, $\xi(y)=1$ and $\xi(z)+\xi(z+1)\geq 1$ for all $z\in[x,y]$ such that $z+1\in[x,y]$. 
In particular the singleton $\{x\}$ is an ergodic interval if $\xi(x)=1$.
We say that an ergodic interval $I$ for the configuration $\xi$ \textit{contains} $k_{I}$ particles if 
$$
\sum_{x\in I}\xi(x)=k_{I}.
$$

We define the set of \textit{ergodic regions} of $\xi$ to be the smallest set of ergodic intervals $\{I_{1},\ldots,I_{n}\}$ which contain all particles in $\xi$. 
In particular, any ergodic interval $[x,y]$ satisfies $[x,y]\subseteq I_{i}$ for some $i\in[n]$, and $\xi(x)=0$ for all $x\notin I_1\cup\cdots\cup I_n$.

We define $\mathcal{I}_{N,k}^{m}\subset\mathcal{G}_{N,k}^{c}$ to be the configurations for which there exists a set of at most $m$ ergodic regions containing at least $N-k$ particles collectively, i.e.\
\begin{align*}
\mathcal{I}_{N,k}^{m}=\Biggl\{\xi\in\mathcal{G}_{N,k}^{c}&:\mbox{ there exists a set of ergodic regions }\{I_{1},\ldots,I_{m'}\},\\[1ex]
&m'\leq m\mbox{ such that }\sum_{i=1}^{m'}\sum_{x\in I_{i}}\xi(x)\geq N-k\Biggr\}.
\end{align*}
\end{definition}
See Figure \ref{fig:regions1} for an example of ergodic regions.
\begin{figure}[tb]
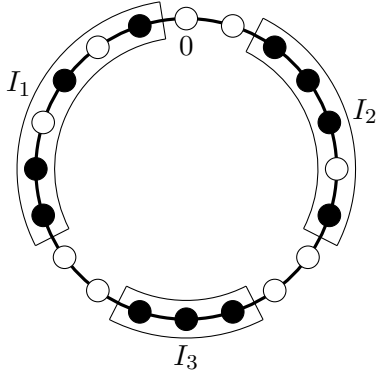

\begin{center}
\regions
\caption{(Ergodic regions) The figure above represents a configuration in $\mathcal{G}_{20,11}^{c}\subset\Omega_{20,11}^{\circ}$. Black circles denote particles and white circles denote holes. The ergodic regions from Definition \ref{def:regions} are given by the intervals $I_{1}=[14,19],I_{2}=[2,6]$ and $I_{3}=[9,11]$.}
\label{fig:regions1}
\end{center}
\end{figure}

Proposition \ref{prop:hiterg} states that the process started from initial configurations in $\mathcal{I}_{N,k}^{m}$ will hit the ergodic component with high probability after a time of order $N^2\log N$.
\begin{prop} \label{prop:hiterg}
Let $\epsilon>0$, $k(N)$ be a sequence, $m>0$ a positive integer that does not depend on $N$.
There exists a constant $C=C(m)>0$ (not depending on $\epsilon$) such that 
\begin{equation} \label{eq:regionhit1}
\max_{\xi\in\mathcal{I}_{N,k}^{m}}P^{\xi}_{C(m)N^{2}\log N}\bigl(\mathcal{G}_{N,k}^{c}\bigr)<\epsilon,
\end{equation}
for all $N$ sufficiently large depending on $\epsilon$.
\end{prop} 

\begin{remark} \label{rmk:circon}
Theorem \ref{th:ergodicmixing} together with Proposition \ref{prop:hiterg} gives that the mixing time started from any `reasonable' initial condition is $N^2\log N$, up to constant prefactors. 
The approach in Proposition \ref{prop:hiterg} is insufficient to obtain a general upper bound of order $N^2\log(N)$ on the hitting time of the ergodic component.
Controlling the hitting time of the ergodic component for general initial conditions is beyond the reach of the methods used here.
We conjecture that the initial conditions that maximise the hitting time of the ergodic component and control the mixing time are those in which there is one ergodic region of length $k$. 
These initial conditions are contained in the set $\mathcal{I}_{N,k}^{m}$ for each $m$. 
We also conjecture that the SFEP on the circle exhibits cutoff under suitable conditions on the initial configuration. \footnote{Since the article was first submitted substantial progress has been made: Erignoux and Massouli\'{e} \cite{ClementBrune24} give uniform estimates on the transience time of order $N^2\log N$ and show that the transience time exhibits cutoff. 
Moreover, Massouli\'{e} \cite{Brune24} shows cutoff and pre-cutoff for the FEP on the circle in different regimes, confirming this conjecture. }

\end{remark}

%%%%%%%%%%%%%%%%%%%%%%%%%%%%%%%%%%%%%%%%%%
%%%%%%%%%%%%%%%%%%%%%%%%%%%%%%%%%%%%%%%%%%

\section{FEP on the segment}
\label{sec:fepseg}
%%%%%%%%%%%%%%%%%%%%%%%%%%%%%%%%%%%%%%%%%%
%%%%%%%%%%%%%%%%%%%%%%%%%%%%%%%%%%%%%%%%%%

The intuitive idea we will use throughout the rest of the section is that particle--hole objects in the FEP move like extended exclusion objects of size two, see Figure \ref{fig:key}. 

\begin{figure}[tb]
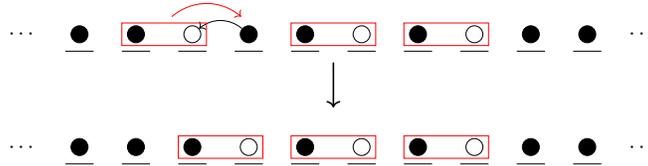

    \centering
    \scalebox{0.8}{\ideakey}
    \caption{In the above figure black circles denote particles and white circles denote holes. Particle--hole objects are placed in red boxes. In the dynamics, the third leftmost particle in the top figure jumps left, which corresponds to the leftmost particle--hole object jumping one space to the right. In particular, the boxes do not overlap in the dynamics and particle--hole objects behave like particles in the simple exclusion process.}
    \label{fig:key}
\end{figure}
This behaviour was observed in Gabel et al.\ \cite{PhysRevLett.105.210603}. It was exploited in
Baik et al.\ \cite{Jinho2018} to couple a TAFEP to a totally asymmetric simple exclusion process (TASEP) on both the half--line and $\mathbb{Z}$. 
Similarly, Ayyer et al.\ \cite[Theorem 4.3]{ayyer2020stationary} couple the AFEP with the ASEP on $\mathbb{Z}$ with this approach.
Exclusion models with particles of size greater than one have also attracted independent attention in the literature \cite{Alcarazexact,Alcarazanom,Lakatos_2003,macdonald69,Shaw2003}.

%%%%%%%%%%%%%%%%%%%%%%%%%%%%%%%%%%%%%%%%%%
%%%%%%%%%%%%%%%%%%%%%%%%%%%%%%%%%%%%%%%%%%

\subsection{Mapping to lattice paths} \label{sec:lattice}

%%%%%%%%%%%%%%%%%%%%%%%%%%%%%%%%%%%%%%%%%%
%%%%%%%%%%%%%%%%%%%%%%%%%%%%%%%%%%%%%%%%%%

In this section we introduce a mapping to lattice paths for the FEP on the segment. 
For any $\xi\in\Omega_{N,k}$ we label particles from left to right and let $x_{i}$ denote the position of the $i$--th leftmost particle for $i=1,\ldots,k$. 
We recursively define a path $\eta^{\xi}$ in $\mathbb{Z}^{2}$ as follows: set $\eta^{\xi}(1)=2(x_{1}-1)$ and for each $i\in[k-1]$
\begin{equation}\label{eq:lpath1}
\eta^{\xi}(i+1)-\eta^{\xi}(i) = 2( x_{i+1}-x_{i}-1)- 1\,,
\end{equation}
equivalently $\eta^{\xi}(i) = 2x_i -3i+1$.
See Figure \ref{fig:generalconfig} for an example of a lattice path. 

The mapping $\xi\mapsto\eta^{\xi}$ is injective and therefore bijective onto its image.  
We write $\overline{\Omega}_{N,k}$ for the image set of $\xi\mapsto\eta^{\xi}$. 
Similarly, we write $\overline{\mathcal{E}}_{N,k}$ for the image set of FEP configurations in the ergodic component.

For configurations $\xi\in\mathcal{E}_{N,k}^{c}$, the lattice path may include `steep' up--segments where $\eta^{\xi}(i+1)-\eta^{\xi}(i)>1$ for $i=1,\ldots,k-1$. See, for example, the segment between $\eta^{\xi}(7)$ and $\eta^{\xi}(8)$ in Figure \ref{fig:fullcouple}. Crucially, the dynamics cannot create steep segments, and steep segments cannot be made more steep. Steep segments are made less steep when the particles in the FEP make jumps that are not reversible, such as those indicated for particles $7$ and $8$ in Figure \ref{fig:fullcouple}. The left and right endpoints of the lattice path may only move down and up, respectively. Once the left endpoint reaches $0$ it cannot move any further, and once the right endpoint reaches $2N-3k+1$ it cannot move any further. This is because the particles $1$ and $k$ in the FEP move left and right, respectively, until they become stuck at the endpoints.

\begin{figure}[tb]
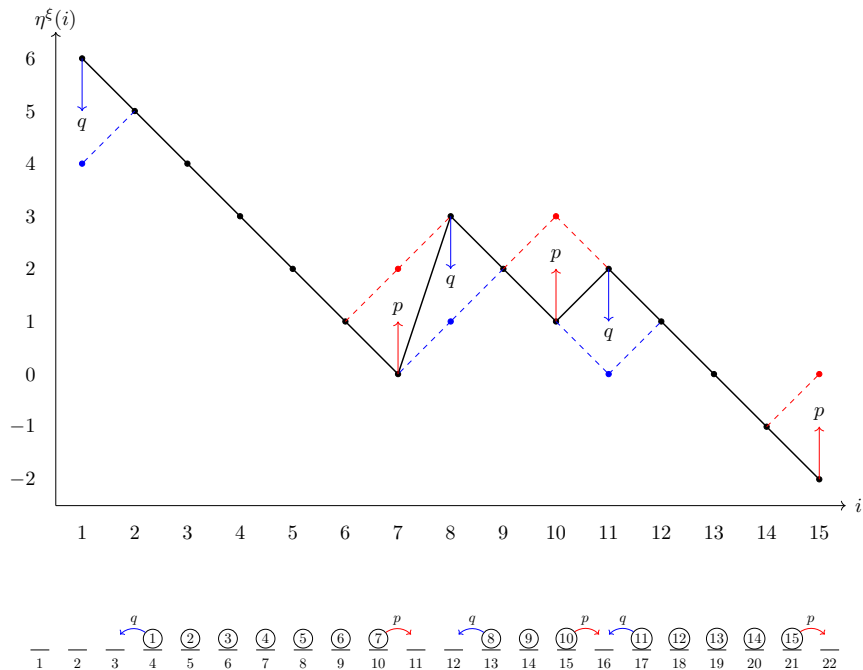

\centering
{\scalebox{0.7}{\fullcouple}} \\[5ex]
{\scalebox{0.65}{\gconfig}}
\caption{\label{fig:fullcouple} Top: The black path gives the lattice path $\eta^{\xi}$ for the configuration $\xi\in\Omega_{22,15}$ drawn below. The red and blue arrows give the only allowed transitions in the lattice path dynamics. These correspond to the clock rings (left to right) $\mathcal{T}_{(1,6)}^{\downarrow}$, $\mathcal{T}_{(7,0)}^{\uparrow}$, $\mathcal{T}_{(8,3)}^{\downarrow}$, $\mathcal{T}_{(10,1)}^{\uparrow}$, $\mathcal{T}_{(11,2)}^{\downarrow}$ and $\mathcal{T}_{(15,-2)}^{\uparrow}$ respectively in the graphical construction. The dashed lines indicate the changes to the lattice path in the dynamics. Bottom: A configuration $\xi\in\Omega_{22,15}$. The $i$--th particle from the left is labelled $i$, and the sites are labelled below. The red and blue arrows indicate all the possible transitions in the FEP dynamics for $\xi$ .}
\label{fig:generalconfig}
\end{figure}

We now introduce a graphical construction for the dynamics on lattice paths. 
This construction is similar to the construction found in Lacoin \cite[Section 8.1]{lacoinsegment}. 
For all $N\geq1$ and $k\geq N/2$ this graphical construction allows us to couple the trajectories $(\eta_{t}^{\xi})_{t\geq0}$ starting from all initial configurations $\xi\in\Omega_{N,k}$, and makes explicit the connection with the exclusion processes with various boundary conditions. Importantly, the construction on lattice paths is also monotone, conserving the natural partial order (discussed further in the following section).

To each site $(i,y)\in \mathbb{Z}_+ \times \mathbb{Z}$ we attach independent Poisson processes (clocks) $\mathcal{T}_{(i,y)}^{\uparrow}$ and $\mathcal{T}_{(i,y)}^{\downarrow}$ of rates $p$ and $q$, respectively.  
From these, and for $\xi\in\Omega_{N,k}$, we deterministically construct a trajectory $(\eta_{t}^{\xi})_{t\geq 0}$ with initial condition $\eta^{\xi}$. Then, $(\eta_{t}^{\xi})_{t\geq 0}$ is the unique, right-continuous, $\overline\Omega_{N,k}$ valued, function which equals $\eta^\xi$ at time zero, is constant outside of $\bigcup_{(i,y)\in\mathbb{Z}_+ \times \mathbb{Z}}\mathcal{T}_{(i,y)}^{\uparrow} \cup \mathcal{T}_{(i,y)}^{\downarrow}$ and evolves according to the following rules.

\noindent If $\mathcal{T}_{(i,y)}^{\uparrow}$ rings at time $t$, then:
\begin{itemize}
    \item if $1<i<k$, $\eta_{t-}^{\xi}(i)=y$ and $\eta_{t-}^{\xi}$ has a local minimum at $x$ then $\eta_{t}^{\xi}(i)=y+2$ and all other coordinates of the lattice path remain unchanged; 
    \item if $i=k$, $y<2N-3k+1$, $\eta_{t-}^{\xi}(i)=y$ and $\eta_{t-}^{\xi}$ has a local minimum at $i$ then $\eta_{t}^{\xi}(x)=y+2$ and all other coordinates of the lattice path remain unchanged; 
    \item otherwise no changes are made to the lattice path at time $t$. 
\end{itemize}
If $\mathcal{T}_{(i,y)}^{\downarrow}$ rings at time $t$, then
\begin{itemize}
    \item if $1<i<k$, $\eta_{t-}^{\xi}(i)=y$ and $\eta_{t-}^{\xi}$ has a local maximum at $i$ then $\eta_{t}^{\xi}(i)=y-2$ and all other coordinates of the lattice path remain unchanged; 
    \item if $i=1$, $y>0$, $\eta_{t-}^{\xi}(i)=y$ and $\eta_{t-}^{\xi}$ has a local maximum at $i$ then $\eta_{t}^{\xi}(i)=y-2$ and all other coordinates of the lattice path remain unchanged; 
    \item otherwise no changes are made to the lattice path at time $t$. 
\end{itemize}
We will denote the law of this construction by $\mathbb{P}$. 
To see that this construction gives the Markov chain with the generator in \eqref{eq:FEPgenseg}, observe that a local maximum in the lattice path will occur at $i\in\{2,\ldots,k-2\}$ if and only if the positions of the particles satisfy $x_{i+1}=x_{i}+1$ and $x_{i-1}<x_{i}-1$. 
Similarly, a local minimum will occur at $i\in\{2,\ldots,k-2\}$ if and only if $x_{i-1}=x_{i}-1$ and $x_{i+1}>x_{i}+1$. 
A local maximum occurs at $i=1$ and $\eta^{\xi}(1)>0$ if and only if $x_{1}>1$ and $x_{2}=x_{1}+1$. A local minimum occurs at $i=k$ and $\eta^{\xi}(k)<2N-3k+1$ if and only if $x_{k}<N$ and $x_{k-1}+1=x_{k}<N$.

Observe that, for the FEP lattice path dynamics on $\overline\Omega_{N,k}$, the left and right endpoints are only allowed to move downward and upward, respectively, in the dynamics. 
Also, the left boundary is at least zero, and the right boundary is at most $2N - 3k + 1$. 
This is due to irreversible jumps in the dynamics towards the boundaries of the segment, see Figure \ref{fig:generalconfig}. 
Also, we observe that 
$$
\overline{\mathcal{E}}_{N,k} = \{\eta \colon \eta(1) =0,\ \eta(k)=2N-3k+1, \eta(x+1)-\eta(x)\in\{-1,1\}\, \forall x\in [k-1]\}
$$
Moreover, the above construction allows us to couple the dynamics of the FEP on segments of different sizes and with different numbers of particles.  

\begin{remark} \label{rmk:FEPandsep}
By comparing the graphical construction $\mathbb{P}$ to the graphical construction given in \cite[Section 8.1]{lacoinsegment}, we observe that the FEP with parameter $p$ restricted to the ergodic component $\mathcal{E}_{N,k}$, i.e.\ the lattice paths with no `steep' segments, is equivalent to a SEP with right jump rate $q$ and left jump rate $p$ on the segment $[k-1]$ with $N-k$ particles. 
The explicit construction of the SEP, $(\sigma_t)_{t\geq 0}$, on $\Omega_{k-1,N-k}$ from a height function $(\eta_t)_{t\geq 0}$ on $\overline{\mathcal{E}}_{N,k}$ is given by
\begin{align*}
    \sigma_t(x) = \frac{1}{2}( \eta_t(x+1)-\eta_t(x) +1 ) \,,
\end{align*}
for $x \in [k-1]$.
\end{remark}

%%%%%%%%%%%%%%%%%%%%%%%%%%%%%%%%%%%%%%%%%%
%%%%%%%%%%%%%%%%%%%%%%%%%%%%%%%%%%%%%%%%%%

\subsection{Monotonicity}

%%%%%%%%%%%%%%%%%%%%%%%%%%%%%%%%%%%%%%%%%%
%%%%%%%%%%%%%%%%%%%%%%%%%%%%%%%%%%%%%%%%%%

In this section we show that the graphical construction given in Section \ref{sec:lattice} preserves a partial order for the lattice paths.
The lattice paths admit a natural partial order; for two lattice paths $\eta,\eta'\in \overline{\Omega}_{N,k}$ we say that $\eta\leq\eta'$ if
\begin{equation} \label{eq:partial1}
    \eta(i)\leq\eta'(i)\mbox{ for all }i\in[k]. 
\end{equation}
Let $\eta=\eta^{\xi_{1}}$ and $\eta'=\eta^{\xi_{2}}$ for two particle configurations $\xi_{1},\xi_{2}\in\Omega_{N,k}$, and let $x^{1}_{i}$ and $x^{2}_{i}$ denote the positions of the $i^\mathrm{th}$ leftmost particles in $\xi_{1}$ and $\xi_{2}$ respectively. 
It is straightforward to see from \eqref{eq:lpath1} that
\eqref{eq:partial1} is equivalent to $x_{i}^{1}\leq x_{i}^{2}$ for all $i\in[k]$.   

The following proposition states that the partial order \eqref{eq:partial1} is conserved under the coupling $\mathbb{P}$, i.e.\ the lattice path dynamics are monotone. 
\begin{prop} \label{prop:partial}
Let $\eta,\eta'\in\overline{\Omega}_{N,k}$ be two lattice paths with $\eta\leq\eta'$ then the lattice path trajectories $(\eta_{t})_{t\geq0}$ and $(\eta_{t}')_{t\geq0}$ satisfy
\begin{equation} \label{eq:partial2}
    \mathbb{P}\bigl(\eta_{t}\leq\eta_{t}'\bigr)=1,
\end{equation}
for all $t\geq0$.
\end{prop}

The proof of Proposition \ref{prop:partial} is similar to the proof of \cite[Proposition 3.1]{lacoinsegment}.

\begin{proof}[Proof of Proposition \ref{prop:partial}]
Fix $\eta,\eta'\in\overline{\Omega}_{N,k}$ to be two lattice paths with $\eta\leq\eta'$.
It is sufficient to check that the partial order is conserved on the clock rings in the graphical construction. 
Suppose that $\eta_{t-}\leq\eta_{t-}'$ and there is a clock ring in the graphical construction at time $t$. 
Assume, wlog, that $\mathcal{T}^{\uparrow}_{(i,y)}$ rings at time $t$ for some $2\leq i\leq k$ and $y \in \mathbb{Z}$ (a symmetric argument applies for $t$ in $\mathcal{T}^{\downarrow}_{(i,y)}$). 
%If $y \not\in\{ \eta_{t-}(i),\eta_{t-}'(i)\}$ then both configurations are unchanged at $t$.
%If $y = \eta_{t-}'(i) > \eta_{t-}$, and the height of the path can only increase at times in $\mathcal{T}^{\uparrow}_{(i,y)}$, the order is conserved. 
%It remains to check the case $y=\eta_{t-}(i)$.

If $\eta_{t-}(i)$ is not a local minimum of $\eta_{t-}$, then no change is made to $\eta_{t}$ and the partial order must hold at time $t$, since $T^{\uparrow}_{(i,y)}$ cannot make coordinates of the lattice path $\eta_{t-}'$ smaller. Henceforth we assume that $\eta_{t-}(i)$ is a local minimum.

If $\eta_{t-}(i)\leq\eta_{t-}'(i)-2$ then $\eta_{t}(i)\leq\eta_{t}'(i)$ since the clock ring $T^{\uparrow}_{(i,y)}$ can increase the height in both lattice paths by at most $2$. Hence, $\eta_{t}\leq\eta_{t}'$.

Since the $i^\mathrm{th}$ coordinates in both lattice paths have the same odd/even parity, the only remaining case is $\eta_{t-}(i)=\eta_{t-}'(i)$, and the lattice paths can only change if $\eta_{t-}(i)=\eta_{t-}'(i)=y$. 
By assumption $\eta_{t-}(i)=y$ is a local minimum and we have $\eta_{t-}(i-1)=y+1$. 
Moreover, it holds that $\eta_{t-}(i+1)>y$. By the partial order at time $t-$ we have that $\eta_{t-}'(i-1)\geq\eta_{t-}(i-1)$ and $\eta_{t-}'(i+1)\geq\eta_{t-}(i+1)$.
Therefore, $\eta_{t-}'(i)$ must also be a local minimum, and hence $\eta_{t}(i)=\eta_{t}'(i)=y+2$. 
No other coordinates are changed in the lattice paths, and \eqref{eq:partial2} holds at time $t$. 
\end{proof}

We now define some special configurations which are maximal and minimal with respect to the partial order, firstly on the full state space and secondly on the ergodic component.
See Figure \ref{fig:special} for an example of the lattice paths from Definition \ref{def:specialpaths}.

\begin{definition} \label{def:specialpaths}
Let $\eta^{-}$ and $\eta^{+}$ denote the minimal and maximal lattice paths on $\overline{\Omega}_{N,k}$ w.r.t.\ the partial order \eqref{eq:partial1}. 
The minimal and maximal lattice paths correspond to FEP configurations, on $\Omega_{N,k}$, given by
\begin{align} \label{eq:minmaxheight}
\xi^{-}(x)=\mathds{1}(x\leq k), \quad \mbox{and} \quad
\xi^{+}(x)=\mathds{1}(x\geq N-k+1),
\end{align}
respectively.

Let $\eta^{\vee}$ and $\eta^{\wedge}$ denote the minimal and maximal configurations w.r.t.\ the partial order restricted to the ergodic component $\overline{\mathcal{E}}_{N,k}$\footnote{This notation is chosen to be consistent with \cite{lacoinsegment}, i.e.\  
the notation ${\vee}$ and ${\wedge}$ is chosen to represent the shapes of the corresponding lattice paths.}.
The lattice paths $\eta^{\vee}$ and $\eta^{\wedge}$ correspond to the FEP configurations
$$
\xi^{\wedge}(x)=\mathds{1}\bigl(x\notin2\mathbb{Z}\cap[1,2N-2k]\bigr)\quad \mbox{and} \quad
\xi^{\vee}(x)=\xi^{\wedge}(N+1-x),
$$
respectively.  
\end{definition}

\begin{figure}[tbh]
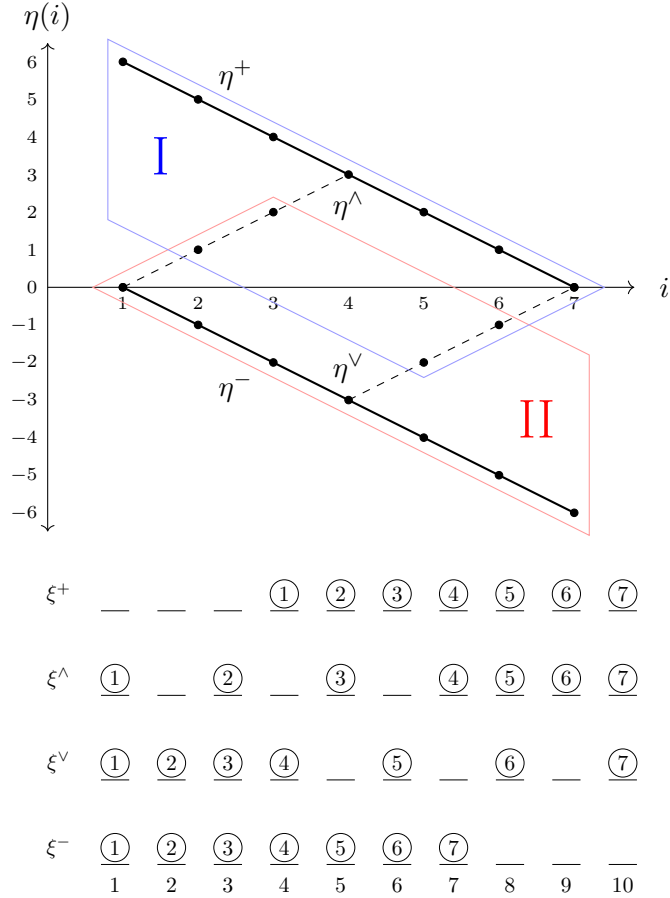

\centering
\spconfigsnew\\[3ex]
\scalebox{0.75}{\specconfigstwo}
\caption{\label{fig:special} Top: The lattice paths $\eta^{+},\eta^{-},\eta^{\wedge}$ and $\eta^{\vee}$ on $\overline{\Omega}_{10,7}$. The upper (resp. lower) dashed line indicates the part of the lattice path $\eta^{\wedge}$ (resp. $\eta^{\vee}$) not contained in $\eta^{+}$ (resp. $\eta^{-}$). The region $\RN{1}$ (resp. $\RN{2}$) bounded by the blue (resp. red) quadrilateral contains all lattice paths $\eta\in\overline{\Omega}_{N,k}$ with $\eta(1)>0$ (resp. $\eta(k)<2N-3k+1$) which may be reached from $\eta^{+}$ (resp. $\eta^{-}$) in the dynamics. On this region the FEP are coupled with the OBEP height function, see Remark \ref{rmk:OBEPisFEP}. Bottom: The corresponding configurations $\xi^{+},\xi^{\wedge},\xi^{\vee}$ and $\xi^{-}$ on the state space $\Omega_{10,7}$ with particles labelled $i=1,\ldots,7$ from left to right. }
\end{figure}

%%%%%%%%%%%%%%%%%%%%%%%%%%%%%%%%%%%%%%%%%%
%%%%%%%%%%%%%%%%%%%%%%%%%%%%%%%%%%%%%%%%%%

\subsection{Coupling with exclusion processes}
\label{sec:coupleSEP}
%%%%%%%%%%%%%%%%%%%%%%%%%%%%%%%%%%%%%%%%%%
%%%%%%%%%%%%%%%%%%%%%%%%%%%%%%%%%%%%%%%%%%

In this section we observe how the FEP may be coupled to the SEP on the segment and the open boundary exclusion process (OBEP).
The couplings are on the probability space defining the graphical construction of the FEP in Section \ref{sec:lattice}.

We adopt the same notation as in Section \ref{sec:notation}. 
We consider the SEP on the interval $[n]=\{1,2,\ldots,n\}$, with $m \leq n$ particles and state space $\Omega_{n,m}$. 
In the SEP, the extra constraint of needing a particle `behind' any particle that moves is absent. 
A particle at site $x$ attempts to jump right to site $x+1$ at rate $q\in(0,1)$, only doing so if $x<n$ and $x+1$ is not occupied by another particle. 
A particle at site $x$ attempts to jump left at rate $p=1-q$, only doing so if $x>1$ and $x-1$ is not occupied by another particle (note that the roles of $p$ and $q$ have swapped with respect to the FEP due to the coupling discussed previously). 
The generator of the SEP acts on observables by
\begin{align}
\begin{split} \label{eq:Lsep}
    \mathcal{L_\mathrm{ex}}f(\xi)=&\sum_{x=1}^{n-1}q\xi(x)\bigl(1-\xi(x+1)\bigr)\bigl[f(\xi^{x,x+1})-f(\xi)\bigr]  \\
    &+\sum_{x=2}^{n}p\xi(x)\bigl(1-\xi(x-1)\bigr)\bigl[f(\xi^{x,x-1})-f(\xi)\bigr].
\end{split}
\end{align}

For the OBEP, we allow for particle creation and annihilation at two endpoints of the segment. 
In this case the particle number is no longer necessarily conserved, and we consider the process on state space 
$$
\Omega_{n}:=\{0,1\}^{n}.
$$
In the OBEP, we interpret the endpoints of the segment as being attached to infinite reservoirs that input and remove particles. 
For boundary rates $\alpha,\beta,\gamma,\delta \geq 0$, we think of site $1$ as being attached to an infinite reservoir of particles that attempts to enter particles into site $1$ with rate $\alpha\geq0$, only doing so if there is a hole at site $1$, and attempts to remove particles from site $1$ with rate $\gamma\geq0$, only doing so if there is a particle at site $1$. 
Similarly, the site $n$ is attached to an infinite reservoir which inputs and removes particles at the site $n$ with rates $\delta\geq0$ and $\beta\geq0$ respectively. 
The OBEP with parameters $(q,\alpha,\beta,\gamma,\delta)$ is the process generated by
\begin{align}
\begin{split} \label{eq:Lobep}
    \mathcal{L_\mathrm{oex}}f(\xi)=& \mathcal{L_\mathrm{ex}}f(\xi) +\Big(\alpha\big(1-\xi(1)\big) + \gamma\xi(1)\Big)\Big(f(\xi^1)-f(\xi)\Big) \\
    &+\Big(\delta\big(1-\xi(n)\big) + \beta\xi(n)\Big)\Big(f(\xi^n)-f(\xi)\Big)\,,
\end{split}
\end{align}
where $\xi^x$ denotes the configuration obtained from $\xi$ by flipping the occupation value only at site $x$. See Figure \ref{fig:obepagain}.
Notice that the OBEP generalises the SEP; the latter can be recovered by taking $\alpha=\beta=\gamma=\delta=0$.

\begin{figure}[tb]
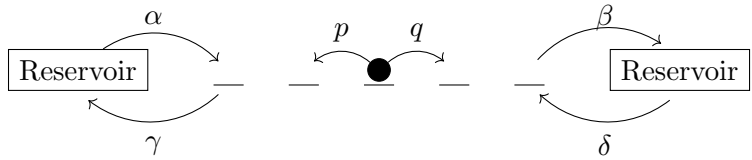

\centering
\obep
\caption{Transition rates for the $(q,\alpha,\beta,\gamma,\delta)$ OBEP on $\Omega_{5}$.}
\label{fig:obepagain}
\end{figure}

\subsubsection*{Exclusion process graphical construction}
We now give a graphical construction of the OBEP on $\Omega_{n}$ with parameters $(q,0,0,0,p)$, i.e. with closed left boundary and particle injection with rate $p$ on the right boundary, in terms of lattice paths. 
For each $\zeta\in\Omega_{n}$, let $h^{\zeta}(1)=0$ and for each $x\in[n]$ set 
\begin{align*}
h^{\zeta}(x+1) - h^{\zeta}(x) = 2\zeta(x)-1 = \begin{cases}
    +1 & \textrm{if $\zeta(x) =1$},\\
    -1 &  \textrm{if $\zeta(x) = 0$}\,.
\end{cases}
\end{align*}
Recall from Section \ref{sec:lattice},
at each site $(i,y)\in \mathbb{Z}_+ \times \mathbb{Z}$ we attach independent Poisson clocks $\mathcal{T}_{(i,y)}^{\uparrow}$ and $\mathcal{T}_{(i,y)}^{\downarrow}$ of rates $p$ and $q$, respectively.  
For $i \in \{2,\ldots,n+1\}$, if $\mathcal{T}_{(i,y)}^{\uparrow}$ rings at time $t$, then 
\begin{itemize}
    \item if $h^{\zeta}_{t-}(i)=y$ and $h^{\zeta}_{t-}$ has a local minimum at $i$ then $h^{\zeta}_{t}(i)=y+2$ and all other coordinates remain unchanged; 
    \item otherwise no changes are made to the lattice path at time $t$. 
\end{itemize}
If $\mathcal{T}_{(i,y)}^{\downarrow}$ rings at time $t$, then 
\begin{itemize}
    \item if $h^{\zeta}_{t-}(i)=y$ and $h^{\zeta}_{t-}$ has a local maximum at $i$ then $h^{\zeta}_{t}(i)=y-2$ and all other coordinates remain unchanged; 
    \item otherwise no changes are made to the lattice path at time $t$. 
\end{itemize}
Recall that we denote the law of the graphical construction by $\mathbb{P}$. 

It is straightforward to check that this gives a construction of the dynamics of the OBEP with parameters $(q,0,0,0,p)$.

\begin{remark} \label{rmk:OBEPisFEP}
By the above graphical construction, we observe that the FEP $(\eta_{t}^{-})_{t\geq0}$, on $\Omega_{N,k}$, is equivalent to an OBEP on $\Omega_{k-1}$ with parameters $(q,0,0,0,p)$ and an empty initial condition, until $N-k$ particles have entered at the right boundary of the OBEP (until the first time $\eta_{t}^{-}(k) = 2N-3k+1$). In particular, they can be coupled using clock rings in region $\RN{2}$ of Figure \ref{fig:special}.
That is, $(\eta_{t}^{-})_{t\geq0}$ is equivalent to an OBEP with a finite reservoir containing $N-k$ particles attached to the site $k-1$. 
After this time, the lattice path dynamics of the FEP are equivalent to a closed exclusion process, i.e. with $\alpha=\beta=\gamma=\delta=0$, see Remark \ref{rmk:FEPandsep}. 

Similarly, by interpreting `down-slopes' as holes and `up-slopes' as particles, the FEP $(\eta_{t}^{+})_{t\geq0}$ is equivalent to an OBEP on $\Omega_{k-1}$ with parameters $(q,q,0,0,0)$ and the empty initial condition until the first time $N-k$ particles have entered the left boundary. 
\end{remark}

In light of the previous remark we introduce notation for the first time that $\ell$ particles have entered the  OBEP on $\Omega_{k-1}$ with empty initial condition and parameters $(q,0,0,0,p)$,
\begin{align}
\label{eq:obepentrytime}
\tau_{\ell}=\inf\bigl\{t\geq0:h_{t}^{\boldsymbol{0}}(k) \geq 1-k + 2\ell \bigr\},
\end{align}
note that particles entered the OBEP with $(q,0,0,0,p)$ at the right boundary only.
In particular, by the Remark \ref{rmk:OBEPisFEP},
\begin{equation} \label{eq:rightboundaryarrive}
\mathbb{P}(\eta_{t}^{-}\notin\overline{\mathcal{E}}_{N,k})=\mathbb{P}(\tau_{N-k}>t).
\end{equation}

\subsection{Mapping to zero--range process}
\label{sec:coupleZRP}

The following mappings of the FEP dynamics to the zero-range process (ZRP) have been used several times in the literature, see for example \cite{PhysRevE.79.041143,blondel2020,ESZ} and references therein.

To each FEP configuration $\xi\in\Omega_{N,k}$ we associate a zero-range configuration $\Pi[\xi]\in \mathbb{N}_0^{N-k+1}$ as follows. 
Label the positions of the holes (empty sites) in $\xi$ by $0=y(0)<y(1)<y(2)<\ldots<y(N-k)<y(N-k+1)=N+1$, where $y(1)$ is the first empty site to the right of site $1$ (included). Then for $i \in [N-k+1]$,
\begin{align*}
\Pi[\xi](i) = y(i)-y(i-1) - 1\,,
\end{align*}
i.e. the number of particles at site $i$ of the configuration $\Pi[\xi]$ is equal to the number of particles between the $(i-1)^{\textrm{th}}$ and $i^{\textrm{th}}$ empty site in $\xi$ (where we consider the left and right boundary to contain empty sites).
It is straightforward to check that $(\omega_t^\xi)_{t\geq0} = (\Pi[\xi_t])_{t\geq 0}$ defines a Markov process on $\mathbb{N}_0^{N-k+1}$ with generator acting on test functions $f$ by
\begin{align}
\begin{split} \label{eq:Lzrp}
    \mathcal{L_\mathrm{ZR}}f(\omega)=&\sum_{i=1}^{N-k}p\1_{\{\omega(i)\geq 2\}}\bigl[f(\omega^{i,i+1})-f(\omega)\bigr]  +\sum_{x=2}^{N-k+1}q\1_{\{\omega(i)\geq 2\}}\bigl[f(\omega^{i,i-1})-f(\omega)\bigr]\,,
\end{split}
\end{align}
where
$$
\omega^{i,j}(z)=
\begin{cases}
    \omega(z)-1&\mbox{ if }z=i,\\
    \omega(z)+1&\mbox{ if }z=j,\\
    \omega(z)&\mbox{ otherwise}.
\end{cases}
$$
For the FEP on the closed segment, the function $\Pi$ is injective, and the ZRP `picture' is simply another interpretation of an equivalent Markov process. 
This is not the case on the circle, but a similar and still useful correspondence still holds.

For the FEP on a circle, see \eqref{eq:genFEPcirc}, we consider a similar mapping as above, however, we now define the ZRP with respect to the location of a tagged hole in the FEP.
For an FEP configuration, $\xi \in \Omega_{N,k}^\circ$, and a tagged hole at site $y(0)$ in $\xi$ we label the position of the remaining holes in clockwise order $y(1),y(2),\ldots,y(N-k-1)$.
This defines a unique ZRP configuration $\Pi^{\circ}[\xi,y(0)] \in \mathbb{N}_0^{\mathbb{T}_{N-k}}$ through
\begin{align*}
\Pi^{\circ}[\xi,y(0)](i) = y(i+1)-y(i) - 1 \ (\mathrm{mod}\ N-k)\,,
\end{align*}

For the dynamics of the corresponding ZRP, given an initial condition of the FEP, $\xi \in \Omega_{N,k}^\circ$, we tag the first hole to the right of the site $0$ and call the location of this empty site $y(0)$. 
The positions of the remaining holes are labeled in clockwise order $y(0) < y(1)<\ldots<y(N-k-1)$.
The position of the tagged hole under the dynamics $(\xi_t)_{t\geq 0}$ is given by $y_t(0)$ and the remaining holes maintain their label in clockwise order. Note that, for $t>0$, $y_t(0)$ is not necessarily the first hole to the right of $0$.
It is straightforward to check that $(\omega_t^\xi)_{t\geq0}$ given by
\begin{align*}
\omega^\xi_t =\Pi^{\circ}[\xi_t,y_t(0)]\,,
\end{align*}
 defines a Markov process on $\mathbb{N}_0^{\mathbb{T}_{N-k}}$ with generator acting on test functions $f$ as
\begin{equation} \label{eq:genzc}
\mathcal{L}^{\circ}_{\mathrm{ZR}} f(\omega)=\sum_{i\in\mathbb{T}_{n}}\Big(p\mathds{1}_{\{\omega(i)\geq2\}}\bigl[f(\omega^{i,i+1})-f(\omega)\bigr] + q\mathds{1}_{\{\omega(i)\geq2\}}\bigl[f(\omega^{i,i-1})-f(\omega)\bigr]\Big),
\end{equation}

Informally, the FEP $(\xi_{t})_{t\geq0}$ is coupled to the ZRP $(\omega_{t})_{t\geq0}$ so that whenever a particle jumps in $\xi_{t}$, a particle in the corresponding pile in $\omega_{t}$ jumps in the same direction. 
The dynamical constraints of the FEP correspond to a zero escape rate from a site that contains only one particle (the last particle is `trapped').  
The mapping is not one-to-one and is only defined up to the position of the hole with label $y_t(0)$;
see \cite[Section 3]{blondel2020} for further discussion.

It is clear from the construction that $\xi_t \in \mathcal{E}_{N,k}$ if and only if $\omega^{\xi}_t(x) \geq 1$ for each $x \in [N-k]$. 
This is useful for bounding the time to reach the ergodic component in both the segment and on the circle. 
This is summarised in the following lemma, which follows immediately from the construction above.

\begin{lemma}
\label{lem:ZRPhit}
    Let $(\xi_t)_{t\geq 0}$ be an FEP on the $\Omega_{N,k}$ and $(\omega^{\xi}_t)_{t\geq 0}$ the associated ZRP on $\mathbb{N}_0^{N-k+1}$. 
    Let $\tau_\mathcal{E} = \inf\{t\geq 0 \,:\, \xi_t \in \mathcal{E}_{N,k}\}$ and $\tau_{\textrm{ZR}} = \inf\{t\geq 0 \,:\, \omega^{\xi}_t(i) \geq 1 \textrm{ for each } i \in [N-k+1]\}$, then $\tau_\mathcal{E} = \tau_{\textrm{ZR}}$ a.s. . 

    Similarly, Let $(\hat\xi_t)_{t\geq 0}$ be an FEP on the $\Omega_{N,k}^\circ$ and $(\hat\omega^{\xi}_t)_{t\geq 0}$ the associated ZRP on $\mathbb{N}_0^{\mathbb{T}_{N-k}}$. 
    Let $\tau_\mathcal{G} = \inf\{t\geq 0 \,:\, \hat\xi_t \in \mathcal{G}_{N,k}\}$ and $\hat\tau_{\textrm{ZR}} = \inf\{t\geq 0 \,:\, \hat\omega^{\xi}_t(i) \geq 1 \textrm{ for each } i \in \mathbb{T}_{[N-k]}\}$, then $\tau_\mathcal{G} = \hat\tau_{\textrm{ZR}}$ a.s. . 
\end{lemma}

%%%%%%%%%%%%%%%%%%%%%%%%%%%%%%%%%%%%%%%%%%
%%%%%%%%%%%%%%%%%%%%%%%%%%%%%%%%%%%%%%%%%%

\subsection{Proofs} \label{sec:segproofs}

%%%%%%%%%%%%%%%%%%%%%%%%%%%%%%%%%%%%%%%%%%
%%%%%%%%%%%%%%%%%%%%%%%%%%%%%%%%%%%%%%%%%%

In this section we prove Theorem \ref{th:fssepsegment} and Theorem \ref{th:fasepsegment}, beginning with the upper bound in \eqref{eq:fssep1} of Theorem \ref{th:fssepsegment}.

The idea for the proof of the upper bound in \eqref{eq:fssep1} is to bound the probability that $\eta^{+}_{t}$ and $\eta^{-}_{t}$ have not coupled under $\mathbb{P}$. 
We split the coupling time into two parts: Firstly, the time until both $\eta^{+}_{t}$ and $\eta^{-}_{t}$ have arrived at the ergodic component, secondly we bound the remaining coupling time by the time it takes to couple from $\eta^\vee$ and $\eta^\wedge$. 
We bound the hitting time of the ergodic component using the coupling with the OBEP in the previous section.
Finally, since the dynamics of the lattice path on the ergodic component is equivalent to the SSEP the conclusion of the proof is a consequence of previous results \cite[Proposition 8.1]{lacoinsegment}.

\begin{proof}[Proof of the SFEP upper bound] We now prove the upper bound in \eqref{eq:fssep1}.
Let $(\eta_{t})_{t\geq 0}$ and $(\eta_{t}')_{t\geq 0}$ be two FEP processes with initial configurations $\eta,\eta'\in\overline{\Omega}_{N,k}$ (respectively). Then by Proposition \ref{prop:partial} both $\eta_{t}$ and $\eta_{t}'$ are squeezed between $\eta_{t}^{-}$ and $\eta_{t}^{+}$ under the graphical construction $\mathbb{P}$.
Therefore
$$
\mathbb{P}\bigl(\eta_{t}\neq\eta_{t}'\bigr)\leq\mathbb{P}\bigl(\eta_{t}^{-}\neq\eta_{t}^{+}\bigr)\,.
$$
Following standard reasoning (see e.g. \cite{Wilson_2004}), it is possible to bound the distance to equilibrium in terms of the coupling time of the maximal and minimal configurations.
Specifically, for all configurations $\eta\in\overline{\Omega}_{N,k}$ we have
\begin{align*}
    \bigl\lVert P_{t}^{\eta}-\mu_{N,k}\bigr\rVert_{\mathrm{TV}}&=\bigl\lVert P_{t}^{\eta}-P_{t}^{\mu_{N,k}}\bigr\rVert_{\mathrm{TV}} \leq\max_{\eta'\in\Omega_{N,k}}\bigl\lVert P_{t}^{\eta}-P_{t}^{\eta'}\bigr\rVert_{\mathrm{TV}}.
\end{align*} 
By the usual coupling bound on the total variation distance (see e.g.\ \cite[Proposition 4.7]{LevinPeresWilmer2006}), it follows that
$$
\bigl\lVert P_{t}^{\eta}-P_{t}^{\eta'}\bigr\rVert_{\mathrm{TV}}\leq\mathbb{P}\bigl(\eta_{t}\neq\eta_{t}'\bigr)\leq\mathbb{P}\bigl(\eta_{t}^{-}\neq\eta_{t}^{+}\bigr)\,.
$$ 
Combining the two inequalities establishes the following bound on the distance to equilibrium,
\begin{align}
\label{eq:coupletime}
d_{N,k}(t)\leq \mathbb{P}\bigl(\eta_{t}^{-}\neq\eta_{t}^{+}\bigr)\,.
\end{align}

Let
$$
\tau=\inf\bigl\{t\geq0:\eta_{t}^{-},\eta_{t}^{+}\in\overline{\mathcal{E}}_{N,k}\bigr\}.
$$
If $t=t_{1}+t_{2}$ then
\begin{align}
    \mathbb{P}\bigl(\eta_{t}^{-}\neq\eta_{t}^{+}\bigr)&\leq\mathbb{P}\bigl(\tau>t_{1}\bigr)+\mathbb{P}\bigl(\eta_{t}^{-}\neq\eta_{t}^{+}\;,\;\tau<t_{1}\bigr) \label{eq:upperssep1}\\[1ex]
    &\leq2\mathbb{P}\bigl(\eta_{t_{1}}^{-}\notin\overline{\mathcal{E}}_{N,k})+\mathbb{P}\bigl(\eta_{t_{2}}^{\vee}\neq\eta_{t_{2}}^{\wedge}\bigr)\,.\label{eq:upperssep2}
\end{align}
The first term in \eqref{eq:upperssep2} follows by a union bound and symmetry, and the second term follows from an application of the Strong Markov property and the partial order \eqref{eq:partial1} on $\overline{\mathcal{E}}_{N,k}$.

We now bound the first term on the RHS of \eqref{eq:upperssep2}.
Recalling Remark \ref{rmk:OBEPisFEP}, $(\eta_{t}^{-})_{t\geq0}$ is equivalent to an OBEP on $\Omega_{k-1}$ with parameters $(1/2,0,0,0,1/2)$ until the first time it hits $\overline{\mathcal{E}}_{N,k}$. 
By \eqref{eq:rightboundaryarrive}
$$
\mathbb{P}\bigl(\eta_{t}^{-}\notin\overline{\mathcal{E}}_{N,k}\bigr)=\mathbb{P}\bigl(\tau_{N-k}>t\bigr),
$$
where the right hand side is the probability at most $N-k-1$ particles have entered at the right boundary in the OBEP by time $t$.

We observe that holes (empty sites) in the OBEP perform simple symmetric random walks with reflection at the left boundary and absorption at the right boundary, see for example \cite[Section 4.1]{gantert2020mixing} for a similar argument. 

Let $(\xi_t)_{t\geq 0}$ denote the OBEP associated with $(\eta_{t}^{-})_{t\geq 0}$, i.e. with empty initial condition.
Label particle holes in $\xi_0$ from left to right by $i=1,2,\ldots,k-1$.
Let $\hat\tau_i$ denote the absorption time of the $i$--th hole at the right boundary.

Since the $i$--th hole performs a simple symmetric random walk, with reflection at the left boundary, the probability that it has not been absorbed at the right boundary by time $2k^2$ is at most $1/2$ (for example, by a standard martingale argument and Markov's inequality).
Hence, by an application of the Markov property,
\begin{equation*} 
\mathbb{P}\bigl(\hat\tau_i > 2nk^2\bigr)\leq \frac{1}{2^{n}},
\end{equation*}
for all $i\in[k-1]$. This bound does not depend on the starting position of the given hole.
It follows, by a union bound over the left most $N-k$ holes, that for each $\epsilon > 0$ there exists a $C>0$ (independent of $\epsilon$) such that 
\begin{equation}\label{eq:sseptime1}
\mathbb{P}\bigl(\tau_{N-k} > 2C k^2\log(N-k)\bigr)\leq \frac{(N-k)}{2^{C \log(N-k)}}\leq \epsilon/3 ,
\end{equation}
for $N-k$ sufficiently large.

We now bound the second term on the RHS of \eqref{eq:upperssep2}.
Fix $\epsilon>0$, then, by \cite[Proposition 8.1]{lacoinsegment} and Remark \ref{rmk:FEPandsep}, there exists a $\delta_\epsilon>0$ such that for
$$
t=\frac{k^{2}}{\pi^{2}}(1+\delta_\epsilon)\log\min(N-k,2k-N)\,.
$$
we have, for all $N$ sufficiently large,
\begin{equation} \label{eq:sseptime2}
\mathbb{P}\bigl(\eta_{t}^{\wedge}\neq\eta_{t}^{\vee}\bigr)\leq\epsilon/3\,.
\end{equation} 
Inserting \eqref{eq:sseptime1} and \eqref{eq:sseptime2} into \eqref{eq:upperssep2}, and applying \eqref{eq:coupletime},  we see that 
$$
\limsup_{N\rightarrow\infty}\frac{T_{\mathrm{seg}}^{N,k}(\epsilon)}{N^{2}\log(N-k)}\leq C_{2},
$$
for some constant $C_{2}>0$, which completes the proof. 
\end{proof}

We now prove the lower bound in Theorem \ref{th:fssepsegment}. 
Recall, by assumption the number of particles in the FEP, $k$, satisfies $N/2 < k < N$.
When the number of particles in the FEP is not too close to $N/2$ we compare the FEP on the ergodic component directly to the SSEP. 
If $k$ is very close to $N/2$, in particular in the regime $\log(2k-N)\ll\log(N-k)$, the mixing time of the associated SSEP is much smaller than the hitting time of the ergodic component. In this case the latter determines the mixing time, which we control using the comparison to the OBEP. 

\begin{proof}[Proof of SFEP lower bound] We now prove the lower bound in \eqref{eq:fssep1}.
Firstly, we restrict to considering the ergodic component
$$
d_{N,k}(t)\geq\max_{\eta\in\mathcal{E}_{N,k}}\bigl\lVert P_{t}^{\eta}-\mu_{N,k}\bigr\rVert_{\mathrm{TV}},
$$
By Remark \ref{rmk:FEPandsep}, under the coupling $\mathbb{P}$ the SFEP on $\mathcal{E}_{N,k}$ is equivalent to a SSEP on $\Omega_{k-1,N-k}$. Hence, by an application of \cite[Lemma 5]{Wilson_2004} (see also \cite[Section 7]{lacoinsegment}), it holds that
\begin{equation} \label{eq:dnklower}
d_{N,k}\Bigl(\frac{1}{\pi^{2}}k^{2}(1-\delta)\log\min(N-k,2k-N)\Bigr)\geq\epsilon,
\end{equation}
for all $\epsilon\in(0,1)$, $\delta>0$ and $N$ sufficiently large.
This is sufficient to give the lower bound in \eqref{eq:fssep1} when $\log(2k-N)$ has the same order as $\log(N-k)$. 

Assume now that $\log(2k-N)\ll\log(N-k)$. In this case the lower bound on the mixing time given by \eqref{eq:dnklower} is much smaller than the time to reach the ergodic component. 
By Remark \ref{rmk:OBEPisFEP}, it is sufficient to consider the time for $N-k$ particles to enter the OBEP on $\Omega_{k-1}$ with parameters $(1/2,0,0,0,1/2)$ and empty initial condition.
The OBEP  with parameters $(1/2,0,0,0,1/2)$ has a unique absorbing state given by the configuration which is completely filled.
This state is reached, starting from the empty configuration, at time $\tau_{k-1}$, defined in \eqref{eq:obepentrytime}. 
It follows from \cite[Lemma 3.1 and Lemma 3.3]{gantert2020mixing}, that 
for $\epsilon\in(0,1)$ and $0<\delta<(1-\epsilon)/\epsilon$ 
\begin{equation} \label{eq:fillsegment}
\mathbb{P}\Bigl(\tau_{k-1}>\frac{2}{\pi^{2}}k^{2}\log k\Bigr)\geq\epsilon(1+\delta),
\end{equation}
for all $N$ sufficiently large. 
Recall from \eqref{eq:rightboundaryarrive}, for the FEP to reach the ergodic component
we only require that $N-k$ particles have entered the OBEP started from the empty configuration. 
Then $\tau_{k-1}-\tau_{N-k}$ is the time for the remaining $2k-N-1$ particles to enter, or equivalently $2k-N-1$ holes to exit at the right boundary. 
By the argument leading to the upper bound in \eqref{eq:sseptime1}, and the strong Markov property at $\tau_{N-k}$,
\begin{equation} \label{eq:lastfillsegment}
\mathbb{P}\bigl(\tau_{k-1}-\tau_{N-k}>Ck^{2}\log(2k-N)\bigr)\leq\epsilon\delta,
\end{equation}
for all $N$ sufficiently large. 
By a union bound 
\begin{align*}
\mathbb{P}\Big(&\tau_{\mathrm{k-1}}>\frac{2}{\pi^{2}}k^{2}\log k\Big) 
\leq \\
&\mathbb{P}\Big(\tau_{k-1}{-}\tau_{N-k}>Ck^{2}\log(2k{-}N)\Big)+\mathbb{P}\Big(\tau_{N-k}>\frac{2}{\pi^{2}}k^{2}\log k{-}Ck^{2}\log(2k{-}N)\Big), 
\end{align*}
and it follows from \eqref{eq:fillsegment} and \eqref{eq:lastfillsegment} that
$$
\mathbb{P}\Bigl(\tau_{N-k}>\frac{2}{\pi^{2}}k^{2}\log k-Ck^{2}\log(2k-N)\Bigr)\geq\epsilon,
$$
for all $N$ sufficiently large. 
Noting that if $\log(N-k)\gg\log(2k-N)$, then there exists a constant $C_{1}>0$ (not depending on $\epsilon$) such that 
$
C_{1} \log(N-k)\leq\log k,
$
for all $N$ sufficiently large. 
Therefore, by \eqref{eq:rightboundaryarrive}
$$
\mathbb{P}\bigl(\eta_{C_{1}k^{2}\log(N-k)}^{-}\notin\overline{\mathcal{E}}_{N,k}\bigr)=\mathbb{P}\bigl(\tau_{N-k}>C_{1}k^{2}\log(N-k)\bigr)\geq\epsilon. 
$$
Since $d_{N,k}(t)\geq\mathbb{P}(\eta_{t}^{-}\notin\overline{\mathcal{E}}_{N,k})$ the lower bound in \eqref{eq:fssep1} follows. 
\end{proof}

We now prove the upper bound in Theorem \ref{th:fasepsegment}, for the mixing time of the AFEP on the interval. 
The idea here is that the AFEP on $\Omega_{N,k}$ with initial condition $\eta^{-}$ behaves like the reverse--bias phase for OBEP on $\Omega_{k-1}$, until the first time that $N-k$ particles have entered, and this hitting time of the ergodic set determines the mixing time. 

\begin{proof}[Proof of the upper bound for the AFEP] 
We prove the upper bound in \eqref{eq:fasep1}.
By the same reasoning in the proof of the upper bound in \eqref{eq:fssep1}, if $t=t_1+t_2$ then
\begin{equation} \label{eq:upperasep2}
\mathbb{P}\bigl(\eta_{t}^{-}\neq\eta_{t}^{+}\bigr)\leq\mathbb{P}\bigl(\eta_{t_{1}}^{\vee}\neq\eta_{t_{1}}^{\wedge}\bigr)+ \mathbb{P}\bigl(\eta_{t_{2}}^{-}\notin\overline{\mathcal{E}}_{N,k})+\mathbb{P}\bigl(\eta_{t_{2}}^{+}\notin\overline{\mathcal{E}}_{N,k}).
\end{equation}
Fix $\epsilon\in(0,1)$.
For the first term on the RHS of \eqref{eq:upperasep2} it follows from Labb\'{e} and Lacoin \cite[Section 3.4]{cutoffasep} that 
\begin{equation} \label{eq:asymcoal}
\mathbb{P}\bigl(\eta_{N^2}^{\vee}\neq\eta_{N^2}^{\wedge}\bigr)\leq\epsilon/3,
\end{equation}
for all $N$ sufficiently large. 

For the second term on the RHS of \eqref{eq:upperasep2} we apply \cite[Theorem 1.3]{gantert2020mixing}. 
\begin{equation} \label{eq:goodextreme}
\mathbb{P}\bigl(\eta_{N^2}^{-}\notin\overline{\mathcal{E}}_{N,k})\leq\epsilon/3,
\end{equation}
for all $N$ sufficiently large.
Note that in the previous two bounds, it is sufficient to take times larger than a suitable constant times $N$.

For the third term on the RHS of \eqref{eq:upperasep2} we wish to apply \cite[Theorem 1.4]{gantert2020mixing}, however the bound obtained by direct application (and Remark \ref{rmk:OBEPisFEP}) is exponential in $k$ rather than $N-k$. 
To fix this we apply the following claim:
\begin{claim} \label{claim:OBEP}
Let $T_{\mathrm{asep}}^{n}(\epsilon)$ denote the mixing time of the $(q,q,0,0,0)$--OBEP on $\Omega_{n}$. 
If
\begin{equation} \label{eq:claimOBEP}
t_{2}\geq T_{\mathrm{asep}}^{N-k}(\epsilon), \quad \textrm{then} \quad \mathbb{P}\bigl(\eta_{t_{2}}^{+}\notin\overline{\mathcal{E}}_{N,k}\bigr)\leq\epsilon. 
\end{equation}

\end{claim}
Combining \eqref{eq:upperasep2}, \eqref{eq:asymcoal}, \eqref{eq:goodextreme} and \eqref{eq:claimOBEP}, it follows that 
\begin{align*}
T_{\mathrm{seg}}^{N,k}(\epsilon)&\leq\inf\bigl\{t>0\,:\,\mathbb{P}(\eta_{t}^{-}\neq\eta_{t}^{+})\leq\epsilon\bigr\}\leq\max\bigl(N^2,T_{\mathrm{asep}}^{N-k}(\epsilon/3)\bigr),
\end{align*}
for all $N$ sufficiently large.
Applying the assumption $N-k\gg\log N$ and \cite[Theorem 1.4]{gantert2020mixing} yields 
$$
\limsup_{N\rightarrow\infty}\frac{\log T_{\mathrm{seg}}^{N,k}(\epsilon)}{N-k}\leq\log\Biggl(\frac{p}{q}\Biggr),
$$
which completes the proof (up to proving Claim \ref{claim:OBEP}).

To prove Claim \ref{claim:OBEP} we appeal to Remark \ref{rmk:OBEPisFEP} again.
The first time, $t$, that $\eta_{t}^{+}\in\overline{\mathcal{E}}_{N,k}$ is equivalent to the first time that $N-k$ particles have entered the left boundary of the corresponding
OBEP on $\Omega_{k-1}$ with parameters $(q,q,0,0,0)$ and with empty initial condition.
It is possible to couple $(q,q,0,0,0)$-OBEPs on $\Omega_{k-1}$ and on a smaller system, $\Omega_{N-k}$, using the same Poisson Process `clock-rings' attached to particles (as they enter), in such a way that the number of particles in the larger system is always at least the number in the smaller system. 
The coupling is standard, and we leave the details to the reader.
Hence, at least $N-k$ particles have entered the $(q,q,0,0,0)$-OBEP on $\Omega_{k-1}$ at the time, $\tau_\textrm{fill}$, that the OBEP on $\Omega_{N-k}$ has completely filled.
Since the $(q,q,0,0,0)$-OBEP has a unique absorbing state (all filled), it follows that $\mathbb{P}(\tau_\textrm{fill} >T_{\mathrm{asep}}^{N-k}(\epsilon)) \leq \epsilon$, as required.
\end{proof}

We now prove the lower bound in Theorem \ref{th:fasepsegment} using a coupling to the ZRP on the segment, and a bound on the hitting time of small sets due to Aldous and Brown \cite{aldousbrown}.

\begin{proof}[Proof of lower bound in \eqref{eq:fasep1}]
We will show that the time to hit the ergodic component started from `bad' initial configurations is exponentially large, which is sufficient since
\begin{align}
P_t^{\nu}(\mathcal{E}_{N,k}^c) \leq \|P_t^{\nu} - \mu_{N,k}\|_{\textrm{TV}}\leq d_{N,k}(t)\,.
\end{align}
for an initial distribution $\nu$.

Let 
$$
\mathcal{H}_{N,k}=\bigl\{\xi\in\Omega_{N,k}:\xi(1)=0,\xi(N)=1,\,\xi(x)+\xi(x+1)\geq1,\mbox{for all }x\in[N-1]\bigr\}.
$$
The configurations $\mathcal{H}_{N,k}$ are almost in the ergodic component; if the particle at $x=2$ jumps left in the dynamics the resulting configuration lies in the ergodic component. 
In particular $\mathcal{H}_{N,k}$ is the set of configurations with $\xi(1)=0,\xi(2)=1$ that may be reached by the dynamics with initial condition $\xi^{+}\in\mathcal{E}_{N,k}^{c}$ from Definition \ref{def:specialpaths}. 

Recall, from Section \ref{sec:coupleZRP}, that the FEP on the segment $[N]$ is equivalent to a ZRP on the segment $[N-k+1]$, via the map $\Pi$. 
Let $\Sigma_{n,m}$ denote the set of zero--range configurations on the segment $[n]$ with $m$ particles.
Observe that
$$
{\Pi}\bigl(\mathcal{H}_{N,k}\bigr)=\bigl\{\omega\in\Sigma_{N-k+1,k}:\omega(1)=0,\,\omega(x)\geq 1,\mbox{ for all }x\in \{2,\ldots,N-k+1\}\bigr\}, 
$$
the set of zero--range configurations in which all sites except $x=1$ are occupied by at least one particle.
By Lemma \ref{lem:ZRPhit}, the time to hit the ergodic component, $\tau_\mathcal{E}$, starting from any initial configuration in $\mathcal{H}_{N,k}$, is given by the first time a particle enters the site $1$ in the corresponding ZRP. 
Furthermore, this time is almost surely larger than the first time, $\tau_2 = \inf\{t\,:\, \omega_t(2) \geq 2\}$, that the ZRP has at least two particles on site $2$ (since particles can only enter site $1$ from site $2$ and the escape rate is only positive if the occupation of the site is greater than one).

For all times $t \leq \tau_2$, we may couple the ZRP with generator $\mathcal{L}_{\mathrm{ZR}}^{N-k+1}$ in the set $\Pi(\mathcal{H}_{N,k})$ to a constant rate ZRP in $[N-k]$, by ignoring the site $1$ and `deleting' the `stuck' particles.
More precisely, if $(\omega_t)_{t\geq 0}$ is a ZRP with generator $\mathcal{L}_{\mathrm{ZR}}^{N-k+1}$ and $\omega_0 \in \Pi(\mathcal{H}_{N,k})$, then for all $t \leq \tau_2$ it is straightforward to check that $(\tilde\omega_t)_{t\geq 0}$, defined by 
\[
\tilde\omega_t(x-1) = \omega_t(x) - 1 \quad \textrm{for} \quad x \in \{2,3,\ldots,N-k+1\},
\]
defines a ZRP with generator 
\begin{align}
\label{eq:constZRP}
    \widetilde{\mathcal{L}}_\mathrm{ZR}f(\omega)=&\sum_{i=1}^{N-k-1}p\1_{\{\omega(i)\geq 1\}}\bigl[f(\omega^{i,i+1})-f(\omega)\bigr]  +\sum_{x=2}^{N-k}q\1_{\{\omega(i)\geq 1\}}\bigl[f(\omega^{i,i-1})-f(\omega)\bigr]\,,
\end{align}
on $\Sigma_{N-k,2k-N}$.
Define the bijection $\widetilde{\Pi}\colon \mathcal{H}_{N,k} \to \Sigma_{N-k,2k-N}$ by
\begin{align*}
     \widetilde{\Pi}(\xi)(x-1) = {\Pi}(\xi)(x)-1 \quad \textrm{for} \quad x \in \{2,3,\ldots,N-k+1\}\,.
\end{align*}
If $\xi \in \mathcal{H}_{N,k}$ then $\tilde\omega_t = \widetilde{\Pi}(\xi_t)$ evolves as a ZRP with generator \eqref{eq:constZRP}, for all times $t \leq \tau_2$.
Let $\widetilde{Q}^\nu$ denote the law of constant rate ZRP, evolving according to \eqref{eq:constZRP} on $\Sigma_{N-k,2k-N}$, starting from the distribution $\nu$.
From the preceding construction, it is clear that for $\xi \in \mathcal{H}_{N,k}$,
\begin{align}
\label{eq:ZRPhit}
   P_t^{\xi}(\mathcal{E}_{N,K}^c) \geq  \widetilde{Q}^{\widetilde{\Pi}(\xi)}(\tau > t),
\end{align}
where $\tau = \inf\{t \geq 0 \,:\,  \tilde\omega \in E\}$ and $E = \{ \tilde\omega\,:\,\tilde\omega(1) \geq 1\}$ .

For the remainder of the proof, let $n =N-k$ and $m =2k-N$.
It is straightforward to check (see for example \cite{SPITZER1970246}) that the process given by \eqref{eq:constZRP} on $\Sigma_{n,m}$ has equilibrium distribution given by
\begin{equation} \label{eq:freezerostat}
\pi_{n,m}(\omega)=\frac{1}{Z_{n,m}}\prod_{x=1}^{n}\lambda^{(n+1-x)\omega(x)},
\end{equation}
where $\lambda = q/p<1$ and 
$$
Z_{n,m}=\sum_{\omega\in\Sigma_{n,m}}\prod_{x=1}^{n}\lambda^{(n+1-x)\omega(x)},
$$
is the normalisation constant.
Under this distribution, the event $E$, that the first site contains a particle, is rare.
Hence by standard bounds on the hitting time of rare sets (see for example \cite{aldousbrown})
\begin{align}
    \label{eq:aldousbrown2}
    \widetilde{Q}^{\pi_{n,m}}(\tau > t) \geq \pi_{n,m}(E^c)\mathrm{exp}\Biggl(-t\frac{q(E,E^{c})}{\pi_{n,m}(E^{c})}\Biggr)\,,
\end{align}
where $q(E,E^{c})$ is the \textit{capacity} given by 
$$
q(E,E^{c})=\sum_{\omega\in E}\sum_{\omega'\in E^{c}}\pi_{n,m}(\omega)\widetilde{\mathcal{L}}_{\mathrm{ZR}}(\omega,\omega') = p\, \pi_{n,m}(\{\omega\,:\,\omega(1) = 1\})\,.
$$
Since the stationary measures are of product form (conditional product measures), the probability on the right--hand side can be expressed explicitly as 
\[
\pi_{n,m}(\{\omega(1) = 1\}) = \frac{\lambda^n Z_{n-1,m-1}}{Z_{n,m}}\,,
\]
and similarly
\[
\pi_{n,m}(\{\omega(1) = 0\}) = \frac{Z_{n-1,m}}{Z_{n,m}}\,.
\]
It follows that 
\begin{align}
\label{eq:pione}
\pi_{n,m}(\{\omega(1) = 1\}) \leq \frac{\pi_{n,m}(\{\omega(1) = 1\})}{\pi_{n,m}(\{\omega(1) = 0\})} = \lambda^{n-1}\frac{\lambda Z_{n-1,m-1}}{Z_{n-1,m}}\,.
\end{align}
Also, the final ratio on the right--hand side can be identified with $\pi_{n-1,m}(\{\omega(n-1)>0\})$, since every configuration $\omega\in\Sigma_{n-1,m-1}$ may be uniquely identified with a configuration in $\omega'\in\Sigma_{n-1,m}$ such that $\omega'(n-1)>0$ by adding a single particle to the site $n-1$ in $\omega$ (and vice versa).
It follows from \eqref{eq:pione} that
\begin{align}
\label{eq:pione2}
\pi_{n,m}(\{\omega(1) = 1\}) \leq \lambda^{n-1}\,.
\end{align}
Moreover, $\pi_{n,m}(E^{c})\geq 1-\pi_{n,m}(\{\omega(1)=1\})\geq 1/2$ for $n-1=N-k$ sufficiently large (depending on $p/q$). 
Combining these bounds with \eqref{eq:ZRPhit} and \eqref{eq:aldousbrown2}, we have 
\begin{equation} \label{eq:aldousbrown3}
\sup_{\xi \in \mathcal{H}_{N,k}}P_{t}^{\xi}\bigl(\mathcal{E}_{N,k}^{c}\bigr)\geq \sup_{\xi \in \mathcal{H}_{N,k}} \widetilde{Q}^{\widetilde{\Pi}(\xi)}(\tau > t) \geq \widetilde{Q}^{\pi_{n,m}}(\tau > t) \geq  \mathrm{exp}\bigl(-2tp\lambda^{N-k}\bigr)\,,
\end{equation}
since $\widetilde{\Pi}(\mathcal{H}_{N,k}) = \Sigma_{N-k,2k-N}$.
Hence, for $\epsilon\in(0,1)$, choosing 
$$
t= \frac{1}{2p}\log(\epsilon^{-1})\Bigl(\frac{p}{q}\Bigr)^{N-k},
$$
it follows from \eqref{eq:aldousbrown3} that 
$$
d_{N,k}(t) = \sup_{\pi}\lVert P_{t}^{\pi}-\mu_{N,k}\rVert_{\mathrm{TV}} \geq \sup_{\xi \in \mathcal{H}_{N,k}}P_{t}^{\xi}\bigl(\mathcal{E}_{N,k}^{c}\bigr) \geq \epsilon \,,
$$
which completes the proof. 
\end{proof}

%%%%%%%%%%%%%%%%%%%%%%%%%%%%%%%%%%%%%%%%%%
%%%%%%%%%%%%%%%%%%%%%%%%%%%%%%%%%%%%%%%%%%

\section{FEP on the circle}

%%%%%%%%%%%%%%%%%%%%%%%%%%%%%%%%%%%%%%%%%%
%%%%%%%%%%%%%%%%%%%%%%%%%%%%%%%%%%%%%%%%%%

In Section \ref{sec:cirlower} we prove Theorem \ref{th:ergodicmixing}(a) and in Section \ref{sec:lowersob} we prove Theorem \ref{th:ergodicmixing}(b). 
The proof of Proposition \ref{prop:hiterg} can be found in Appendix \ref{sec:hitergodic}. 

%%%%%%%%%%%%%%%%%%%%%%%%%%%%%%%%%%%%%%%%%%

\subsection{Lower bound on the mixing time} \label{sec:cirlower}

%%%%%%%%%%%%%%%%%%%%%%%%%%%%%%%%%%%%%%%%%%

We first give a proof of Theorem \ref{th:ergodicmixing}(a).
The idea of the proof is to map the FEP on the ergodic component, $\mathcal{G}_{N,k}$, to a simple symmetric exclusion process (up to a random rotation).
We show that, started with all the excess particles in one half of the circle, before a time of order $c k^2 \log (N-k)$ there exists a region with anomalously low density with respect to $\nu_{N,k}$.

\begin{proof}[Proof of Theorem \ref{th:ergodicmixing}(a)]
Let $\xi\in\mathcal{G}_{N,k}$ and assume that $k\geq2N/3$.
Similarly to the construction in Section \ref{sec:coupleZRP}, where we define the mapping from the FEP on the circle to a ZRP via the position of a tagged particle, we now describe a mapping to a SSEP -- roughly, the exclusion process performed by the particle-hole dimers.

Given a FEP configuration, $\xi \in  \mathcal{G}_{N,k}$, and a tagged particle in $\xi$, at position $x(0)$, label the positions of the remaining particles in clockwise order $x(1),x(2),\ldots,x(k-1)$. 
This defines a unique SEP configuration $\mathcal{S}[\xi,x(0)] \in \Omega_{k,N-k}^\circ$ via
\begin{align*}
    \mathcal{S}[\xi,x(0)](i) =\mathds{1}\bigl(\xi(x(i)+1)=0\bigr) = 1-\xi(x(i)+1) \,, \quad \textrm{for $0\leq i\leq k-1$.} 
\end{align*}
Let $X_t(0)$ denote the location of the tagged particle initially at $x(0)$ under the FEP dynamics, and let the remaining particles retain their labels in clockwise order for all $t \geq 0$. 
It is straightforward to check that $(\zeta_t^\xi)_{t\geq 0}$ defined by
\begin{align*}
    \zeta_t^\xi = \mathcal{S}[\xi_t,X_t(0)]\,,
\end{align*}
defines a simple symmetric exclusion processes (SSEP) on $\mathbb{T}_k$ with $n:=N-k$ particles.
Furthermore, if $\xi \sim \nu_{N,k}$ then $\mathcal{S}[\xi,x(0)]\sim \pi^\circ$, where $\pi^{\circ}$ is the uniform measure on $\Omega_{k,n}^\circ$.
It turns out that, under $\pi^\circ$, it is exponentially unlikely to observe \emph{any} region in which the total number of particles differs from the expected number by an amount of order $\sqrt{n}$,  see for example \cite[Remark 4.3]{lacoindiff}. More precisely,
\begin{align}
\label{eq:expsmall-fluc}
\pi^\circ\left( \exists\, x,y \in \mathbb{T}_k, \ \left| \sum_{z=x+1}^y\left(\zeta(z) - \frac{n}{k}\right)\right| \geq s \sqrt{n}\right) \leq e^{-cs^2}\,.
\end{align}

Adapting a slightly refined version of the argument of Morris \cite[Section 6]{Morris_2006}, we show that there exists an initial condition $\xi$ and a constant $c >0$ such that for $t \leq c k^2 \log n$, with high probability, there exists a region with anomalously low density in $\zeta_t^\xi$, and hence in $\mathcal{S}[\xi_t,x'(0)]$, where $x'(0)$ is the location of the first particle to the right of site zero in $\xi_t$, since these two SSEP configurations are equal up to a random rotation. Hence, by the previous observation, the FEP cannot be well mixed.

We use the standard construction of the SSEP on $\Omega_{k,n}^\circ$ in terms of the interchange process (see, e.g. \cite{Morris_2006}). Assume, without loss of generality that $k$ is even.
Fix $\xi \in \cG_{N,k}$ and $\zeta_0 = \mathcal{S}[\xi,x(0)]$. 
Label the particles in the SEP configuration $\zeta_0$ by $1,2,\ldots,n$, and their positions, on $\mathbb{T}_k$, by $Y_0(1),\ldots,Y_0(n)$.
At rate one each particle chooses a neighbour (each with probability $1/2$) and exchanges the local configuration.
Then the unordered set $\{Y_t(1),\ldots,Y_t(n)\}$ (of the positions of the labelled particles) behaves as the set of occupied sites in $\zeta_t^\xi$ and, for each $i \in \{1,\ldots,n\}$, the marginal process on $Y_t(i)$ behaves as a simple symmetric random walk on $\mathbb{T}_k$.
Fix $\xi \in \cG_{N,k}$ such that the particle positions in $\zeta_0$ are located to minimise 
\[
\sum_{i=1}^n\left|Y_0(i) - \frac{k}{4}\right|\,.
\]
By assumption on $N$ and $k$, the left half of the circle, $U=\{\frac{k}{2},\frac{k}{2}+1,\ldots,k-1\}\subset \mathbb{T}_k$, is initially empty in $\zeta_0$.
Also, at least $n/2$ particles are initially located in the interval $S=\{\frac{k}{8},\frac{k}{8}+1,\ldots,\frac{3k}{8}\}$.
By standard results for the continuous--time simple random walk (for example by spectral decomposition), for each $i$ such that $Y_0(i) \in S$, there exists universal constants $\delta,\gamma >0$ such that
\[
\mathbb{P}(Y_t(i) \in U) \leq \frac{1}{2} - \delta\, \mathrm{exp}\left({-\frac{\gamma t}{k^2} }\right)\,.
\]
Note, if $Y_0(i)\notin S$ then $\mathbb{P}(Y_t(i) \in U)\leq1/2$. 
Let $N_t = |\{ i\,:\, Y_t(i) \in U\}|$, then since at least half the particles start in $S$,
\[
\mathbb{E}(N_t)= \mathbb{E}\Big(\sum_{i=1}^n\1_{\{Y_t(i) \in U\}} \Big)\leq \frac{n}{4}+
\frac{n}{2}\left[\frac{1}{2} -   \delta\, \mathrm{exp}\left({-\frac{\gamma t}{k^2} }\right)\right]=\frac{n}{2}\left[1 -   \delta\, \mathrm{exp}\left({-\frac{\gamma t}{k^2} }\right)\right]\,.
\]
Moreover, by negative correlations (see, e.g. \cite{Morris_2006}), we have $\textrm{Var}(N_t) \leq \mathbb{E}(N_t) \leq n/2$.
Thus, by Chebyshev's inequality,
\begin{align*}
    \mathbb{P}\left(N_t \geq \frac{n}{2} - \frac{\delta n^{3/4}}{4}\right) \leq \frac{\textrm{Var}(N_t)}{\left((n\delta/2)(e^{-\gamma t/k^2} - (1/2)n^{-1/4})\right)^2}\,. %\leq \frac{1}{(k\delta/2)(e^{-\gamma t/k^2} - (1/2)k^{-1/4})}\frac{}{}
\end{align*}
Hence, for $t \leq \frac{k^2}{4\gamma}\log n$, we have 
\[
\mathbb{P}\left(  \exists\, x,y \in \mathbb{T}_k, \ \left| \sum_{z=x+1}^y\left(\zeta_t^\xi(z) - \frac{n}{k}\right)\right| \geq \delta n^{3/4}\right) \geq \mathbb{P}\left(N_t < \frac{n}{2} - \frac{\delta n^{3/4}}{4}\right) \geq 1 - \frac{8}{\delta^2 \sqrt{n}}\,.
\]
The right-hand side is greater than $1/2$ provided that $n > (16/\delta)^2$.
However, it follows from \eqref{eq:expsmall-fluc} that the probability of the event on the left hand side is exponentially small (in $n$) under $\nu_{N,k}$. Hence, for $N$ and $k$ sufficiently large $T^{N,k}_{\mathrm{cir}}(\epsilon)\geq \frac{k^2}{4\gamma}\log (N-k)$, which completes the proof.

If $k\leq 2N/3$, we may swap the roles of particles and holes in the associated SEP configurations.
\end{proof}

%%%%%%%%%%%%%%%%%%%%%%%%%%%%%%%%%%%%%%%%%%%%%%

\subsection{Lower bound on the log--Sobolev constant} \label{sec:lowersob}

%%%%%%%%%%%%%%%%%%%%%%%%%%%%%%%%%%%%%%%%%%%%%%

We will lower bound the Dirichlet form by removing the transitions across a fixed edge (see the partitions \eqref{eq:gpartition} below), and we will upper bound the entropy using a quasi--factorisation in \cite{cesientropy}, see Proposition \ref{prop:cesi} below. 
This will allow us to compare with the log--Sobolev constant for the SSEP, for which known bounds (see \cite{yaulogsobolev,salez2020sharp}) will yield the result. 

For our purposes, it will be beneficial to relate the entropy to the conditional entropy, which we define as follows: 
\begin{definition} \label{def:condentropy}
Let $(\Omega,\mathcal{P}(\Omega),\pi)$ be a finite probability space, where $\mathcal{P}(\Omega)$ denotes the power set of $\Omega$, and let $\mathcal{F}_{1},\mathcal{F}_{2}$ be two sub--$\sigma$--algebras of $\mathcal{P}(\Omega)$. 
For $i \in \{1,2\}$, we write $\pi_{i}(f):=\pi(f\mid\mathcal{F}_{i})$. If $f$ is a nonnegative function, then we define the conditional entropy by 
\begin{equation} \label{eq:condentropy}
\mathrm{Ent}_{\pi_i}(f)=\pi_{i}\Biggl(f\log\Biggl(\frac{f}{\pi_{i}(f)}\Biggr)\Biggr).
\end{equation}
\end{definition}
The following result, due to Cesi \cite{cesientropy}, gives a quasi--factorisation of the entropy.
In particular if the two sub--$\sigma$--algebras are `close' to being independent, the entropy is bounded by the average of the conditional entropies, up to some constant. For $f\colon \Omega\to \mathbb{R}$, let $\|f\|_{1,\pi} = \pi(|f|)$ and $\|f\|_{\infty,\pi} = \sup_{\sigma\,:\, \pi\{\sigma\}>0} |f(\sigma)|$.
\begin{prop}[{\cite[Proposition 2.1]{cesientropy}}] \label{prop:cesi}
Let 
$
\vartheta(\epsilon):=84\epsilon/(1-\epsilon)^{2}.
$
If, for some $\epsilon\in[0,1)$,
\begin{equation} \label{eq:cesi1}
\bigl\lVert\pi_{2}(g)-\pi(g)\bigr\rVert_{\infty,\pi}\leq\epsilon\bigl\lVert g\rVert_{1,\pi},
\end{equation}
for all $\mathcal{F}_{1}$--measurable functions $g$, then for each $f\colon \Omega\to \mathbb{R}$
\begin{equation}\label{eq:cesi2}
\mathrm{Ent}_{\pi}(f^{2})\leq\pi\bigl[\mathrm{Ent}_{\pi_{1}}(f^{2})+\mathrm{Ent}_{\pi_{2}}(f^{2})\bigr]+\vartheta(\epsilon)\mathrm{Ent}_{\pi}(f^{2}).
\end{equation}
\end{prop}
We note that Proposition \ref{prop:cesi} is the restriction of \cite[Proposition 2.1]{cesientropy} to finite state spaces, and the statement of \cite[Proposition 2.1]{cesientropy} is far more general.

We consider two `almost independent' partitions of the ergodic component $\mathcal{G}$ (the almost independence is contained in Lemma \ref{lem:asympind}). 
Fix $2<\ell<N$, then
\begin{align}
\begin{split}\label{eq:gpartition}
\mathcal{G}_{1}^{(0)}&=\{\xi\in\mathcal{G}:\xi(1)=0\}, \quad
\mathcal{G}_{1}^{(1)}=\{\xi\in\mathcal{G}:\xi(1)=1\}, \ \ \textrm{and}\\
\mathcal{G}_{\ell}^{(0)}&=\{\xi\in\mathcal{G}:\xi(\ell)=0\}, \quad
\mathcal{G}_{\ell}^{(1)}=\{\xi\in\mathcal{G}:\xi(\ell)=1\},
\end{split}
\end{align}
For any configuration $\xi\in\mathcal{G}_{1}^{(0)}$, we must have that $\xi(2)=\xi(N)=1$ since no two holes are adjacent in the ergodic component.
Similarly, for $\xi\in\mathcal{G}_{\ell}^{(0)}$, we have $\xi(\ell-1)=\xi(\ell+1)=1$.

\begin{remark} \label{rmk:transitions}
We observe that the transitions of the SFEP restricted to $\mathcal{G}_{x}^{(0)}$ are the equivalent to those of the SFEP on the ergodic component for the segment, $\mathcal{E}_{N-1,k}$, by treating $2$ and $N$ ($\ell+1$ and $\ell-1$ respectively) as the endpoints of the segment.
Furthermore, by Remark \ref{rmk:FEPandsep} this is equivalent to the transitions of a SSEP on the segment $[k-1]$ with $N-k-1$ particles.
Similarly, restricted to $\mathcal{G}_{x}^{(1)}$ the process is equivalent to a SFEP on $\mathcal{E}_{N+1,k+1}$ and therefore to a SSEP on the segment $[k]$ with $N-k$ particles.
\end{remark}

In order to apply Proposition \ref{prop:cesi}, we need to construct two weakly dependent $\sigma$--algebras. 
Let $\mathcal{F}_{1}$ be the $\sigma$--algebra generated by $\mathcal{G}_{1}^{(0)}$ and let $\mathcal{F}_{\ell}$ be the $\sigma$--algebra generated by $\mathcal{G}_{\ell}^{(0)}$. 
These will be almost independent as $N \to \infty$ and $k/N \to \rho$ if $\ell$ is sufficiently large. 
This follows from the decay of correlations shown in \cite{blondel2020} and we postpone the proof until the end of this section.

\begin{lemma} \label{lem:asympind}
Let $k(N)$ be a sequence such that $k(N)/N\rightarrow \rho \in (\frac{1}{2},1)$ as $N\rightarrow\infty$.
Then for $i,j\in \{0,1\}$ 
\begin{equation} \label{eq:cesicondition2}
\lim_{\ell \to \infty}\lim_{N \to \infty}\frac{\nu_{N,k}\bigl(\mathcal{G}_{1}^{(i)}\cap\mathcal{G}_{\ell}^{(j)}\bigr)}{\nu_{N,k}\bigl(\mathcal{G}_{1}^{(i)}\bigr)\nu_{N,k}\bigl(\mathcal{G}_{\ell}^{(j)}\bigr)}= 1\,.
\end{equation}
\end{lemma}

Before proving Lemma \ref{lem:asympind}, we show how it may be used to yield Theorem \ref{th:ergodicmixing}(b). 
Let $\nu_{x}(\cdot)=\nu_{N,k}(\cdot\mid\mathcal{F}_{x})$ for $x\in \{1,\ell\}$.
Fix $\epsilon > 0$ such that $\vartheta(\epsilon) < 1/2$.
It follows immediately from Lemma \ref{lem:asympind} that there exists $\ell$, depending on $\epsilon$, such that for each $N$ sufficiently large (depending on $\epsilon$ and $\ell$) 
\begin{align*}
    \left\| \nu_1(g) - \nu_{N,k}(g) \right\|_{\infty,\nu_{N,k}} \leq \epsilon \left\| g\right\|_{1,\nu_{N,k}}\,,
\end{align*}
for each $\mathcal{F}_{\ell}$ measurable function $g$.
Hence by Proposition \ref{prop:cesi}, for all  $f\colon \Omega_{N,k}^\circ \to \mathbb{R}$
\begin{align}\label{eq:entbound}
    \mathrm{Ent}_{\nu_{N,k}}(f^{2})\leq 2\nu_{N,k}\bigl(\mathrm{Ent}_{\nu_{1}}(f^{2})+\mathrm{Ent}_{\nu_{\ell}}(f^{2})\bigr)\,.
\end{align}
So it remains to give an upper bound on the right hand side in terms of the Dirichlet form.
By removing transitions between parts of the partitions \eqref{eq:gpartition}  from the right hand side of \eqref{eq:dirichletform}
\begin{equation} \label{eq:diric2}
\mathcal{D}(f)\geq \frac{1}{2}\sum_{\sigma \in \{0,1\}}\sum_{\xi,\xi'\in\mathcal{G}^{(\sigma)}_{x}}\nu_{N,k}(\xi)c(\xi,\xi')\bigl(f(\xi)-f(\xi')\bigr)^{2}.
\end{equation}
Since $\nu_{N,k}$ is uniform on $\mathcal{G}$, observe that $\nu_{x}(f) = \nu_{N,k}(\cdot\mid \mathcal{F}_x)$ is the function given by  
\begin{align*}
\nu_{x}\bigl(f\bigr)(\xi)=\mathds{1}\bigl(\xi_x = 0\bigr)\sum_{\xi'\in\mathcal{G}_{x}^{(0)}}\frac{f(\xi')}{\bigl\lvert\mathcal{G}_{x}^{(0)}\bigr\rvert} +\mathds{1}\bigl(\xi_x = 1\bigr)\sum_{\xi'\in\mathcal{G}_{x}^{(1)}}\frac{f(\xi')}{\bigl\lvert\mathcal{G}_{x}^{(1)}\bigr\rvert} \,,
\end{align*}
for $\xi\in\mathcal{G}$. 
In particular the value of the function $\nu_{x}(f)$ only depends on the occupation of site $x$ and for each configuration in $\mathcal{G}_{x}^{(i)}$ it is the uniform average of the function $f$ over  $\mathcal{G}_{x}^{(i)}$. 
This is also the stationary measure of the FEP restricted to the appropriate interval (see Remark \ref{rmk:transitions}).
Hence, from \eqref{eq:diric2} we have
\begin{align}
    \mathcal{D}(f)\geq \frac{k}{N} \mathcal{D}_x^{(1)}(f) + \left(1-\frac{k}{N}\right)\mathcal{D}_x^{(0)}(f)\,, 
\end{align}
where 
\begin{align*}
    \mathcal{D}_x^{(i)}(f) = \frac{1}{2}\sum_{\xi,\xi'\in\mathcal{G}_{x}^{(i)}}\frac{1}{\bigl\lvert\mathcal{G}_{x}^{(i)}\bigr\rvert}c(\xi,\xi')&\bigl(f(\xi)-f(\xi')\bigr)^{2}
\end{align*}
is the Dirichlet form of the process restricted to the interval described in Remark \ref{rmk:transitions}.
Denoting the log--Sobolev constant of the SSEP on $m$ sites with $n$ particles by $\tilde\alpha_{m,n}$ it follows that 
\begin{align*}
    \mathcal{D}(f)&\geq \tilde\alpha_{k,N-k}\,\frac{k}{N} \sum_{\xi'\in\mathcal{G}_{x}^{(1)}}\frac{f^{2}(\xi')}{\bigl\lvert\mathcal{G}_{x}^{(1)}\bigr\rvert}\log\Biggl(\frac{f^{2}(\xi')}{\nu_{x}\bigl(f^{2}\bigr)(\xi')}\Biggr) \nonumber\\ &\quad + \tilde\alpha_{k-1,N-k-1}\,\left(1-\frac{k}{N}\right)\sum_{\xi'\in\mathcal{G}_{x}^{(0)}}\frac{f^{2}(\xi')}{\bigl\lvert\mathcal{G}_{x}^{(0)}\bigr\rvert}\log\Biggl(\frac{f^{2}(\xi')}{\nu_{x}\bigl(f^{2}\bigr)(\xi')}\Biggr) \nonumber\\
    &\geq\min(\tilde\alpha_{k-1,N-k-1},\tilde\alpha_{k,N-k})\nu_{N,k}\bigl(\mathrm{Ent}_{\nu_{x}}(f^{2})\bigr)\,.
\end{align*}
Hence 
\begin{align*}
     \mathcal{D}(f)&\geq \frac{1}{2}\min(\tilde\alpha_{k-1,N-k-1},\tilde\alpha_{k,N-k})\nu_{N,k}\bigl(\mathrm{Ent}_{\nu_{1}}(f^{2}) +\mathrm{Ent}_{\nu_{\ell}}(f^{2})\bigr) \\&\geq C\,N^{-2}\nu_{N,k}\bigl(\mathrm{Ent}_{\nu_{1}}(f^{2}) + \mathrm{Ent}_{\nu_{\ell}}(f^{2})\bigr)\,,
\end{align*}
for some constant that depends on $\rho$ but not $N$, where the last inequality follows from \cite[Theorem A]{yaulogsobolev}.
Combining with \eqref{eq:entbound} this gives the desired lower bound on the log-Sobolev constant.
It only remains to show the asymptotic independence in Lemma \ref{lem:asympind}.

\begin{proof}[Proof of Lemma \ref{lem:asympind}]
The reversible measures, $\nu_{N,k}$, converge locally, as $N \to \infty$ and $k/N \to \rho$, to the infinite volume grand canonical measures, $\pi_\rho$, defined on $\{0,1\}^{\mathbb{Z}}$ (see Definition \ref{def:GC} below).
Although these can be expressed explicitly, they are not of product form and they were investigated in detail in \cite{blondel2020}.
In particular, it has been shown that they admit exponential decay of correlations, see Theorem \ref{thm:GCdecay} below.
This is sufficient to derive the necessary decay of correlations, and hence asymptotic independence required in Lemma \ref{lem:asympind}.

\begin{definition}[{\cite[Definition 6.2]{blondel2020}} Grand canonical measures]\label{def:GC}
    Fix $\rho \in(1/2,1)$ and $\ell \geq 1$, and let $\Lambda_\ell = \{1,2,\ldots,\ell\}$, then for $\sigma \in \{0,1\}^{\Lambda_\ell}$
    \begin{align}
        \pi(\eta_{\vert\Lambda_\ell} = \sigma) = \mathds{1}_{\hat{\mathcal{G}}_\ell}(\sigma)\rho^{\sigma_1+\sigma_\ell -n}(1-\rho)^{\ell -n}(2\rho -1)^{2n+1-\ell - \sigma_1-\sigma_\ell}\,,
    \end{align}
where $n(\sigma) = \sum_{x\in \Lambda_\ell}\sigma(x)$ and $$\hat{\mathcal{G}}_{\ell} = \{\sigma \in \{0,1\}^{\Lambda_L}\,:\,\forall\,(x,x+1)\in \Lambda_\ell,\ \sigma(x)+\sigma(x+1)\geq 1\}\,.$$
\end{definition}
In \cite[Chapter 6]{blondel2020} they showed a more  general version of the following decay of correlations under the grand canonical measures (see Corollary 6.6 in \cite{blondel2020}).
\begin{theorem}[Grand canonical decay of correlations]\label{thm:GCdecay}
    For $\sigma_1, \sigma_\ell \in \{0,1\}$ and any $\rho \in (\frac{1}{2},1)$, there exists a $C(\rho)>0$ such that
    \begin{align}\label{eq:GCdecay}
        \frac{\pi_\rho\big(\eta(1)=\sigma_1,\ \eta(\ell)=\sigma_\ell\big)}{\pi_\rho\big(\eta(1)=\sigma_1\big)\pi_\rho\big(\eta(\ell)=\sigma_\ell\big)} = 1 + O(e^{-C \ell})\,.
    \end{align}
\end{theorem}
To complete the proof of Lemma \ref{lem:asympind} we show that, for any fixed $\ell$, the finite dimensional marginals on $\ell$ sites under the measures $\nu_{N,k}$ converge to those of the grand canonical measure.
Lemma \ref{lem:asympind} follows immediately from the following claim together with \eqref{eq:GCdecay}.
\begin{claim}[Equivalence of ensembles]
\label{claim:equiv}
Fix $\ell\geq 2$. Let $\Lambda_\ell = \{1,2,\ldots,\ell\}$. For $\rho \in (1/2,1)$ and $\sigma \in \{0,1\}^{\Lambda_\ell}$,
\begin{align}
    \nu_{N,k}(\eta_{\vert\Lambda_\ell} = \sigma) \to \pi_\rho(\eta_{\vert\Lambda_\ell} = \sigma)\quad \textrm{as}\ N \to \infty\  \textrm{ and }\ k/N \to \rho \,.
\end{align}
\end{claim}
The proof of the claim is a straightforward calculation. 
For more details on the equivalence of ensembles in a more general setting, see \cite[Chapter 6]{blondel2020}. 
Fix $\sigma \in \{0,1\}^{\Lambda_\ell}$ and let $n = \sum_{x \in \Lambda_\ell}\sigma_\ell$.
Recall that $\nu_{N,k}$ is uniform on $\cG_{N,k}$. 
By counting the number of configurations on $\{\ell+1,\ldots,N\}\subset \mathbb{T}_N$ that are compatible with $\sigma_1$ and $\sigma_\ell$, and recalling that $|\mathcal{E}_{N,k}|=\binom{k-1}{N-k}$, we find
\begin{align*}
    \nu_{N,k}(\eta_{\vert\Lambda_\ell} = \sigma) &= \binom{k -n -1+\sigma_1+\sigma_\ell}{N-\ell-k+n}/|\cG_{N,k}| \\[1ex]
    &= \frac{k}{N}\binom{k -(n + 1-\sigma_1-\sigma_\ell)}{N - k -(\ell-n)} / \binom{k}{N-k}\\[1ex]
   &\to \rho^{\sigma_1+\sigma_2-n} (1-\rho)^{\ell-n}(2\rho -1)^{2n-\ell+1-\sigma_1-\sigma_\ell} = \pi_\rho(\eta_{\vert\Lambda_\ell}=\sigma)\,,
\end{align*}
as $N\to \infty$ and $k/N \to \rho \in(1/2,1)$.
\end{proof}

\appendix

%%%%%%%%%%%%%%%%%%%%%%%%%%%%%%%%%%%%%%%%%%%%%%

\section{Hitting time of the ergodic component} \label{sec:hitergodic}

%%%%%%%%%%%%%%%%%%%%%%%%%%%%%%%%%%%%%%%%%%%%%%

In this section we prove Proposition \ref{prop:hiterg}. 
To prove the proposition we consider the ZRP on the circle $\mathbb{T}_{n}$ with the generator in \eqref{eq:genzc}.
We let $Q_{t}^{\omega}$ denote the law of the ZRP at time $t$ with initial condition $\omega$. 

\begin{definition} \label{def:ergodicregionzero}
For a zero--range configuration $\omega\in\mathbb{N}_{0}^{\mathbb{T}_{N-k}}$ we say that the (clockwise) interval $[i,j]\subsetneq\mathbb{T}_{N-k}$ with $\lvert i-j\rvert<N-k$ is an \textit{ergodic region} for the configuration $\omega$ if $\omega(z)\geq1$ for all $z\in[i,j]$, and $\omega(i-1)=\omega(j+1)=0$. 
\end{definition}
We note that every ergodic region for an FEP configuration $\xi\in\mathcal{G}_{N,k}^{c}$, by Definition \ref{def:regions}, corresponds to an ergodic region from Definition \ref{def:ergodicregionzero} in the zero--range configuration $\Pi^\circ(\xi)$ as defined in Section \ref{sec:coupleZRP}. 
We define the ergodic component for the ZRP by
\begin{equation} \label{eq:ergodiczc}
G_{n}=\Bigl\{\omega\in\mathbb{N}_{0}^{\mathbb{T}_{n}}:\omega(i)\geq1\mbox{ for all }i\in\mathbb{T}_{n}\Bigr\}. 
\end{equation}
By Lemma \ref{lem:ZRPhit} it holds that 
\begin{equation} \label{eq:lawequiv}
P_{t}^{\xi}\bigl(\cG_{N,k}^{c}\bigr)=Q_{t}^{\Pi^\circ(\xi)}\bigl(G_{N-k}^{c}\bigr). 
\end{equation}

A key feature of the ZRP with (weakly) increasing rates is the following monotonicity property; see, for example, \cite{kipnislandim}[Section 2.5]. 
Let $(\omega_{t})_{t\geq0}$ and $(\omega'_{t})_{t\geq0}$ be two trajectories driven by the generator $\mathcal{L}^{\circ}_{\mathrm{ZR}}$, with initial conditions $\omega$ and $\omega'$ respectively. 
We say that $\omega\leq\omega'$ if $\omega(i)\leq\omega'(i)$ for all $i\in\mathbb{T}_{n}$. 
If $\omega\leq\omega'$ then there exists a coupling such that $\omega_{t}\leq\omega'_{t}$ for all $t\geq0$. 
An event $E$ is increasing if $\omega\in E$ implies $\omega'\in E$. 
If $E$ is increasing then the monotonicity implies that 
\begin{equation*} \label{eq:zrmon}
Q_{t}^{\omega}(E)\leq Q_{t}^{\omega'}(E) \quad \textrm{for} \ \ \omega \leq \omega'\,.
\end{equation*}
The ergodic component $G_{n}$ is an increasing event, therefore 
\begin{equation} \label{eq:increasing2}
\textrm{if }\ \omega\leq\omega' \ \textrm{ then}\quad  Q_{t}^{\omega}(G_{n}^{c})\geq Q_{t}^{\omega'}(G_{n}^{c}).
\end{equation}

\begin{proof}[Proof of Proposition \ref{prop:hiterg}]
Let $\xi\in\mathcal{I}_{N,k}^{m}$ be an arbitrary configuration containing at least $N-k$ particles in at most $m$ ergodic regions.
In particular, at most $m$ ergodic regions in the zero-range configuration $\Pi^\circ(\xi)$ collectively contain at least $N-k$ particles.
Deleting all particles outside of these $m$ regions there is still enough particles to reach the ergodic component $G_{N-k}$ (in which there is at least one particle at each site).
Assume, WLOG,  $\Pi^\circ(\xi)$ contains exactly $m$ such ergodic regions $I_{1},\ldots,I_{m}$ who's union contains at least $N-k$ particles, and let 
$$
\bar{\omega}(i)=
\begin{cases}
\Pi^\circ(\xi)(i)&\mbox{ if }i\in I_{1}\cup\cdots\cup I_{m},\\
0&\mbox{ otherwise},
\end{cases}
$$
i.e.\ the configuration in which all particles outside of the $m$ ergodic regions are deleted. 
By definition $\bar{\omega}\leq\Pi^\circ(\xi)$, so it follows from \eqref{eq:lawequiv} and \eqref{eq:increasing2} that 
$$
P_{t}^{\xi}\bigl(\cG_{N,k}^{c})=Q_{t}^{\Pi^\circ(\xi)}\bigl(G_{N-k}^{c}\bigr)\leq Q_{t}^{\bar{\omega}}\bigl(G_{N-k}^{c}\bigr).
$$

We now define a collection of stopping times for the zero--range process $(\bar{\omega}_{t})_{t\geq0}$ with law $Q_{t}^{\bar{\omega}}$. 
Let $\tau_{0}=0$ and suppose that $\tau_{s-1}$ is given for some $s\geq1$ with $\bar{\omega}_{\tau_{s-1}}\notin G_{N-k}$, then $\tau_{s}$ is the first time $t>\tau_{s-1}$ that one of the following three conditions is met:
\begin{itemize}
    \item[(1)] The number of ergodic regions in $\bar{\omega}_{t}$ is one less than the number of ergodic regions in $\bar{\omega}_{\tau_{s-1}}$, i.e.\ two ergodic regions evolve and meet to form a single ergodic region;
    \item[(2)] The number of `frozen' ergodic regions in $\bar{\omega}_{t}$ for which all particles are stuck is one greater than the number of `frozen' ergodic regions in $\bar{\omega}_{\tau_{s-1}}$, i.e.\ an ergodic region evolves until all of its particles become stuck;
    \item[(3)] $\bar{\omega}_{t}\in G_{N-k}$. 
\end{itemize}  
The number of such stopping times is at most $2m-1$ since (1) and (2) can occur at most $m-1$ times, and (3) occurs exactly once.
In particular, 
$$
Q_{t}^{\Pi(\xi)}(G_{N-k}^{c})\leq Q_{t}^{\bar{\omega}}(G_{N-k}^{c})= Q_{t}^{\bar{\omega}}\Biggl(\sum_{s\geq1}(\tau_{s}-\tau_{s-1})>t\Biggr).
$$
We define the set of frozen zero--range configurations on $\mathbb{T}_{n}$ to be
$$
\mathcal{F}_{n}=\bigl\{\omega\in\mathbb{N}_{0}^{\mathbb{T}_{n}}:\omega(i)\leq 1\mbox{ for all }i\in\mathbb{T}_{n}\bigr\},
$$
Let $\hat{\omega}\leq\bar{\omega}$ be the configuration obtained by deleting all but one ergodic region $I_{1}$, i.e.
$$
\hat{\omega}(i)=\bar{\omega}(i)\1_{\{i \in I_1\}}\,,
$$
and assume, wlog, that  $\hat{\omega}\notin\mathcal{F}_{N-k}$ (since at least one ergodic block must not be frozen). 
Let $(\hat{\omega}_{t})_{t\geq0}$ be the ZRP generated by $\mathcal{L}_\mathrm{ZR}^{N-k}$ with initial condition $\hat{\omega}$.

By monotonicity, the hitting time of the event $G_{N-k}\cup\mathcal{F}_{N-k}$ in the process $(\hat{\omega}_{t})_{t\geq0}$ is stochastically larger than $\tau_{1}$; in the dynamics of $(\bar{\omega}_{t})_{t\geq0}$ the ergodic region $I_{1}$ either evolves to become completely frozen or the stopping time $\tau_{1}$ occurs at some earlier time.  
In particular, it holds that
$$
Q_{t}^{\bar{\omega}}(\tau_{1}>t)\leq Q_{t}^{\hat{\omega}}(G_{N-k}^{c}\cap\mathcal{F}_{N-k}^{c}). 
$$
We use the following claim: 
\begin{claim} \label{claim:zero_obep}
Let $E_{n,\ell}$ denote the set of zero--range configurations on $\mathbb{T}_{n}$ containing a single ergodic region containing $\ell$ particles. 
For $\epsilon>0$ there exists a constant $C>0$ which does not depend on $\epsilon$ such that
$$
\max_{\omega\in E_{n,\ell}}Q_{Cn^{2}\log n}^{\omega}\bigl(G_{n}^{c}\cap F_{n}^{c})\leq\epsilon,
$$
for all $n$ sufficiently large, uniformly in $\ell$.
\end{claim}
By Claim \ref{claim:zero_obep} it follows that 
$
Q_{CN^{2}\log N}^{\hat{\omega}}(G_{N-k}^{c}\cap\mathcal{F}_{N-k}^{c})\leq\epsilon,
$
for all $N$ sufficiently large. 
By application of the Strong Markov property, and repeating the above argument for each time interval $[\tau_{s-1},\tau_{s}]$, it holds that
$$
P_{C(2m-1)N^{2}\log N}^{\xi}(\mathcal{G}_{N,k}^{c})\leq(2m-1)\epsilon,
$$
for $N$ sufficiently large.
Since $\xi\in\mathcal{I}_{N,k}^{m}$ was arbitrary this completes the proof. 
\end{proof} 
Claim \ref{claim:zero_obep} will follow from a comparison to the $(1/2,1/2,0,0,1/2)$--OBEP. 
\begin{proof}[Proof of Claim \ref{claim:zero_obep}]
We define a map $\Phi:E_{n,\ell}\rightarrow\Omega_{\ell-1}$ from the zero-range configurations to the OBEP configurations on a segment of length $\ell-1$. 
If the ergodic region in $\omega\in E_{n,\ell}$ consists of a single site we set $\Phi(\omega)=\boldsymbol{0}$, the empty segment. 
If the ergodic region is given by $[i,j]$ we label zero--range stacks (clockwise) from left to right by $r=1,\ldots,j-i+1$. 
Each of the stacks corresponds to a particle in the OBEP, except the rightmost stack which corresponds to the reservoir on the right boundary. 
Labelling particles from left to right in the OBEP the position of the $r$-th particle is given by
$
\sum_{z=i}^{r}\omega(z). 
$
In particular the number of holes between the $r$-th and $(r+1)$-th particles in the OBEP is $\omega(r)-1$. 
See Figure \ref{fig:zeropens}.
\begin{figure}[tb]
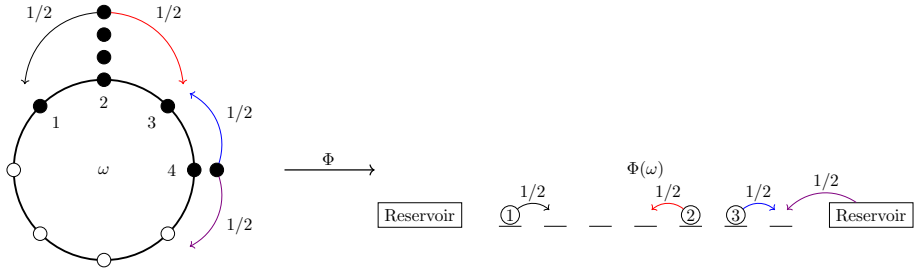

\begin{center}
\scalebox{0.6}{\zeropens}
\end{center}
\caption{The mapping $\Phi:E_{8,8}\rightarrow\Omega_{7}$ for a configuration $\omega\in E_{8,8}$. 
All possible transitions of the ZRP and OBEP are indicated. 
Matching colours indicate corresponding transitions in the coupling, and the stacks in $\omega$ are labelled $1,2,3$ and $4$.}
\label{fig:zeropens}
\end{figure}
To couple these processes, each time a particle in the ZRP jumps left from the leftmost stack to an empty site, we input a particle at the left boundary in the OBEP. 
The zero-range stacks and OBEP particles are then relabelled.  
Each time a particle in the ZRP jumps right from the rightmost stack to an empty site we input a particle at the right boundary in the OBEP. 
If a particle in the ZRP jumps left (\textit{resp.} right) from stack $r$ to stack $r-1$ (\textit{resp.}  $r+1$), the $(r-1)$-th (\textit{resp.}  $r$-th) particle in the OBEP jumps one step right (\textit{resp.} left). 

The mapping $\Phi$ is not bijective; every rotation of a configuration in $E_{n,\ell}$ results in the same OBEP configuration. 
However, 
$$
Q_{t}^{\omega}\bigl(G_{n}^{c}\cap\mathcal{F}_{n}^{c}) = P_{t}^{\Phi(\omega)}\Bigl(\zeta\in\Omega_{\ell-1}:\sum_{x}\zeta(x)<\min(n-1,\ell-1)\Bigr)\leq P_{t}^{\Phi(\omega)}\bigl(\{\zeta \equiv 1 \}^{c}\bigr).
$$
The equilibrium distribution of the OBEP is concentrated on the full configuration $\{\zeta \equiv 1 \}=\{\zeta\,:\, \forall x \in [\ell-1],\ \zeta(x)= 1 \}$ and the claim follows from \cite[Theorem 1.1]{gantert2020mixing}.
\end{proof}
\bibliographystyle{plain} 
\bibliography{ref_paper}

@article{Alcarazexact,
  title = {Exact solution of the asymmetric exclusion model with particles of arbitrary size},
  author = {Alcaraz, F. and Bariev, R.},
  journal = {Physical Review E},
  volume = {60},
  pages = {79-88},
  year = {1999},
  publisher = {American Physical Society},
  doi = {10.1103/PhysRevE.60.79},
  url = {https://link.aps.org/doi/10.1103/PhysRevE.60.79}
}

@article{aldousbrown,
 ISSN = {07492170},
 URL = {http://www.jstor.org/stable/4355727},
 author = {Aldous, D. and Brown, M.},
 journal = {Lecture Notes-Monograph Series},
 pages = {1-16},
 publisher = {Institute of Mathematical Statistics},
 title = {Inequalities for Rare Events in Time-Reversible {M}arkov Chains {I}},
 volume = {22},
 year = {1992}
}

@article{ayyer2020stationary,
author = {Ayyer, A. and Goldstein, S. and Lebowitz, J. and Speer, E. },
title = {{Stationary states of the one-dimensional facilitated asymmetric exclusion process}},
volume = {59},
pages = {726--742},
journal = {Annales de l'Institut Henri Poincar{\'e}, Probabilit{\'e}s et Statistiques},
doi = {10.1214/22-AIHP1264},
URL = {https://doi.org/10.1214/22-AIHP1264},
year = {2023}, 
publisher = {Institut Henri Poincar{\'e}}
}

@InProceedings{Jinho2018,
   title={Facilitated Exclusion Process},
   ISBN={9783030015930},
   ISSN={2197-8549},
   url={http://dx.doi.org/10.1007/978-3-030-01593-0_1},
   DOI={10.1007/978-3-030-01593-0_1},
   booktitle={Computation and Combinatorics in Dynamics, Stochastics and Control},
   publisher={Springer International Publishing},
   author={Baik, J. and Barraquand, G. and Corwin, I. and Suidan, T.},
   year={2018},
   pages = {1–35}, 
   volume  = {13}
}

@article{barraquand2023weakly,
      title={Weakly asymmetric facilitated exclusion process}, 
      author={Barraquand, G. and Blondel, O. and Simon, M.},
volume = {30},
journal = {Electronic Journal of Probability},
publisher = {Institute of Mathematical Statistics and Bernoulli Society},
pages = {1 -- 51},
year = {2025}
}

@article{PhysRevE.79.041143,
  title = {Active--absorbing-state phase transition beyond directed percolation: A class of exactly solvable models},
  author = {Basu, U. and Mohanty, P. },
  journal = {Physical Review E},
  year = {2009},
  publisher = {American Physical Society},
  doi = {10.1103/PhysRevE.79.041143},
  url = {https://link.aps.org/doi/10.1103/PhysRevE.79.041143},
  volume = {79},
  pages = {041143}
}

@article{blondel2020,
author = {Blondel, O. and Erignoux, C. and Sasada, M. and Simon, M.},
title = {Hydrodynamic limit for a facilitated exclusion process},
journal = {Annales de l'Institut Henri Poincar{\'e}, Probabilit{\'e}s et Statistiques},
publisher = {Institut Henri Poincar{\'e}},
doi = {10.1214/19-AIHP977},
url = {https://doi.org/10.1214/19-AIHP977},
year={2020},
volume = {56},
pages = {667-714}
}

@article{Blondel_2021,
   author={Blondel, O. and Erignoux, C. and Simon, M.},
   title={Stefan problem for a nonergodic facilitated exclusion process},
   journal={Probability and Mathematical Physics},
   year={2021},
   publisher={Mathematical Sciences Publishers},
   DOI={10.2140/pmp.2021.2.127},
   url={http://dx.doi.org/10.2140/pmp.2021.2.127},
   ISSN={2690-0998},
   volume={2},
   pages={127–178}
}

@article{cesientropy,
author = {Cesi, F.},
year = {2001},
title = {Quasi--factorization of the entropy and logarithmic {S}obolev inequalities for {G}ibbs random fields},
volume = {120},
pages = {569–584},
journal = {Probability Theory and Related Fields},
doi = {10.1007/PL00008792}, 
url = {https://doi.org/10.1007/PL00008792}
}

@article{DSClogsob,
 ISSN = {10505164},
 author = {P. Diaconis and L. Saloff--Coste},
 journal = {The Annals of Applied Probability},
 pages = {695-750},
 publisher = {Institute of Mathematical Statistics},
 title = {Logarithmic {S}obolev Inequalities for Finite {M}arkov Chains},
 volume = {6},
 year = {1996}, 
 URL = {http://www.jstor.org/stable/2245210}
}

@article{Alcarazanom,
  title = {Anomalous tag diffusion in the asymmetric exclusion model with particles of arbitrary sizes},
  author = {Ferreira, A. and Alcaraz, F.},
  journal = {Physical Review E},
  volume = {65},
  year = {2002},
  publisher = {American Physical Society},
  doi = {10.1103/PhysRevE.65.052102},
  url = {https://link.aps.org/doi/10.1103/PhysRevE.65.052102},
  pages = {052102}
}

@article{PhysRevLett.105.210603,
  author = {Gabel, A. and Krapivsky, P. and Redner, S.},
  title = {Facilitated Asymmetric Exclusion},
  journal = {Physical Review Letters},
  year = {2010},
  publisher = {American Physical Society},
  doi = {10.1103/PhysRevLett.105.210603},
  url = {https://link.aps.org/doi/10.1103/PhysRevLett.105.210603},
  volume = {105},
  pages = {210603}
}

@article{gantert2020mixing,
      author={Gantert, N. and Nestoridi, E. and Schmid, D.},
      title={Mixing times for the simple exclusion process with open boundaries},
      journal = {The Annals of Applied Probability},
      volume = {33},
      pages = {1172 -- 1212},
      year = {2023},
      doi = {10.1214/22-AAP1839},
      URL = {https://doi.org/10.1214/22-AAP1839},
}

@article{GLSexact,
    author = {{Goldstein}, S. and {Lebowitz}, J. and {Speer}, E.},
    title = {Exact solution of the facilitated totally asymmetric simple exclusion process},
    journal = {Journal of Statistical Mechanics: Theory and Experiment},
    year = {2019},
    volume = {2019},
    pages = {123202},
    doi = {10.1088/1742-5468/ab363f},
    url = {https://dx.doi.org/10.1088/1742-5468/ab363f}
}

@article{GLSdiscrete,
  author = {Goldstein, S. and Lebowitz, J. and Speer, E.},
  title = {The Discrete-Time Facilitated Totally Asymmetric Simple Exclusion Process},
  journal = {Pure and Applied Functional Analysis},
  volume = {6},
  pages = {177-203},
  year = {2021}
}

@article{glsstat,
author = {Goldstein, S. and Lebowitz, J. and Speer, E. },
title = {{Stationary states of the one-dimensional discrete-time facilitated symmetric exclusion process}},
journal = {Journal of Mathematical Physics},
volume = {63},
pages = {083301},
year = {2022},
issn = {0022-2488},
doi = {10.1063/5.0085528},
url = {https://doi.org/10.1063/5.0085528},
eprint = {https://pubs.aip.org/aip/jmp/article-pdf/doi/10.1063/5.0085528/16563099/083301\_1\_online.pdf}
}

@article{GoncalvesReservoir,
  author = {Gon\c{c}alves, P. and Jara, M. and Marinho, R. and Menezes, O.},
  title = {Sharp Convergence to Equilibrium for the {SSEP} with Reservoirs},
volume = {62},
journal = {Annales de l'Institut Henri Poincaré, Probabilités et Statistiques},
number = {1},
publisher = {Institut Henri Poincaré},
pages = {146 -- 165},
year = {2026}
}

@book{kipnislandim,
  author = {Kipnis, C. and Landim, C.},
  title = {Scaling limits of interacting particle systems},
  publisher = {Springer Berlin, Heidelberg},
  year = {1999},
  doi = {https://doi.org/10.1007/978-3-662-03752-2},
  edition = {first}
}

@article{cutoffasep,
author = {Labb\'{e}, C. and Lacoin, H.},
title = {Cutoff phenomenon for the asymmetric simple exclusion process and the biased card shuffling},
journal = {The Annals of Probability},
year = {2019},
doi = {10.1214/18-AOP1290},
publisher = {Institute of Mathematical Statistics},
pages = {1541--1586},
ISSN = {00911798, 2168894X},
URL = {https://doi.org/10.1214/18-AOP1290},
volume = {47},
urldate = {2023-01-12}
}

@article{lacoindiff,
author = {Lacoin, H.},
year = {2017},
title = {The Simple Exclusion Process on the Circle has a diffusive Cutoff Window},
journal = {Annales de l'Institut Henri Poincar{\'e}, Probabilit{\'e}s et Statistiques},
doi = {10.1214/16-AIHP759},
volume = {53},
pages = {1402-1437},
URL = {https://doi.org/10.1214/16-AIHP759}, 
publisher = {Institut Henri Poincar{\'e}}
}

@article{lacoinsegment,
author = {Lacoin, H.},
title = {Mixing time and cutoff for the adjacent transposition shuffle and the simple exclusion},
journal = {The Annals of Probability},
publisher = {Institute of Mathematical Statistics},
year = {2016},
doi = {10.1214/15-AOP1004},
URL = {https://doi.org/10.1214/15-AOP1004},
volume = {44},
pages = {1426-1487}
}

@article{Lakatos_2003,
doi = {10.1088/0305-4470/36/8/302},
url = {https://dx.doi.org/10.1088/0305-4470/36/8/302},
year = {2003},
volume = {36},
pages = {2027},
author = {Lakatos, G. and  Chou, T.},
title = {Totally asymmetric exclusion processes with particles of arbitrary size},
journal = {Journal of Physics A: Mathematical and General}
}

@book{LevinPeresWilmer2006,
  author = {Levin, D. and Peres, Y.},
  title = {{M}arkov chains and mixing times},
  publisher = {American Mathematical Society},
  year = 2017, 
  edition = {second}
}

@article{macdonald69,
author = {MacDonald, C. and Gibbs, J.},
title = {Concerning the kinetics of polypeptide synthesis on polyribosomes},
journal = {Biopolymers},
volume = {7},
pages = {707-725},
doi = {https://doi.org/10.1002/bip.1969.360070508},
url = {https://onlinelibrary.wiley.com/doi/abs/10.1002/bip.1969.360070508},
eprint = {https://onlinelibrary.wiley.com/doi/pdf/10.1002/bip.1969.360070508},
year = {1969}
}

@article{Morris_2006,
	doi = {10.1214/105051605000000728},
	year = 2006,
	publisher = {Institute of Mathematical Statistics},
	volume = {16},
    author = {Morris, B.},
	title = {The mixing time for simple exclusion},
	journal = {The Annals of Applied Probability},
    pages= {615-635}
}

@article{uclassfield,
  title = {Universality Class of Absorbing Phase Transitions with a Conserved Field},
  author = {Rossi, M. and Pastor-Satorras, R. and Vespignani, A.},
  journal = {Physical Review Letters},
  volume = {85},
  pages = {1803--1806},
  year = {2000},
  publisher = {American Physical Society},
  doi = {10.1103/PhysRevLett.85.1803},
  url = {https://link.aps.org/doi/10.1103/PhysRevLett.85.1803}
}

@article{salez2020sharp,
  TITLE = {{A sharp log-Sobolev inequality for the multislice}},
  AUTHOR = {Salez, J.},
  URL = {https://hal.archives-ouvertes.fr/hal-03353026},
  JOURNAL = {{Annales Henri Lebesgue}},
  PUBLISHER = {{UFR de Math{\'e}matiques - IRMAR}},
  VOLUME = {4},
  PAGES = {1143-1161},
  YEAR = {2021},
  DOI = {10.5802/ahl.99},
}

@article{salez2023,
author = {J. Salez},
title = {{Universality of cutoff for exclusion with reservoirs}},
volume = {51},
journal = {The Annals of Probability},
publisher = {Institute of Mathematical Statistics},
pages = {478 -- 494},
year = {2023},
doi = {10.1214/22-AOP1600},
URL = {https://doi.org/10.1214/22-AOP1600}
}

@incollection{Saloff-Coste1997,
author={Saloff--Coste, L.},
title={Lectures on finite {M}arkov chains},
booktitle={Lectures on Probability Theory and Statistics: Ecole d'Et{\'e} de Probabilit{\'e}s de Saint-Flour XXVI-1996},
pages={301-413},
year={1997},
publisher={Springer Berlin Heidelberg}
}

@article{ESZ,
author = {Erignoux, C. and Simon, M. and Zhao, L.},
title = {{Mapping hydrodynamics for the facilitated exclusion and zero-range processes}},
volume = {34},
journal = {The Annals of Applied Probability},
number = {1B},
publisher = {Institute of Mathematical Statistics},
pages = {1524 -- 1570},
year = {2024},
doi = {10.1214/23-AAP1997},
URL = {https://doi.org/10.1214/23-AAP1997}
}

@article{zhaochen,
title = {The invariant measures and the limiting behaviors of the facilitated TASEP},
journal = {Statistics \& Probability Letters},
volume = {154},
pages = {108557},
year = {2019},
issn = {0167-7152},
doi = {https://doi.org/10.1016/j.spl.2019.108557},
url = {https://www.sciencedirect.com/science/article/pii/S0167715219302032},
author = {L. Zhao and D. Chen}
}

@article{EZ23,
      title={Stationary fluctuations for the facilitated exclusion process}, 
      author={Erignoux, C. and Zhao, L.},
volume = {29},
journal = {Electronic Journal of Probability},
publisher = {Institute of Mathematical Statistics and Bernoulli Society},
pages = {1 -- 41},
year = {2024}
}

@article{dacunha2024hydrodynamic,
      title={Hydrodynamic limit for a boundary-driven facilitated exclusion process}, 
      author={Da Cunha, H. and Erignoux, C. and Simon, M.},
      year={2024},
      journal={arXiv preprint arxiv:2401.16535},
      primaryClass={math.PR}
}

@article{PhysRevE.64.016123,
  title = {Scaling behavior of the absorbing phase transition in a conserved lattice gas around the upper critical dimension},
  author = {L\"ubeck, S.},
  journal = {Phys. Rev. E},
  volume = {64},
  issue = {1},
  pages = {016123},
  numpages = {6},
  year = {2001},
  month = {Jun},
  publisher = {American Physical Society},
  doi = {10.1103/PhysRevE.64.016123},
  url = {https://link.aps.org/doi/10.1103/PhysRevE.64.016123}
}

@article{PhysRevE.71.016112,
  title = {Conserved lattice gas model with infinitely many absorbing states in one dimension},
  author = {de Oliveira, M.},
  journal = {Phys. Rev. E},
  volume = {71},
  issue = {1},
  pages = {016112},
  numpages = {4},
  year = {2005},
  month = {Jan},
  publisher = {American Physical Society},
  doi = {10.1103/PhysRevE.71.016112},
  url = {https://link.aps.org/doi/10.1103/PhysRevE.71.016112}
}

@article{Shaw2003,
  title = {Totally asymmetric exclusion process with extended objects: A model for protein synthesis},
  author = {Shaw, L. and Zia, R. and Lee, K.},
  journal = {Physical Review E},
  volume = {68},
  year = {2003},
  publisher = {American Physical Society},
  doi = {10.1103/PhysRevE.68.021910},
  url = {https://link.aps.org/doi/10.1103/PhysRevE.68.021910},
  pages = {021910}
}

@article{SPITZER1970246,
title = {{Interaction of Markov processes}},
journal = {Advances in Mathematics},
volume = {5},
pages = {246-290},
year = {1970},
issn = {0001-8708},
doi = {https://doi.org/10.1016/0001-8708(70)90034-4},
url = {https://www.sciencedirect.com/science/article/pii/0001870870900344},
author = {Spitzer, F.}
}

@article{tran2023,
author = {H-Q. Tran},
title = {{Cutoff for the non reversible SSEP with reservoirs}},
volume = {28},
journal = {Electronic Journal of Probability},
publisher = {Institute of Mathematical Statistics and Bernoulli Society},
pages = {1--24},
year = {2023},
doi = {10.1214/23-EJP1044},
URL = {https://doi.org/10.1214/23-EJP1044}
}

@article{Wilson_2004,
	year = 2004,
	publisher = {Institute of Mathematical Statistics},
	volume = {14},
	author = {Wilson, D.},
	title = {Mixing times of lozenge tiling and card shuffling Markov chains},
	journal = {The Annals of Applied Probability},
    pages = {274-325}
}

@article{yaulogsobolev,
title = {Logarithmic {S}obolev inequality for generalized
simple exclusion processes},
author = {Yau, H-T.},
journal = {Probability and related fields},
year = {1997},
volume = {109},
pages = {507–538} 
}

@article{Brune24,
      title={Cutoff for the mixing time of the Facilitated Exclusion Process}, 
      author={B. Massouli\'{e}},
      journal = {arXiv e-prints arXiv:2412.04032},
      year={2024},
      eprint={2412.04032},
      archivePrefix={arXiv},
      primaryClass={math.PR},
      url={https://arxiv.org/abs/2412.04032}, 
}

@article{ClementBrune24,
      title={Cutoff for the transience and mixing time of a {SSEP}s and consequences on the {FEP}}, 
      author={C. Erignoux and B. Massouli\'{e}},
      journal = {arXiv e-prints arXiv:2403.20010},
      year={2024},
      eprint={2403.20010},
      archivePrefix={arXiv},
      primaryClass={math.PR},
      url={https://arxiv.org/abs/2403.20010}, 
}
\end{document}